\documentclass{amsart}
\usepackage{bbm}
\usepackage{color}
\usepackage{amscd}
\usepackage{nicefrac}
\usepackage{graphicx}
\usepackage{hyperref}
\usepackage{amssymb}
\usepackage{euscript}
\usepackage{enumitem}
\usepackage[cmtip,all]{xy}
\usepackage[export]{adjustbox}

\numberwithin{equation}{section}
\renewcommand{\subsection}[1]{\refstepcounter{subsection}{\bf (\arabic{section}\alph{subsection}) #1.}}

\allowdisplaybreaks

\theoremstyle{plain}
\newtheorem{thm}{Theorem}[section]

\newtheorem{theorem}[thm]{Theorem}

\newtheorem{cor}[thm]{Corollary}
\newtheorem{corollary}[thm]{Corollary}

\newtheorem{lemma}[thm]{Lemma}

\newtheorem{prop}[thm]{Proposition}
\newtheorem{definition}[thm]{Definition}

\newtheorem{remark}[thm]{Remark}
\newtheorem{convention}[thm]{Convention}

\newtheorem{example}[thm]{Example}
\newtheorem{non-example}[thm]{Non-example}

\newtheorem{lem}[thm]{Lemma}
\newtheorem{scholium}[thm]{Scholium}

\newtheorem{conjecture}[thm]{Conjecture}

\newtheorem*{claim*}{Claim} 
\newtheorem*{lemma*}{Lemma}
\newtheorem*{theorem*}{Theorem}
\newtheorem*{conjecture*}{Conjecture}
\newtheorem{remarks}[thm]{Remarks}

\newcommand{\bC}{{\mathbb C}}

\newcommand{\bF}{{\mathbb F}}

\newcommand{\bQ}{{\mathbb Q}}
\newcommand{\bR}{{\mathbb R}}

\newcommand{\bZ}{{\mathbb Z}}

\newcommand{\scrO}{\EuScript O}

\newcommand{\bfe}{\mathbf{e}}

\newcommand{\bfm}{\mathbf{m}}
\newcommand{\bfn}{\mathbf{n}}

\newcommand{\iso}{\cong}

\newcommand{\mymod}{\operatorname{mod}}

\newcommand{\val}{\mathrm{val}}

\begin{document}
\title[$P$-adic Gamma classes]{$P$-adic Gamma classes and overconvergent \\ Frobenius structures for quantum connections}
\author{Shaoyun Bai, Dan Pomerleano, Paul Seidel}

\begin{abstract}
Consider the small quantum connection on a monotone symplectic manifold, with $p$-adic coefficients. We conjecture that this always admits an overconvergent Frobenius structure, whose constant term is given by a characteristic class associated to Morita's $p$-adic Gamma function. We prove this conjecture for toric Fano varieties and Grassmannians, and also supply additional experimental evidence.
\end{abstract}

\maketitle

\section{Introduction}
Let $M$ be a closed monotone ($c_1(TM) = [\omega_M]$) symplectic manifold; the main source of such manifolds are smooth complex projective Fano varieties. Take the (single-variable small) quantum product on $H^*(M;\bZ)[q]$,
\begin{equation} \label{eq:quantum-product}
x \ast_q y = x \smile y + q (x \ast^{(1)} y) + q^2 (x \ast^{(2)} y) + \cdots
\end{equation}
Here, $\ast^{(k)}$ counts rational curves in $M$ with first Chern number $k$, hence is an operation of degree $-2k$. The corresponding version of the quantum connection is
\begin{equation} \label{eq:quantum-connection}
\begin{aligned}
& \nabla_{q\partial_q}: H^*(M;\bZ)[q] \longrightarrow H^*(M;\bZ)[q], \\
& \nabla_{q\partial_q} x = q\partial_q x + c_1(TM) \ast_q x = q\partial_q x + c_1(TM) \smile x + 
c_1(TM) \ast^{(1)} x + \cdots
\end{aligned}
\end{equation}
We will be interested in arithmetic aspects of this connection, specifically making contact with the theory of $p$-adic differential equations. This means switching from integer coefficients to ones in a $p$-adic field.

Let's give a minimal amount of background. Take an odd prime $p$. We use the field
\begin{equation} \label{eq:cyclotomic}
K = \bQ_p(\mu), 
\end{equation}
where $\mu$ is a primitive $p$-th root of unity. The $p$-adic valuation on $K$, normalized so that $\mathrm{val}(p) = 1$, takes values in $\frac{1}{p-1} \bZ$. The $p$-adic norm is 
\begin{equation}
|x| = p^{-\mathrm{val}(x)}. 
\end{equation}
Our field $K$ contains a unique element $\pi$ with \cite[Proposition 4.4.36]{cohen-vol1}
\begin{equation} \label{eq:pi}
\begin{aligned}
& \pi^{p-1} = -p, \\
& |(1+\pi)-\mu| \leq p^{-2/(p-1)}. 
\end{aligned}
\end{equation}
Let $O \subset K$ be the ring of integers. The maximal ideal is $\pi O \subset O$, and the residue field is $O/\pi O = \bF_p$. 

Suppose we have a connection (a system of first order linear differential equations) in one variable $t$, defined over $K$. In informal terms, a Frobenius structure is an isomorphism between the original equation and its pullback under the coordinate change $F(t) = t^p$. Even if the original equation is algebraic in $t$, the Frobenius structure generally won't be, hence it makes sense to ask about the region in which it is defined. We will be focused on the behaviour near $t = 0$, so the main question is the radius of convergence around that point. Number theory and algebraic geometry both underline the importance of Frobenius structures which are overconvergent, meaning with convergence radius $>1$ (see Section \ref{sec:diff-eq} for rigorous definitions). 

\begin{example} \label{th:exponential}
Consider the equation 
\begin{equation}
df/dt = cf, 
\end{equation}
for some constant $c$. The Frobenius pullback is $df^{(p)}/dt = c\, p t^{p-1} f^{(p)}$. The relation between the two equations is that
\begin{equation} \label{eq:example-phi}
f(t) = \Phi(t) f^{(p)}(t), \;\; \Phi(t) = e^{c(t-t^p)}, 
\end{equation}
so $\Phi(t)$ is a Frobenius structure (and the unique one, up to multiplying by a nonzero constant). In the simplest case $c = 1$, the function $\Phi(t)$, defined by the usual power series, has $p$-adic convergence radius $<1$; so that equation doesn't admit an overconvergent Frobenius structure. However, when $c = \pi$ or an integer multiple of that, $\Phi(t)$ is overconvergent (see Section \ref{sec:p-adic-functions}); which is why the constant $\pi$ is important in this theory.
\end{example}

Returning to the quantum connection, now with $K$-coefficients, we rescale the variable to 
\begin{equation} \label{eq:t-q}
t = q/\pi
\end{equation}
for $\pi$ as in \eqref{eq:pi}. For simplicity, we also restrict to the even degree part of cohomology (one could carry out a parallel discussion for the odd degree part). The outcome is
\begin{equation} \label{eq:quantum-connection-2}
\begin{aligned} 
& \nabla_{t\partial_t}: H^{\mathrm{even}}(M;K)[t] \longrightarrow H^{\mathrm{even}}(M;K)[t], \\
& \nabla_{t\partial_t} x = t\partial_t x + c_1(M) \ast_{\pi t} x = t\partial_t x + c_1(M) \smile x + \pi t c_1(M) \ast^{(1)} x + \cdots
\end{aligned}
\end{equation}
For the following elementary facts, see Section \ref{sec:gamma-class}. Given any class $b \in H^{\mathrm{even}}(M;K)$ which is invertible for the cup product (meaning, its degree $0$ part is nonzero), there is a unique Frobenius structure $\Phi = \Phi_0 + t\Phi_1 + \cdots$ for the quantum connection, whose constant term is the endomorphism 
\begin{equation} \label{eq:b-frobenius}
\Phi_0(x) = p^{-\mathrm{deg}(x)/2} b \smile x.
\end{equation}
For general $b$, this structure is not overconvergent. We now bring in an additional ingredient, Morita's $p$-adic Gamma function $\Gamma_p(z)$ (see Section \ref{sec:p-adic-functions}). This has a Taylor expansion with coefficients in $\bQ_p$. Hence, it gives rise to a multiplicative characteristic class for any complex vector bundle $E \rightarrow M$,
\begin{equation}
\Gamma_p(E) \in H^{\mathrm{even}}(M;\bQ_p).
\end{equation}

\begin{conjecture} \label{th:gamma-conjecture}
For the quantum connection on any monotone symplectic manifold, the Frobenius structure with constant term \eqref{eq:b-frobenius}, where $b = \Gamma_p(TM)$, is overconvergent.
\end{conjecture}

\begin{example} \label{th:cp1}
Take $M = \bC P^1$. In the standard basis $(1, \, [point])$ of $H^*(M)$, the quantum connection is 
\begin{equation} \label{eq:quantum-connection-cp1}
q\partial_q + 2\begin{pmatrix} 0 & q^2 \\ 1 & 0 \end{pmatrix};
\end{equation}
or after the change of variables \eqref{eq:t-q},
\begin{equation} \label{eq:quantum-connection-cp1-2}
t\partial_t + 2\begin{pmatrix} 0 & \pi^2 t^2 \\ 1 & 0 \end{pmatrix}. 
\end{equation}
The constant term from Conjecture \ref{th:gamma-conjecture} is
\begin{equation} \label{eq:constant-term-cp1}
\Phi_0 = 
\begin{pmatrix} 1 & 0 \\ 2\Gamma_p'(0) & 1/p \end{pmatrix}.
\end{equation}
\end{example}

Continuing the general discussion of $p$-adic differential equations, an overconvergent Frobenius structure can be evaluated at $t = \theta$ for some $(p-1)$-st root of unity $\theta$ (there is one such $\theta \in \bZ_p$ in each nonzero residue class). Since $\theta = \theta^p$, this is an endomorphism of the vector space on which our differential equation lives. In applications, its eigenvalues often have number-theoretic or algebro-geometric meaning.

\begin{example} \label{th:exponential-2}
Take the situation from Example \ref{th:exponential}, with $c = k \pi$ for some $k \in \bZ$. Then $\Phi(\theta) \in K$ is a $p$-th root of unity, for instance $\Phi(1) = \mu^k$ by \cite[p.~396]{robert}. In particular, $\mathrm{val}(\Phi(\theta)) = 0$.
\end{example}

\begin{conjecture} \label{th:gamma-conjecture-2}
Take the Frobenius structure from Conjecture \ref{th:gamma-conjecture}, and consider $\Phi(\theta)$ for some $(p-1)$-st root of unity $\theta$. Then, the valuations of the eigenvalues of $\Phi(\theta)$ are nonpositive integers; and among them (listed with multiplicities), the number of times that each $k$ appears equals the Betti number $b_{-2k}(M)$.
\end{conjecture}

We point out that for elementary reasons (see the discussion preceding Lemma \ref{th:equal-valuations}), the valuations under consideration in Conjecture \ref{th:gamma-conjecture-2} are independent of $\theta$.

Some cases of these conjectures reduce to classical results on $p$-adic differential equations.
\begin{itemize} \itemsep.5em
\item The pioneering paper \cite{dwork74} constructed an overconvergent Frobenius structure, or more precisely its inverse, for the Bessel equation; in our terms, this is the quantum connection for $M = \bC P^1$. The computation there proves that case of Conjecture \ref{th:gamma-conjecture}. Specifically, the inverse $\Phi(0)^{-1}$ of\eqref{eq:constant-term-cp1} appears in \cite[top of p.\ 723]{dwork74}; the result is not stated in terms of $\Gamma_p$, but it suffices to compare the formula there with \eqref{eq:first-gamma-derivative}.

\item Dwork's construction was generalized to $M = \bC P^n$ in \cite{sperber77}; explicit formulae for $\Phi(0)^{-1}$ are given in \cite[Equation 3.3.3, Theorem 4.2.11]{sperber77}. Again, the outcome is not stated in terms of the $p$-adic Gamma function, as would be required by Conjecture \ref{th:gamma-conjecture}. In principle, the compatibility of the two formulae should be verifiable by an elementary argument, since both use Taylor coefficients of the Dwork exponential; we have not checked that, but one can close the gap by using the computation from Section \ref{sec:constant-term} instead. Note that in this case, the overconvergent Frobenius is unique up to scalars, by \cite[p.~94--96]{dwork89}. The paper \cite{sperber80}, which is a sequel to \cite{sperber77}, computed the valuations of the eigenvalues of $\Phi(\theta)$, assuming $p \geq n+3$; it confirms Conjecture \ref{th:gamma-conjecture-2} in that case.
\end{itemize}

\begin{remark}
An entirely different approach to Frobenius structures for the quantum connection on $\bC P^n$ was given in \cite[Theorem 7.2]{smirnov24}. There, the Frobenius acts by $q \mapsto q^p$ rather than $t \mapsto t^p$, and overconvergence is replaced by a weaker condition \cite[Definition 5.1]{smirnov24}. Hence, it is not immediately clear how those results relate to ours. 
\end{remark}

The constructions in \cite{dwork74, sperber77} are explicit (and so is \cite{adolphson-sperber89}, which we will mention below). In contrast, more recent work tends to use the abstract formalism of overconvergent $F$-isocrystals and $p$-adic cohomology: see e.g.\ \cite[Appendix B]{kedlaya} for a bite-sized introduction; \cite{bourgeois99} for the relation between the two approaches; and \cite{xu-zhu19} for a representation-theoretic application, which in geometric terms includes the quantum connection for minuscule flag varieties \cite{lam-templier24} (and in particular, once more, projective spaces). While the abstract approach is a more natural and powerful way of constructing overconvergent Frobenius structures, it does not seem to lend itself naturally to the computation of the $t = 0$ term, which is a problem for the purposes of Conjecture \ref{th:gamma-conjecture}.

\begin{remark}
There are visible similarities between Conjecture \ref{th:gamma-conjecture} and a series of previous works: the lattice from \cite{iritani09} and \cite[Section 3]{katzarkov-kontsevich-pantev08}, and the ``Gamma conjecture I'' in \cite{galkin-golyshev-iritani16}. These also concern the quantum connection, and use the characteristic class derived from a modified version of the classical Gamma function. However, while the coincidence is striking, there seems to be no a priori mathematical connection between the two contexts. This unexplained parallel is a puzzling but well-known phenomenon in differential equations (a basic example is \cite[Equations 6.1 and 6.7]{boyarski80}; see also \cite[p.~47--48]{andre04} for a more succinct explanation).
\end{remark}

\begin{remark}
There is also a significant body of work concerning the $p$-adic quantum connection for Calabi-Yau manifolds, much of which deals with specific examples; e.g.\ \cite{shapiro12, candelas-delaossa-vanstraten21, beukers-vlasenko23}. We do not discuss it further here, since the aims appear to be different from the monotone case.
\end{remark}

The main result of this paper is the following:

\begin{theorem} \label{th:toric}
Conjecture \ref{th:gamma-conjecture} holds when $M$ is a smooth projective toric Fano variety.
\end{theorem}

The proof of Theorem \ref{th:toric} combines mirror symmetry with an explicit construction of the Frobenius structure. The (closed string, toric, equivariant) mirror symmetry statement we need is a simplified version of that in \cite{iritani17}; we derive it in Sections \ref{sec:equivariant}--\ref{sect:toricms}. Sections \ref{sec:dwork}--\ref{sec:overconvergent} concern the Frobenius structure. The construction has a precedent in \cite{adolphson-sperber89}, but that paper was focused on defining it for one value of $t$ at a time, and applications to exponential sums (hence does not mention differential equations at all). We set up the framework based on convergent power series in $t$, generalizing the approach in \cite{sperber77}. There are actually two versions of that formalism in the paper. The first one (Section \ref{sec:dwork}) is weaker but easier to do computations in; we determine the constant term in that setup, using an argument that's elementary but inspired by mirror symmetry (Section \ref{sec:constant-term}). The second, more sophisticated, version is then used to prove overconvergence (Section \ref{sec:overconvergent}). 

There is a relation between the quantum connection for projective spaces and Grassmannians. This was used in \cite[Section 6]{galkin-golyshev-iritani16} for the ordinary Gamma conjecture. Using the same idea, we obtain (Section \ref{sec:grassmannian}):

\begin{theorem} \label{th:grassmannian}
Conjecture \ref{th:gamma-conjecture} holds for Grassmannians $M = \mathit{Gr}_{k,N}$; and if $p \geq N+2$, so does Conjecture \ref{th:gamma-conjecture-2}.
\end{theorem}

Finally, in addition to the theoretical work, an Appendix to this paper provides experimental evidence supporting Conjectures \ref{th:gamma-conjecture} and \ref{th:gamma-conjecture-2}. Part of that (Examples \ref{th:cp2} and \ref{th:f1}) may seem redundant, since it overlaps with the previous literature or with Theorem \ref{th:toric}; however, even in those cases, the computation provides some new indications concerning the actual radius of convergence of the Frobenius. The remaining examples are not covered by existing results: among them is one class of non-algebraic monotone symplectic manifolds ($12$-dimensional twistor spaces, Example \ref{th:twistor-space}). 

{\em Acknowledgments.} The authors would like to thank Hiroshi Iritani, Andrey Smirnov, Steven Sperber, Daxin Xu, and Dingxin Zhang for helpful explanations. S.B. was partially supported by NSF grant DMS-2404843. D.P was partially supported by NSF grant DMS-2306204.

\section{Background on $p$-adic functions\label{sec:p-adic-functions}}
We recall $p$-adic estimates for the coefficients in some Taylor series. Legendre's formula implies 
\begin{equation} \label{eq:legendre}
\frac{m}{p-1} - \log_p(m) - 1 \leq \mathrm{val}(m!) \leq \frac{m-1}{p-1}. 
\end{equation}
As a consequence, $\exp(z)$ has radius of convergence $p^{-1/(p-1)}$. Next, take the Dwork exponential
\begin{equation} \label{eq:dwork-exponential}
D(z) \stackrel{\mathrm{def}}{=} \exp(z + z^p/p) = \sum_{m \geq 0} d_m z^m.
\end{equation}
By definition $D'(z) = (1+z^{p-1})D(z)$, which translates into the relation
\begin{equation} \label{eq:dwork-exponential-recursion}
d_m = (m+1)d_{m+1} - d_{m-p+1}
\end{equation}
between Taylor coefficients (with the negative index ones set to zero). The radius of convergence of \eqref{eq:dwork-exponential} is \cite[p.~389]{robert} $p^{(1-2p)/(p^3-p^2)}$, which is more than that of the exponential, but less than $1$ (one can think of \eqref{eq:dwork-exponential} as including the first term in the Artin-Hasse exponential, which has radius of convergence $1$). More precisely, we have \cite[Theorem 4.2.22]{cohen-vol1}
\begin{equation} \label{eq:dwork-exponential-val}
\mathrm{val}(d_m) \geq \frac{1-2p}{p^3-p^2} m. 
\end{equation}
These statements become more memorable after rescaling by \eqref{eq:pi}:
\begin{align} 
\label{eq:rescaled-exponential}
& \exp(\pi z) \text{ has radius of convergence $1$}, \\
\label{eq:rescaled-dwork}
& D(\pi z) = \exp(\pi (z - z^p)) \text{ has radius of convergence $p^{(p-1)/p^2} > 1$.}
\end{align}

The Morita $p$-adic Gamma function $\Gamma_p(z) \in \bZ_p^\times$ is a continuous function of $z \in \bZ_p$, which is determined by
\begin{equation}
\Gamma_p(z+1) = 
\begin{cases}
 -z \cdot\Gamma_p(z) & \text{if } z \in \bZ_p^{\times}, \\
-\Gamma_p(z) & \text{if } z \in p\bZ_p.
\end{cases}
\end{equation}
For the purpose of this paper, we follow the discussion of $\Gamma_p(z)$ from \cite[Chapter 11]{cohen-vol2}, which uses the collection of Mahler expansions \cite[Proposition 11.6.15(3)]{cohen-vol2}
\begin{equation} \label{eq:mahler}
\Gamma_p(-j-pz) = \sum_{m \geq 0} (-p)^m d_{mp+j} z(z-1)\cdots(z-m+1), \quad 
j \in \{0,\dots,p-1\}, \; 
z \in \bZ_p
\end{equation}
where the coefficients come from \eqref{eq:dwork-exponential}. Note that the formulae \eqref{eq:mahler} for different $j$ compute $\Gamma_p$ on disjoint subsets of $\bZ_p$. Nevertheless, they are related by \eqref{eq:dwork-exponential-recursion}, which implies
\cite[Corollary 11.6.8(3)]{cohen-vol2}
\begin{equation}
\Gamma_p(-j-pz) = \frac{\Gamma_p(-pz)}{(pz+1)\cdots(pz+j)}, \quad
\text{$j,z$ as before.}
\end{equation}

Our first task is to convert the $j = 0$ case of \eqref{eq:mahler} into a Taylor expansion for $\Gamma_p(z)$ around $z = 0$. From \eqref{eq:dwork-exponential-val} we have
\begin{equation} \label{eq:val-mahler}
\mathrm{val}( (-p)^m d_{pm} ) \geq m \frac{p^2-3p+1}{p^2-p}
\end{equation}
which goes to infinity as $m$ grows. Hence, we can extract the $z^k$ term from each summand in the Mahler expansion, and add those up over all $m \geq k$, which yields a Taylor series
\begin{equation} \label{eq:gamma-taylor}
\Gamma_p(z) = \sum_{k \geq 0} g_k z^k
\end{equation}
with $\mathrm{val}(g_k) \geq k (1-2p)/(p^2-p)$ (this is by no means an optimal bound, compare the proof of Lemma \ref{th:approximate-derivatives} below). For instance,
\begin{equation} \label{eq:first-gamma-derivative}
\Gamma_p'(0) = g_1 = \sum_{k>0} p^{k-1} d_{kp} \,(k-1)!
\end{equation}

\begin{remark}
There is a plethora of related formulae for $\Gamma_p(z)$. For instance \cite[Proposition 11.6.15(1)]{cohen-vol2}
\begin{equation}
\Gamma_p(z) = \sum_{m \geq 0} (-1)^m d_m z(z-1)\cdots(z-m+1), \quad z \in p\bZ_p
\end{equation}
from which one gets another formula \cite[Corollary 3.4]{shapiro12} for \eqref{eq:first-gamma-derivative},
\begin{equation}
\Gamma_p'(0) = -\sum_{k \geq 0} d_k\,(k-1)!
\end{equation}
\end{remark}

One has \cite[Corollary 11.6.8(2) and Proposition 11.6.12]{cohen-vol2}
\begin{equation} \label{eq:gamma-symmetry}
\Gamma_p(-z)\Gamma_p(z) = -\Gamma_p(1-z)\Gamma_p(z) = 1, \quad z \in p\bZ_p. 
\end{equation}
For the Taylor expansion, this implies \cite[Theorem 11.6.18(2)]{cohen-vol2}
\begin{equation} \label{eq:even-gamma-derivatives}
g_k = (-1)^{k/2-1} (g_{k/2}^2)/2 + \sum_{0<j<k/2} (-1)^{j-1} g_j g_{k-j} \quad \text{for even $k>0$,}
\end{equation}
which reduces the computation of derivatives $\Gamma^{(k)}_p(0)$ to the case of odd $m$: 
\begin{equation} \label{eq:second-gamma-derivative}
\begin{aligned}
& \Gamma_p''(0) = \Gamma_p'(0)^2, \\
& \Gamma_p^{(4)}(0) = 4\Gamma_p'(0)\Gamma^{(3)}_p(0) - 12 \Gamma_p'(0)^4, \\
& \dots
\end{aligned}
\end{equation}
Alternatively, one can work with the logarithm of the $p$-adic Gamma function,
\begin{equation} \label{eq:log-gamma}
\log \Gamma_p(z) = \sum_{m > 0} \frac{l_m}{m!} z^m.
\end{equation}
Because of \eqref{eq:gamma-symmetry} this is an odd function, meaning $l_m = 0$ for even $m$. The first nontrivial coefficients are
\begin{equation}
\begin{aligned}
& l_1 = \Gamma_p'(0), \\
& l_3 = \Gamma_p^{(3)}(0) - \Gamma_p'(0)^3, \\
& l_5 = \Gamma_p^{(5)}(0) + 10 \Gamma_p'(0)^2\Gamma_p^{(3)}(0) + 9 \Gamma_p'(0)^5, \\
& \dots
\end{aligned}
\end{equation}

\begin{remark}
The $l_m$ have a meaning in terms of $p$-adic $L$-functions \cite[Prop.~11.5.19]{cohen-vol2}. Concretely, $-\Gamma_p'(0)$ is the $p$-adic Euler constant, and 
\begin{equation} \label{eq:p-adic-zeta}
(-1)^ \frac{1}{(m-1)!} l_m = L_p(\omega^{1-m},m), \quad m \geq 2.
\end{equation}
The number on the right hand side of \eqref{eq:p-adic-zeta} is often denoted by $\zeta_p(m)$, since it can be obtained as a $p$-adic limit of classical zeta values; see e.g. \cite[Equation (1.1)]{lai-sprang23}.
\end{remark}

For computational purposes, it is useful to make parts of the discussion above more precise:

\begin{lemma} \label{th:approximate-derivatives}
To compute the $k$-th derivative $\Gamma_p^{(k)}(0)$ up to an error in $p^G\bZ_p$, it is sufficient to consider the finite sum over $m < M$ in the $j = 0$ case of \eqref{eq:mahler}, where $M$ is chosen such that
\begin{equation} \label{eq:m-bound}
M \geq \frac{kp}{(p-1)\log(p)} + 1 \quad \text{and} \quad \frac{k}{p-1} + M\frac{p-1}{p} - \log_p(k) - k\log_p(M-1) - \frac{2p-1}{p-1}  > G.
\end{equation}
\end{lemma}

\begin{proof}
The $z^k$ coefficient of $z(z-1)\cdots (z-m+1)$ is a sum of terms, each of which is obtained from $\pm (m-1)!$ by removing $(k-1)$ factors. Hence, from \eqref{eq:legendre} we get
\begin{equation}
\begin{aligned}
& \mathrm{val}\big(\text{$z^k$-coefficient of } z(z-1)\cdots(z-m+1) \big) \\ & \quad
\geq \mathrm{val}( (m-1)! ) - (k-1) \log_p(m-1) \geq \frac{m-1}{p-1} - k\log_p(m-1) - 1.
\end{aligned}
\end{equation}
Combining that with \eqref{eq:val-mahler} yields
\begin{equation} \label{eq:mahler-coefficients}
\mathrm{val}\big(\text{$z^k$-coefficient of }  (-p)^m d_{mp} z(z-1)\cdots(z-m+1) \big) \\
 \geq m \frac{p-1}{p} - k \log_p(m-1) - \frac{p}{p-1}.
\end{equation}
Hence, the contribution of the $m$-th term in \eqref{eq:mahler} to $\Gamma_p^{(k)}(0)$ has valuation
\begin{equation}
\begin{aligned}
& \mathrm{val}(k!) + \mathrm{val}\big(\text{$z^k$ coefficient of }  (-p)^m d_{mp} z(z-1)\cdots(z-m+1) \big) 
\\ & \qquad \geq \frac{k}{p-1} + m \frac{p-1}{p} - \log_p(k) - k\log_p(m-1) - \frac{2p-1}{p-1}.
\end{aligned}
\end{equation}
If we write $f(m)$ for the right hand side, fixing $k$, then $f'(m) = (p-1)/p - k/((m-1) \, \log(p)) \geq 0$ for $m \geq \frac{kp}{(p-1)\log(p)} + 1$. Hence, if $f(M) > G$ for some $M \geq kp/{(p-1)\log(p)} + 1$, then $f(m) > G$ for all $m > M$ as well. These are the conditions \eqref{eq:m-bound}.
\end{proof}

%

\section{Background on differential equations\label{sec:diff-eq}}
Consider formal power series connections in one variable $t$, of rank $r$, with a nilpotent simple pole, and having coefficients in the field $K$ from \eqref{eq:cyclotomic}. This means 
\begin{equation} \label{eq:connection}
\nabla_{t\partial_t} = t\partial_t + A(t), \quad A(t) = A_0 + tA_1 + t^2A_2 + \cdots,
\quad \text{$A_0$ nilpotent,}
\end{equation}
where the $A_m$ are $r \times r$ matrices over $K$. The first part of the following Lemma gives the formal gauge classification of such connections. For the second part, see \cite[Proposition 18.1.1]{kedlaya}.

\begin{lemma} \label{th:gauge-transformation}
(i) Given \eqref{eq:connection}, there is a unique
\begin{equation} \label{eq:gauge-transformation}
\Pi(t) = \Pi_0 + t\Pi_1 + \cdots, \quad \Pi_0 = I,
\end{equation}
which intertwines
\begin{equation} \label{eq:pole-connection}
\nabla_{t\partial_t}^{\mathrm{pole}} = t\partial_t + A_0
\end{equation}
and the original connection. Concretely, this means that 
\begin{equation} \label{eq:gauge-equation}
t\partial_t \Pi + A(t)\Pi(t) - \Pi(t) A_0 = 0.
\end{equation}

(ii) Suppose that $A(t)$ converges on the open disc with radius $\rho>0$. Then: $\Pi(t)$ converges on that disc iff the same holds for its (formal power series) inverse $\Pi(t)^{-1}$.

(iii) Suppose that there is some $\sigma \in \bR$ such that
\begin{equation} \label{eq:a-growth}
\mathrm{val}(A_m) \geq \sigma m \;\; \text{ for all $m \geq 0$.}
\end{equation}
Then $\Pi$ converges, and is bounded, on the open disc with radius $p^{\sigma-2N/(p-1)}$, where $(N+1)$ is the size of the largest Jordan block for $A_0$.
\end{lemma}

\begin{proof}
(i) Write out \eqref{eq:gauge-equation} order by order:
\begin{equation} \label{eq:gauge-recursion}
m\Pi_m + [A_0,\Pi_m] = -\!\!\sum_{0 \leq i < m} A_{m-i}\Pi_i, \quad m \geq 1.
\end{equation}
Let $R_m$ be the right hand side of \eqref{eq:gauge-recursion}. Since $A_0$ is nilpotent, we can solve for $\Pi_m$ by the formula
\begin{equation} \label{eq:invert-commutator}
\Pi_m = R_m/m - [A_0,R_m]/m^2 + [A_0,[A_0,R_m]]/m^3 + \cdots, 
\end{equation}
where the last term is divided by $m^{2N}$. 

(ii) Suppose that $\Pi(t)$ converges on the open disc with radius $\rho$. Writing $\wp(t) = \mathrm{det}\,\Pi(t)$, $a(t) = \mathrm{tr}\, A(t)$, $a_0 = \mathrm{tr}(A_0)$, one has $\wp(0) = 1$ and 
\begin{equation} \label{eq:determinant-equation}
t\partial_t \wp(t) + (a(t) - a_0) \wp(t) = 0.
\end{equation}
This implies (using the Weierstrass preparation theorem, e.g.\ \cite[p.~105]{koblitz}) that $\wp(t) \neq 0$ for all $|t| < \rho$. As a consequence, $1/\wp(t)$, as defined by its formal power series, is also convergent for $|t| < \rho$ \cite[Appendix to Ch.~2]{dwork-gerotto-sullivan}. Note that all minors of $A(t)$ also converge for $|t| < \rho$. Therefore, Cramer's rule gives a formula for $\Pi(t)^{-1}$ with the same convergence behaviour. The argument in converse direction is parallel.

(iii) Taking \eqref{eq:a-growth} into account, the equation \eqref{eq:invert-commutator} implies that 
\begin{equation}
\begin{aligned}
\mathrm{val}(\Pi_m) & \geq \mathrm{val}(R_m) - 2N\mathrm{val}(m)
\\ &
\geq \mathrm{min}\Big(\mathrm{val}(\Pi_0) + \sigma m, \dots, \mathrm{val}(\Pi_{m-1}) + \sigma\Big) - 2N \mathrm{val}(m);
\end{aligned}
\end{equation}
or equivalently
\begin{equation}
\mathrm{val}(\Pi_m) - \sigma m \geq \mathrm{min}\Big(\mathrm{val}(\Pi_0),\dots, \mathrm{val}(\Pi_{m-1}) - \sigma (m-1)\Big) - 2N\mathrm{val}(m).
\end{equation}
By induction, and using \eqref{eq:legendre}, one gets 
\begin{equation}
\mathrm{val}(\Pi_m) - \sigma m \geq -2N\mathrm{val}(m!) \geq -2N \frac{m-1}{p-1}.
\end{equation}
\end{proof}

\begin{example}
Take the connection
\begin{equation}
\nabla_{t\partial_t} = t\partial_t - \begin{pmatrix} 0 & t \\ 1 & 0 \end{pmatrix}.
\end{equation}
The equation $\nabla_{t\partial_t}^{\mathrm{pole}}(v^{\mathrm{pole}}) = 0$ has the constant solution $v^{\mathrm{pole}} = (0,1)$. By applying Lemma \ref{th:gauge-transformation}(i), one sees that there is a (unique) solution $v(t) = \Pi(t) v^{\mathrm{pole}}$ of $\nabla_{t\partial_t}v = 0$ with the same initial term. One can write this out directly: $v(t) = (x(t),x'(t))$, where $x''(t) = x(t)/t$, hence
\begin{equation}
x(t) = \sum_{k=1}^\infty \frac{t^k}{(k-1)! k!}
\end{equation}
This series has radius of convergence $p^{-2/(p-1)}$, compatibly with Lemma \ref{th:gauge-transformation}(iii) ($\sigma = 0$, $N = 1$).
\end{example}

For context concerning the following notions, see e.g.\ \cite[Section 2]{kedlaya22} (compared to that reference, we have drastically reduced the scope). The Frobenius pullback of a connection \eqref{eq:connection} is its pullback under $F(t) = t^p$, explicitly:
\begin{equation} \label{eq:frobenius-connection}
\nabla_{t\partial_t}^{(p)} = t\partial_t + p A(t^p).
\end{equation}

\begin{definition} \label{th:structure}
A Frobenius structure for $\nabla_{t\partial_t}$ is an isomorphism
\begin{equation}
\Phi(t) = \Phi_0 + t\Phi_1 + \cdots, \;\; \Phi_0 \text{ invertible,}
\end{equation}
intertwining \eqref{eq:frobenius-connection} and the original connection. Explicitly, this satisfies
\begin{equation} \label{eq:structure}
t\partial_t \Phi + A(t)\Phi(t) - p \Phi(t) A(t^p) = 0.
\end{equation}
%
\end{definition}
%
The equation \eqref{eq:structure} is quite similar to \eqref{eq:gauge-equation}. This leads to the following, which is mostly a counterpart of Lemma \ref{th:gauge-transformation}; the exception is the last part, an instance of Dwork's trick, see e.g.\ \cite[Corollary 17.2.4]{kedlaya}.

\begin{lemma} \label{th:make-frobenius}
(i) The constant term $\Phi_0$ of any Frobenius structure satisfies
\begin{equation} \label{eq:constant-term}
A_0\Phi_0 = p\Phi_0A_0.
\end{equation}
Conversely, every invertible $\Phi_0$ with \eqref{eq:constant-term} is the constant term of a unique Frobenius structure.

(ii) Consider a Laurent series solution $\Phi(t) = \Phi_m t^m + \Phi_{m+1} t^{m+1} + \cdots$ of \eqref{eq:structure}, for some $m<0$. Then $\Phi_m = \Phi_{m+1} = \cdots = \Phi_{-1} = 0$, so $\Phi$ is in fact a Frobenius structure.

(iii) Suppose that $A(t)$ converges on the open disc of radius $\rho>0$. Let $\Phi$ be any Frobenius structure. Then: $\Phi(t)$ converges on the open disc of radius $\mathrm{min}(\rho,\rho^{1/p})$ iff the same holds for $\Phi(t)^{-1}$.

(iv) Suppose that the map $\Pi(t)$ from Lemma \ref{th:gauge-transformation} converges on an open disc of radius $\rho$. Then any Frobenius structure converges on the open disc of radius $\mathrm{min}(\rho,\rho^{1/p})$.

(v) Suppose that our equation admits a Frobenius structure $\Phi$ which converges on the open disc of radius $1$. Then, if the map $\Pi$ from Lemma \ref{th:gauge-transformation} converges on some open disc around $0$, it must converge on the open disc of radius $1$ (and by (iv), the same will then hold for all Frobenius structures).
\end{lemma}

\begin{proof}
(i) The necessity of \eqref{eq:constant-term} is clear. The order-by-order equation for $\Phi_m$, $m \geq 1$, is
\begin{equation} \label{eq:order-by-order-2}
\begin{aligned}
& m \Phi_m + A_0\Phi_m - p\Phi_m A_0 \\ & \qquad \qquad = 
\sum_{0 \leq i < m} -A_{m-i}\Phi_i + p \Phi_i A_{\frac{m-i}{p}};
\end{aligned}
\end{equation}
Here, the convention is that any terms $A_j$ with non-integer index $j$ are set to zero. One argues as in Lemma \ref{th:gauge-transformation}(i) that these equations have a unique solution using the fact that the map $\Phi_m \mapsto A_0 \Phi_m - p\Phi_m A_0$ is nilpotent.

(ii) There is an analogue of \eqref{eq:order-by-order-2} for the most negative order $m<0$, which is $m\Phi_m + A_0 \Phi_m - p\Phi_m A_0 = 0$. By the same argument as in (i), that implies $\Phi_m = 0$.

(iii) Set $\varphi(t) = \mathrm{det}\,A(t)$, $a(t) = \mathrm{tr}\,A(t)$. Then
\begin{equation}
t\partial_t \varphi + (a(t)-p\, a(t^p)) \varphi(t) = 0.
\end{equation}
Note that $a(t)-p\, a(t^p)$ converges on the open disc of radius $\mathrm{min}(\rho,\rho^{1/p})$. From here on, the argument proceeds as in Lemma \ref{th:gauge-transformation}(ii).

(iv) The constant endomorphism $\Phi_0$ is a Frobenius structure for \eqref{eq:pole-connection}. This can be turned into one for $\nabla_{t \partial_t}$, with the same constant term, by the formula
\begin{equation} \label{eq:structure-by-gauge}
\Phi(t) = \Pi(t) \Phi_0 \Pi^{-1}(t^p),
\end{equation}
which converges over the open disc of radius $\mathrm{min}(\rho,\rho^{1/p})$.

(v) Consider
\begin{equation}
\tilde{\Pi}(t) = \Phi(t) \Pi(t^p) \Phi_0^{-1}.
\end{equation}
If $\Pi$ converges for $|t| < \rho < 1$, then $\tilde{\Pi}$ converges on the larger disc $|t| < \rho^{1/p}$. On the other hand, $\tilde{\Pi}$ again satisfies \eqref{eq:gauge-equation} and has constant term equal to the identity, hence must be equal to $\Pi$. This shows that the radius of convergence of $\Pi$ cannot be less than $1$.
\end{proof}

\begin{remarks}
(i) The equation \eqref{eq:constant-term} implies that $A_0$ and $pA_0$ are conjugate, which is only possible if $A_0$ is nilpotent: connections with other kinds of pole terms cannot have Frobenius structures in our sense. However, they could still admit some $\Phi$ as in Lemma \ref{th:make-frobenius}(ii), meaning a Laurent series.

(ii) The Hukuhara-Turrittin-Levelt classification prevents connections with irregular singularities at $t = 0$ from admitting a Frobenius structure, even a Laurent series one.
\end{remarks}

\begin{definition} \label{th:overconvergent}
A Frobenius structure $\Phi$ is called overconvergent if it converges on a disc of radius $>1$.
\end{definition}

Take a Frobenius structure, and consider its characteristic polynomial
\begin{equation} \label{eq:characteristic-polynomial}
\phi(t,z) = \mathrm{det}(zI - \Phi(t)) = \phi_0(t) + z\phi_1(t) + \cdots + z^r \phi_r(t).
\end{equation}
Suppose that $\Phi$ is overconvergent, therefore can be evaluated at $t = \theta$, where $\theta$ is a $(p-1)$-st root of unity. $\Phi(\theta)$ is a fibrewise endomorphism, so it makes sense to consider the eigenvalues of that endomorphism, which are the roots of $\phi(\theta,z)$. The simplest aspect is the $p$-adic valuation of the eigenvalues. These can be read off from the valuations of the $\phi_k(\theta)$ from \eqref{eq:characteristic-polynomial}, thanks to the following basic fact. Let 
\begin{equation}
g(z) = b_0 + b_1 z + \cdots b_d z^d
\end{equation}
be a polynomial of degree $d>0$ over $K$, with $b_0 = g(0) \neq 0$. The Newton polygon of $g$ is the convex hull of 
\begin{equation}
\bigcup_{b_i \neq 0} \{(x,y) \;:\; 0 \leq x \leq i, \, y \geq \mathrm{val}(b_i)\} \subset \bR^{\geq 0} \times \bR.
\end{equation}
Then \cite[Theorem 7.4.47]{gouvea}:

\begin{lemma} \label{th:newton}
The valuations of the roots of $g$ (in the algebraic closure of $K$) are $(-1)$ times the slopes of the non-vertical sides of its Newton polygon. Here, each slope appears with multiplicity equal to the length of the underlying horizontal segment.
\end{lemma}

We conclude our exposition with a minor reformulation, which brings one back to more familiar territory insofar as the quantum connection is concerned. Namely, let's use the rescaled variable $q$ from \eqref{eq:t-q}, so that our connection is given by
\begin{equation}
\nabla_{q\partial_q} = q\partial_q + \bar{A}_0 + \bar{A}_1 q + \bar{A}_2 q^2 + \dots,
\end{equation}
with $A_m = \pi^m \bar{A}_m$. In these terms, the definition of Frobenius structure involves the pullback $q \mapsto -q^p/p$. Such a structure is 
\begin{equation} \label{eq:frobenius-2}
\begin{aligned}
& \bar{\Phi}(q) = \bar\Phi_0 + \bar\Phi_1 q + \cdots, \\
& q\partial_q \bar\Phi + \bar{A}(q)\bar\Phi(q) - p \bar{\Phi}(q) \bar{A}(-q^p/p) = 0;
\end{aligned}
\end{equation}
or written out as in \eqref{eq:order-by-order-2},
\begin{equation} \label{eq:order-by-order-3}
\begin{aligned}
& m \Phi_m + A_0\Phi_m - p\Phi_m A_0 \\ & \qquad \qquad = 
\!\!\sum_{0 < i \leq m} -A_{m-i}\Phi_i - \Phi_i A_{\frac{m-i}{p}} (-p)^{1-(m-i)/p}.
\end{aligned}
\end{equation}
Again, this is related to our usual definition by a simple rescaling, which for the coefficients is $\Phi_m = \pi^m \bar{\Phi}_m$. When thinking in these terms, overconvergence means that $\bar\Phi(q)$ has convergence radius $> p^{-1/(p-1)}$ in $q$. 

\begin{example}
Take $\nabla_{q\partial_q} = q\partial_q - cq$ with $c \in \bZ$, which corresponds to $\nabla_{t\partial_t} = t\partial_t - c\pi t$ (compare Examples \ref{th:exponential} and \ref{th:exponential-2}). The Frobenius structure is $\bar\Phi(q) = D(q)^c$, which is defined over $\bQ_p$.
\end{example}

Similarly, when applied to the quantum connection, we have $\nabla_{q\partial_q}$ defined over $\bQ_p$. Hence, if the constant term of $\bar\Phi_0 = \Phi_0$ is defined over $\bQ_p$, then the same holds for the entire $\bar\Phi(q)$. Assuming overconvergence, one then considers the roots of $\phi(\theta,z) = \bar\phi(\pi\theta,z)$, where
\begin{equation} \label{eq:characteristic-polynomial-2}
\bar\phi(q,z) = \mathrm{det}\big(z I - \bar\Phi(q) \big) = \bar\phi_0(q) + z\bar\phi_1(q) + \cdots + z^r \bar\phi_r(q)
\end{equation}
again has $\bQ_p$-coefficients.
By Lemma \ref{th:newton}, the valuations of the roots depend only on the valuations of the numbers $\bar\phi_k(\pi\theta)$. These are independent of $\theta$, due to the following elementary observation:

\begin{lemma} \label{th:equal-valuations}
Let $g(q)$ be a power series over $\bQ_p$ with convergence radius $> p^{-1/(p-1)}$. Then, the valuations of $g(\pi\theta)$, for different choices of $(p-1)$-st roots of unity $\theta$, are the same.
\end{lemma}

\begin{proof}
Write $g(q) = \sum_{j=0}^{p-2} q^j g_j(q^{p-1})$, so that
\begin{equation}
g(\pi\theta) = \sum_{j=0}^{p-2} (\pi\theta)^j g_j(-p),
\end{equation}
The $j$-th term has valuation in $j/(p-1) + \bZ$, so these valuations are pairwise distinct, and
\begin{equation} \label{eq:fractional-valuation}
\mathrm{val}(g(\pi\theta)) = \mathrm{min}_j\big( \mathrm{val}(g_j(-p)) + j/(p-1) \big).
\end{equation}
\end{proof}


\section{The $P$-adic Gamma class\label{sec:gamma-class}}
We now return to the quantum connection. The map $\Pi$ from Lemma \ref{th:gauge-transformation} is essentially Givental's fundamental solution (see e.g.\ \cite[Section 1.3]{pandharipande97}). 
\begin{equation} \label{eq:givental}
\begin{aligned}
& \int_M x \smile \Pi(y)  = \int_M x \smile y + \sum_{m>0} q^{(\mathrm{deg}(x)+\mathrm{deg}(y))/2+m+1-\mathrm{dim}_{\bC}(M)}(-1)^{m+1} \langle x, \psi^m(y) \rangle, \\
& \int_M x \smile \Pi^{-1}(y) = \int_M x \smile y + \sum_{m>0} q^{(\mathrm{deg}(x)+\mathrm{deg}(y))/2+m+1-\mathrm{dim}_{\bC}(M)} \langle \psi^m(x), y \rangle.
\end{aligned}
\end{equation}
Here, the expression in $\langle \cdot \rangle$ is a genus zero two-point Gromov-Witten invariant with gravitational descendants inserted at one of the marked points. For dimension reasons, it contains only contributions from curves in homology classes $A$ such that
\begin{equation}
\mathrm{dim}_{\bC}(M) + \int_A c_1(\mathit{TM}) - 1 =  (\mathrm{deg}(x) + \mathrm{deg}(y))/2 + m.
\end{equation}
We have written this in terms of the variable $q$, which is more natural geometrically. To discuss Frobenius structures, we now apply the usual change of variables \eqref{eq:t-q}, so we are considering the quantum connection in the form \eqref{eq:quantum-connection-2}.

\begin{lemma} \label{th:general-structure}
(i) Take any $b \in H^{\mathrm{even}}(M;K)$ whose degree $0$ part is nonzero. There is a unique Frobenius structure $\Phi(t)$ for $\nabla_{t\partial_t}$, whose constant term is \eqref{eq:b-frobenius}.

(ii) Any such structure $\Phi(t)$ has radius of convergence $\geq p^{(1-2\mathrm{dim}_{\bC}(M))/(p-1)}$.

(iii) Assume that for some choice of $b$, the associated Frobenius structure converges on the open disc of radius $1$. Then the same holds for all Frobenius structures.
\end{lemma}

\begin{proof}
These are applications of Lemma \ref{th:make-frobenius}, and only (ii) requires some explanation. Recall that the coefficients in \eqref{eq:quantum-product} are integers, as one can see from their symplectic topology definition (see e.g.\ \cite[Chapter 7]{mcduff-salamon-big}). Therefore, the coefficients in \eqref{eq:quantum-connection-2} satisfy the assumption from Lemma \ref{th:gauge-transformation}(iii) with $\sigma = \mathrm{val}(\pi) = 1/(p-1)$; and the nilpotent pole term $x \mapsto c_1(TM) \smile x$ has maximal Jordan block size $N + 1= \dim_{\bC} M + 1$. One then uses Lemma \ref{th:make-frobenius}(iv).
\end{proof}

Let
\begin{equation}
\Gamma_p(E) \in H^{\mathrm{even}}(M;\bQ_p)
\end{equation}
be the multiplicative characteristic class associated to the $p$-adic Gamma function, applied to a complex vector bundle $E \rightarrow M$. In terms of Chern roots $c(E) = 1 + c_1(E) + c_2(E) + \cdots = \prod_{i=1}^n (1+r_i)$, 
\begin{equation} \label{eq:gamma-characteristic-class}
\begin{aligned}
\Gamma_p(E) = \prod_{i=1}^n \Gamma_p(r_i) &
= 1 + \Gamma_p'(0)c_1(E) + \Gamma_p'(0)^2 (c_1(E)^2/2) + 
\\[-.5em] 
& 
+ \Gamma_p'''(0) \big(c_3(E)/2 - c_1(E)c_2(E)/2 + c_1(E)^3/6 \big) \\ 
& 
+ \Gamma_p'(0)^3 \big(-\!c_3(E)/2 + c_1(E)c_2(E)/2 \big) + \cdots
\end{aligned}
\end{equation}
where we have used \eqref{eq:second-gamma-derivative}. Note that as a consequence of \eqref{eq:gamma-symmetry},
\begin{equation} \label{eq:inverse-gamma}
\Gamma_p(E^\vee) = \Gamma_p(E)^{-1}.
\end{equation}
Equivalently one can work in terms of \eqref{eq:log-gamma} and the Chern character, leading to slightly simpler formulae using only $\mathit{ch}_m(E)$ for odd $m$. Explicitly,
\begin{equation} \label{eq:log-gamma-class}
\begin{aligned}
\log\, \Gamma_p(E) & = \sum_{m \geq 1} l_m \,\mathit{ch}_m(E) \\
& = \Gamma_p'(0) \,\mathit{ch}_1(E) + \big(\Gamma'''_p(0) - \Gamma_p'(0)^3\big) \,\mathit{ch}_3(E) \\
& \qquad + \big
( \Gamma^{(5)}_p(0) - 10 \Gamma_p'(0)^2 \Gamma_p'''(0) + 
9 \Gamma_p'(0)^5\big) \mathit{ch}_5(E) + \cdots
\end{aligned}
\end{equation}

At this point, we have assembled all the ingredients that enter into Conjectures \ref{th:gamma-conjecture} and \ref{th:gamma-conjecture-2}. A few additional remarks may be helpful.

\begin{remark}
Suppose that Conjecture \ref{th:gamma-conjecture} holds. From Lemma \ref{th:make-frobenius}(v), it then follows that the series in \eqref{eq:givental} must converge for $|q| < p^{-1/(p-1)}$. Equivalently, $-\mathrm{val}(m! \langle x, \psi^m(y) \rangle )$ grows slower than linearly.
\end{remark}

\begin{remark}
Take a monotone {\em Spin} Lagrangian submanifold $L \subset M$, with nonzero (Poincar{\'e} dual) cohomology class $[L]$. To fit it into the framework of our discussion, let's assume that $\mathrm{dim}(L) = \mathrm{dim}_{\bC}(M)$ is even. There is a disc-counting invariant $w_L \in \bZ$, such that \cite[Proposition 2.4.A]{biran-membrez16}
\begin{equation}
c_1(TM) \ast_q [L] = q\, w_L [L].
\end{equation} 
In particular, \eqref{eq:quantum-connection-2} restricted to the one-dimensional subspace spanned by $[L]$ is $t\partial_t + \pi t w_L$. Because $TM|L \iso TL \otimes_{\bR} \bC$, we have $\mathit{ch}_m(TM) \smile [L] = 0$ for odd $m$, hence 
\begin{equation}
\Gamma_p(TM) \smile [L] = [L].
\end{equation}
The Frobenius structure from Conjecture \ref{th:gamma-conjecture}, restricted to our one-dimensional subspace, is $\Phi(t) = p^{-\mathrm{dim}_{\bC}(M)/2} D(\pi t)^{-w_L}$; except for the constant, this is as in Example \ref{th:exponential}, hence overconvergent. From Example \ref{th:exponential-2} we see that the valuation of $\Phi(\theta)$ is $-\mathrm{dim}_{\bC}(M)/2$, compatibly with Conjecture \ref{th:gamma-conjecture-2}.
\end{remark}

\begin{remark}
Suppose that there was a monotone symplectic manifolds $M$ which, for generic almost complex structure, contains no rational curves (this can't happen in algebraic geometry, and is thought to be unlikely in symplectic geometry, but can't be ruled out at present). Then for any $b$, the constant Frobenius $\Phi = \Phi_0$ is obviously overconvergent; and any such Frobenius also satisfies Conjecture \ref{th:gamma-conjecture-2}.
\end{remark}

\section{Mirror symmetry for Fano toric varieties\label{sec:ms}}
The aim of this section, which follows (a simplified version of) arguments from \cite{iritani17, iritani2017b}, is to prove a form of closed string mirror symmetry for compact Fano toric varieties. The proof passes through computations in equivariant quantum cohomology, but that is in a sense a technical tool; the ultimately desired statements drop the equivariance. 

\begin{convention}
Our setup will involve several quantum variables $(q_1,\dots,q_r)$, one for each toric divisor. Nevertheless, the collection of such variables will often be denoted by a single $q$, as in the quantum product \eqref{eq:equivariant-quantum-product} or the superpotential \eqref{eq:q-superpotential}. In particular, we use notation
\begin{equation} \label{eq:multi-q}
\begin{aligned}
& \vec{v} = (v_1,\dots,v_r) \in \bZ^r, \\
& q^{\vec{v}} = q_1^{v_1} \cdots q_r^{v_r}.
\end{aligned}
\end{equation}
The reason for this multivariable setup is not to have a richer theory: rather, we need it in order to carry out a specific argument (Section \ref{sec:constant-term}). Following that, we will specialize back to a single variable (Section \ref{sec:single-variable}).
\end{convention}

\subsection{Fano toric varieties\label{sec:toric}}
Let $T$ be a rank $n$ torus, and $N = \mathit{Hom}(\bC^*,T)$ its cocharacter lattice (of one-parameter subgroups). We fix an identification $T = (\bC^*)^n$, and correspondingly $N = \bZ^n$.  Take a fan in $N \otimes \bR$ giving rise to a smooth projective Fano toric variety $M$, on which $T$ acts. Let $\bfe_1,\dots,\bfe_r \in N$ be the primitive generators of the rays of the fan. The following statement expresses combinatorially the smoothness of $M$:

\begin{lemma} \label{th:toric-smoothness}
Each $k$-dimensional cone of the fan is of the form $\bR^{\geq 0} \bfe_{i_1} + \cdots + \bR^{\geq 0} \bfe_{i_k}$, where $(\bfe_{i_1},\dots,\bfe_{i_k})$ form part of a basis of the abelian group $N$.
\end{lemma}

Given $\bfn \in N$, write 
\begin{equation} \label{eq:nonnegative-linear}
\mathbf{n} = \psi_1 \bfe_1 + \cdots + \psi_r \bfe_r, \quad \psi_i = \psi_i(\bfn) \geq 0,
\end{equation}
with the property that all $\bfe_i$ with $\psi_i>0$ belong to a single cone of the fan. This expression is unique, by Lemma \ref{th:toric-smoothness}, hence defines a function
\begin{equation} \label{eq:vec-psi}
\vec{\psi} = (\psi_1,\dots,\psi_r): N \longrightarrow (\bZ^{\geq 0})^r.
\end{equation}
Then,
\begin{equation} \label{eq:toric-weight-function}
w = \psi_1 + \cdots + \psi_r: N \longrightarrow \bZ^{\geq 0}
\end{equation}
is the unique function which is linear on each cone of the fan, and such that $w(\bfe_i) = 1$ for all $i$. We extend $w$ to $N \otimes \bR \rightarrow \bR^{\geq 0}$, without changing notation. The following statement expresses combinatorially the fact that $M$ is Fano \cite[Lemma 6.1.13 and Theorem 6.1.14]{cox-little-schenk}:

\begin{lemma} \label{th:toriccombinatorics}
(i) The function $w$ is convex, meaning that
\begin{equation} \label{eq:convex-psi}
w(r_1 \bfn_1 + \cdots + r_k \bfn_k) \leq r_1 w(\bfn_1) + \cdots + r_k w(\bfn_k)
\quad
\text{for $\bfn_j \in N \otimes \bR$, $r_j > 0$.}
\end{equation}

(ii) $w$ is strictly convex in the following sense: if in \eqref{eq:convex-psi}, not all $\bfn_j$ are contained in a single cone of the fan, then the inequality is strict.
\end{lemma}

Geometrically, the $\{ \bfe_i \}$ correspond to the toric divisors $\{ D_i \subset M \}$, hence to equivariant cohomology classes $[D_i]_T \in H^2_T(M;\bZ)$. These form a basis of that (free) cohomology group, $H^2_T(M;\bZ) = \bZ^r$. Taking the $k$-th components of the vectors $\bfe_i$, we get classes 
\begin{equation}\label{eq:lambdak}
\lambda_k = e_{1,k} [D_1]_T + \cdots + e_{r,k} [D_r]_T \in H^2_T(M;\bZ), \quad k = 1,\dots,n.
\end{equation}
The $\lambda_k$ are the generators of $H^2_T(\mathit{point};\bZ) \subset H^2_T(M;\bZ)$; and 
\begin{equation} \label{eq:forget-2}
H^2(M;\bZ) = H^2_T(M;\bZ)/(\bZ \lambda_1 \oplus \cdots \oplus \bZ \lambda_n).
\end{equation}
The cohomology $H^*(M;\bZ)$ is concentrated in even degrees, and is a free abelian group. Hence, the Leray--Serre spectral sequence from $H^*_T(\mathit{point};\bZ) \otimes H^*(M;\bZ)$ to $H^*_T(M;\bZ)$ degenerates, and the forgetful map 
\begin{equation} \label{eq:forgetful}
H^*_T(M;\bZ) \longrightarrow H^*(M;\bZ) 
\end{equation}
is split surjective. From a choice of splitting, we get an isomorphism of $H^*_T(\mathit{point};\bZ)$-modules,
\begin{equation} \label{eq:equi-free}
H^*_T(M;\bZ) \iso H^*(M;\bZ) \otimes H^*_T(\mathit{point};\bZ),
\end{equation}
which reduces to \eqref{eq:forgetful} when composed with $H^*_T(\mathit{point};\bZ) \rightarrow \bZ$. This shows that \eqref{eq:forgetful} induces an isomorphism
\begin{equation} \label{eq:non-equivariant-limit}
H^*_T(M;\bZ)/(\lambda_1,\dots,\lambda_n)H^*_T(M;\bZ) \stackrel{\iso}{\longrightarrow} H^*(M;\bZ),
\end{equation}
generalizing \eqref{eq:forget-2}. For any $\mathbf{n} \in N$, define
\begin{equation} \label{eq:srbasis} 
\alpha_{\mathbf{n}} = \prod_{i=1}^{r} [D_i]^{\psi_i(\mathbf{n})} \in H_T^{2w(\mathbf{n})}(M;\mathbb{Z}). \end{equation} 
These elements form a basis for $H^*_T(M;\mathbb{Z})$ over $\bZ$. For future purposes, we introduce the following:

\begin{definition} \label{def:f-basis}
Let $F$ be any field. An $F$-basis set is a finite set $\Delta \subset N$ such that the non-equivariant images of the classes $(\alpha_{\bfn})_{\bfn \in \Delta}$ form a basis of $H^*(M;F)$.
\end{definition}

\begin{lemma} \label{th:f-basis}
(i) For any field $F$, there is an $F$-basis set.

(ii) If $\Delta$ is an $F$-basis set, it has the same property for any other field with the same characteristic.

(iii) If $\Delta$ is an basis set for a characteristic $p>0$ field, then the same holds in characteristic $0$.
\end{lemma}

\begin{proof}
(i) By applying the universal coefficient theorem, one sees that the $(\alpha_{\bfn})_{\bfn \in N}$ form a basis of $H^*_T(M;F)$, and that the forgetful map \eqref{eq:forgetful} is onto with $F$-coefficients.

(ii), (iii) Take the map $\bZ^\Delta \rightarrow H^*(M;\bZ)$ which takes the standard generators to the non-equivariant images of the $\alpha_{\bfn}$. We can think of it as a matrix with integer coefficients. Then, $\Delta$ is a generating set for fields of characteristic $0$ if the matrix (is square and) has nonzero determinant; and in characteristic $p$, if the determinant is not a multiple of $p$.
%
%
\end{proof}

Returning to integer coefficients, we recall the combinatorial description of the ring structure of $H_T^*(M;\bZ)$ (see e.g.\ \cite{davisjanuszkiewicz91}). The equivariant cup product is 
\begin{equation} \label{eq:sr}
\alpha_{\mathbf{m}} \smile \alpha_{\mathbf{n}} = 
\begin{cases}
 \alpha_{\mathbf{m} +\mathbf{n}} & \text{if } \mathbf{m} \text{ and } \mathbf{n} \text{ lie in a common cone}, \\
0 & \text{otherwise}.
\end{cases}
\end{equation}
Applying \eqref{eq:forgetful} recovers the simpler fact that the $[D_i]$ are generators of $H^*(M;\bZ)$ as a ring.

The anticanonical line bundle corresponds to the function \eqref{eq:toric-weight-function}, which geometrically means that
\begin{equation} \label{eq:firstchern}
c_1(\mathit{TM}) = \sum_{i=1}^r [D_i] \in H^2(M;\bZ).
\end{equation}
Extending \eqref{eq:firstchern}, the total Chern class of our toric variety \cite[Proposition 13.1.2]{cox-little-schenk} is 
\begin{equation}
c(\mathit{TM}) = \prod_i (1+[D_i]) \in H^*(M;\bZ).
\end{equation}
Correspondingly, 
\begin{equation}
\Gamma_p(\mathit{TM}) = \prod_{i=1}^r \Gamma_p([D_i]) \in H^*(M;\bQ_p);
\end{equation}
and by \eqref{eq:inverse-gamma}, also
\begin{equation} \label{eq:inverse-toric-gamma}
\Gamma_p(\mathit{TM})^{-1} = \Gamma_p(TM^\vee) = \prod_{i=1}^r \Gamma_p(-[D_i]).
\end{equation}

\subsection{Equivariant quantum cohomology\label{sec:equivariant}}
Identify $H_2^T(M;\bZ) = \mathbb{Z}^r$ via the basis dual to $[D_i]_T \in H^2_T(M;\bZ)$, so that $\bQ[H_2^T(M)] = \bQ[q_1^{\pm 1},\dots,q_r^{\pm 1}]$. Recall the notation \eqref{eq:multi-q}. We also write 
\begin{equation}
|\vec{v}| = v_1+\cdots+v_r.
\end{equation}
Consider the ring
\begin{equation} \label{eq:lambdageq0} 
\Lambda = \big\{ \sum_{|\vec{v}| \geq 0} a_{\vec{v}} q^{\vec{v}} \big\} \subset \bQ[q_1^{\pm 1},\dots,q_r^{\pm 1}]
\end{equation}
and the ideal
\begin{equation}
I\Lambda = \big\{ \sum_{|\vec{v}| > 0} a_{\vec{v}} q^{\vec{v}} \big\} \subset \Lambda.
\end{equation}
Occasionally we find it useful to have a grading, given by setting $\mathrm{deg}(q_i) = 2$; which means that $\Lambda$ is nonnegatively graded, and $I\Lambda$ is the part with positive gradings.

The equivariant quantum product
\begin{equation} \label{eq:equivariant-quantum-product}
\ast_{q}^T: H^*_T(M;\Lambda) \otimes H^*_T(M;\Lambda) \longrightarrow H^*_T(M;\Lambda) 
\end{equation}
counts curves in any class $A \in H_2(M)$ with $q_1^{A \cdot D_1} \cdots q_r^{A \cdot D_r}$. Note that $(A \cdot D_1) + \cdots (A \cdot D_r) = A \cdot c_1(TM) \geq 0$, since our variety is Fano. Equivalently, the weight can be written as $q^{\vec{j}(A)}$, where $\vec{j}: H_2(M) \rightarrow H_2^T(M) = \bZ^r$ is the canonical map. We add a loop rotation parameter $u$ (corresponding to the trivial action of $S^1$ on $M$), and define the equivariant quantum connection as the collection of mutually commuting operators
\begin{equation} \label{eq:equivariant-quantum-connection}
\begin{aligned}
& \nabla_{uq_i\partial_{q_i}}^T: H^*_T(M;\Lambda[u]) \longrightarrow H^*_T(M;\Lambda[u]), \\
& \nabla_{uq_i\partial_{q_i}}^T(x) = uq_i\partial_{q_i} x + [D_i]_T \ast_{q}^T x.
\end{aligned}
\end{equation}
For the grading, we set $\mathrm{deg}(u) = 2$ as well, so that 
\begin{align} \label{eq:degreeQHeq} 
\deg(\alpha u^k q^{\vec{v}}) = \deg(\alpha)+2k+2|\vec{v}|. 
\end{align} 
Then, \eqref{eq:equivariant-quantum-product} preserves the grading, and \eqref{eq:equivariant-quantum-connection} has degree $2$.

For any $\mathbf{n} \in N$ there is a shift operator on the $(T \times S^1)$-equivariant quantum cohomology of $M$. We consider these in two slightly different versions:
\begin{align} 
\label{eq:latticeshift} 
& S_\mathbf{n}: H^*_T(M;\Lambda[u]) \longrightarrow H^*_T(M;\Lambda[u]), \\
\label{eq:tildeshift}
& \tilde{S}_{\mathbf{n}}: H^*_T(M;\bQ[q_1^{\pm 1},\dots,q_r^{\pm 1},u]) \longrightarrow
H^*_T(M; \bQ[q_1^{\pm 1},\dots,q_r^{\pm 1},u]),
\end{align}
related by a monomial rescaling given by \eqref{eq:vec-psi}:
\begin{equation} \label{eq:monomial-rescaling}
S_\mathbf{n} = q^{\vec{\psi}(\mathbf{n})} \tilde{S}_\mathbf{n} \,|_{H^*_T(M;\Lambda[u])}.
\end{equation}
The construction of shift operators involves a fibration
\begin{equation} \label{eq:e-fibration}
E_\mathbf{n} \longrightarrow \bC P^1 
\end{equation}
with fibre $M$, which carries an action of $T \times S^1$. One considers classes $A \in H_2(E_{\bfn};\bZ)$ which have intersection number $1$ with the fibre. For \eqref{eq:tildeshift}, the equivariant count of holomorphic curves in \eqref{eq:e-fibration} is weighted, as in \eqref{eq:equivariant-quantum-product}, by powers of $q$ corresponding to the intersection number of $A$ with divisors in $E_{\bfn}$ which are fibrewise versions of the $D_i$. To express this in a more invariant way, consider the Borel space $M_T = M \times_T ET \rightarrow BT$. A cocharacter $\mathbf{n} \in N$ defines a map
\begin{align} \label{eq:inclusioncorrcharacter} 
\mathbb{C}P^1 \hookrightarrow \bC P^\infty = BS^1 \longrightarrow BT.
\end{align}
Topologically, $E_\mathbf{n}$ is the pullback of $M_T$ along \eqref{eq:inclusioncorrcharacter}. From the resulting map $E_{\mathbf{n}} \rightarrow M_T$ we get $\vec{j}_{\bfn}: H_2(E_{\mathbf{n}};\bZ) \rightarrow H_2(M_T;\bZ) = H_2^T(M;\bZ) = \bZ^r$, and the weighting of class $A$ in \eqref{eq:tildeshift} is given by $q^{\vec{j}_{\bfn}(A)}$. It follows from \cite[Lemmas 2.1, 3.3]{iritani17} that the nonzero terms have coefficients in $q^{-\vec{\psi}(\mathbf{n})}\Lambda[u]$; hence \eqref{eq:monomial-rescaling} yields an operator $S_{\mathbf{n}}$ which preserves the subspace $H^*_T(M;\Lambda[u])$.

\begin{remark} \label{rem:nobulkdeformedshift}
Our conventions for the shift operators follow  \cite[\S 2.3]{iritani17} and \cite[\S 3]{iritani2017b} with only minor changes. We use a slightly different coefficient ring, and work in less generality: \cite{iritani17} allows for bulk-deformations with respect to classes $\tau \in H_T^*(M)$, while our shift operators correspond to the case where those bulk-parameters are set to zero.
\end{remark}

\begin{scholium}\label{thm:shiftproperties}
(i) The endomorphisms $\tilde{S}_\mathbf{n}$ are $\Lambda[u]$-linear. For any $\xi \in H^2_T(\mathit{point};\mathbb{Z})$, 
\begin{align} 
\label{eq:shiftu} 
\xi \smile \tilde{S}_\mathbf{n}(x)= \tilde{S}_\mathbf{n}\big(\xi \smile x + (\xi \cdot \mathbf{n})u x\big), 
\end{align}
where the pairing $\xi \cdot \mathbf{n} \in \bZ$ uses the identification $H_2^T(\mathit{point};\mathbb{Z}) = \bZ^n$. The same then holds for $S_{\mathbf{n}}$. These properties are straightforward consequences of the definitions \cite[Definition 2.2]{iritani17}.

(ii) For the trivial $\bfn = \mathbf{0}$, we have $\tilde{S}_{\mathbf{0}} = S_{\mathbf{0}} = \mathit{id}$.
We have
\begin{align} \label{eq:shiftcompositionrule}
& \tilde{S}_{\bfm} \circ \tilde{S}_{\bfn} = \tilde{S}_{\bfm + \bfn}, \\
& S_{\bfm} \circ S_{\bfn} = q^{\vec{\psi}(\bfm) + \vec{\psi}(\bfn) - \vec{\psi}(\bfm+\bfn)} S_{\bfm+\bfn}.
\end{align}
After adjusting to our coefficient rings, this is \cite[Cor.~3.4]{iritani17}, specialized to the toric case \cite[Prop.~2.4 (3)]{iritani17}. Note that $q^{\vec{\psi}(\bfm) + \vec{\psi}(\bfn) - \vec{\psi}(\bfm+\bfn)} \in \Lambda$, since $|\vec{\psi}| = w$ is convex (Lemma \ref{th:toriccombinatorics}). 

(iii) If we reduce the shift operators by setting $u = 0$, then
\begin{align} \label{eq:shiftmoduproduct}
& \tilde{S}_{\bfm}(1)_{u=0} \ast^T_{q} \tilde{S}_{\bfn}(1)_{u=0} = \tilde{S}_{\bfm+\bfn}(1)_{u=0}, \\
& \label{eq:shiftmoduproduct-2}
S_{\bfm}(1)_{u=0} \ast^T_{q} S_{\mathbf{n}}(1)_{u=0} = q^{\vec{\psi}(\bfm) + \vec{\psi}(\bfn) - \vec{\psi}(\bfm+\bfn)}
S_{\bfm+\bfn}(1)_{u=0}.
\end{align}

(iv) We have 
\begin{align} \label{eq:mcdufftolman} 
& \tilde{S}_{\bfe_i}(1) = q_i^{-1} [D_i]_{T}, \\
& S_{\bfe_i}(1)= [D_i]_{T}. 
\end{align}
This is \cite[Eq.~5.18]{liebenschutz-jones21} (building on \cite[Theorem 1.9(iii), Example 5.3]{mcDuff-tolman06}), again after adjusting to our conventions.

(v) With respect to the gradings \eqref{eq:degreeQHeq}, $\tilde{S}_{\bfn}$ is of degree $0$, and correspondingly $S_\mathbf{n}$ is of degree $2w(\mathbf{n})$. One can derive this from the counterpart of \eqref{eq:firstchern} for the fibrewise tangent bundle of \eqref{eq:e-fibration}.
\end{scholium}

\begin{lemma} \label{th:leadingorder} 
We have
\begin{equation}
S_{\mathbf{n}}(1)_{u = 0} = \alpha_{\mathbf{n}}\;\; \mymod I\Lambda,
\end{equation}
where $\alpha_\mathbf{n}$ is the cohomology class \eqref{eq:srbasis}. 
\end{lemma}

\begin{proof} 
Apply \eqref{eq:shiftmoduproduct-2} to $\bfn$ written as in \eqref{eq:nonnegative-linear}. Since all the $\bfe_i$ with $\psi_i(\bfn) > 0$ lie in a single cone, the $q$-exponent in that formula is zero in this case, and
\begin{align} 
S_\mathbf{n}(1)_{u=0}= (S_{\bfe_i}(1)_{u=0})^{\ast_{q}^T \psi_1(\bfn)} \ast_{q}^T \cdots \ast_{q}^T
(S_{\bfe_r}(1)_{u=0})^{\ast_{q}^T \psi_r(\bfn)}.
\end{align}
The quantum product agrees with the cup product mod $I\Lambda$; the result follows from this and Scholium \ref{thm:shiftproperties}(iv).
\end{proof} 

\begin{lemma}\label{lemma:intertwining}
The operators \eqref{eq:tildeshift} commute with the equivariant quantum connection: 
\begin{align}\label{eqn:commute} 
\nabla_{uq_i\partial_{q_i}}^T \circ \tilde{S}_\mathbf{n} = \tilde{S}_\mathbf{n} \circ \nabla_{uq_i\partial_{q_i}}^T.
\end{align}
\end{lemma}

\begin{proof}
This is a variant of \cite[Proposition 2.4(2)]{iritani17}, and can be deduced from it using the divisor equation for shift operators \cite[Remark 3.14]{iritani17}. As mentioned in Remark \ref{rem:nobulkdeformedshift}, \cite{iritani17} includes bulk deformations of the quantum product and shift operators. For our purpose, it suffices to consider bulk parameters
$(\tau_1,\dots,\tau_r)$ corresponding to the standard generators of $H^2_T(M)$. Both $\ast^T_{q}$ and $S_{\mathbf{n}}$ will now depend on these parameters, without specifically indicating this in the notation. By \cite[Proposition 2.4(2)]{iritani17}, the operation $\nabla_{u\partial_{\tau_i}} = u\partial_{\tau_i} + [D_i]_T \ast^T_{q} \cdot$ satisfies
\begin{align} \label{eq:iritanicommuting} 
\nabla_{u\partial_{\tau_i}} \circ S_{\mathbf{n}} = S_{\mathbf{n}} \circ \nabla_{u\partial_{\tau_i}}. 
\end{align} 
By the divisor equation for shift operators \cite[Remark 3.14]{iritani17}, 
\begin{equation} \label{eq:shiftdivisorequation}
q_i \partial_{q_i} (S_{\mathbf{n}}) = (\partial_{\tau_i} + \psi_i(\mathbf{n})) (S_{\mathbf{n}}).
\end{equation}
Combining \eqref{eq:iritanicommuting} and \eqref{eq:shiftdivisorequation}, we get \eqref{eqn:commute} for general bulk parameters; specializing to $\tau_i = 0$ gives the version we wanted.
\end{proof}

\subsection{Toric mirror symmetry\label{sect:toricms}}
Let
\begin{equation}
\label{eq:positiveLaurentring} 
R_{\Lambda} =\Big\{\sum_{\substack{(\vec{v},\mathbf{n}) \\ |\vec{v}| \geq w(\mathbf{n})}}a_{\vec{v},\mathbf{n}}q^{\vec{v}}z^{\mathbf{n}} \Big\} \subset \bQ[q_i^{\pm 1}, z_j^{\pm 1}].
\end{equation}
As a $\Lambda$-module, this is freely generated by $q^{\vec\psi(\bfn)}z^\bfn$, $\bfn \in N$. The mirror superpotential is 
\begin{equation} \label{eq:q-superpotential}
W_{q} = \sum_{i=1}^r  q_i z^{\bfe_i} \in R_\Lambda.
\end{equation}
Similarly, we consider the space $\Omega^k_{\Lambda}$ of differential $k$-forms in the variables $(z_1,\dots,z_n)$; this is the free $R_\Lambda$-module with basis 
\begin{equation} \label{eq:differential-form-basis}
(dz_{i_1}/z_{i_1}) \wedge \cdots \wedge (dz_{i_k}/z_{i_k}),\quad i_1 < \cdots < i_k.
\end{equation}
In particular, write
\begin{equation} \label{eq:vol}
\mathit{vol} = \frac{dz_1}{z_1} \wedge \cdots \wedge \frac{dz_n}{z_n} \in \Omega^n_\Lambda.
\end{equation}
We re-insert the parameter $u$, and consider $\Omega^*_\Lambda[u]$ with differential
\begin{equation}
ud + dW_{q}: \Omega^*_\Lambda[u] \longrightarrow \Omega^{*+1}_{\Lambda}[u],
\end{equation}
where $dW_q$ acts by wedge product on the left; and with the connection
\begin{equation} \label{eq:u-mirror-connection}
\nabla_{uq_i\partial_{q_i}} = uq_i\partial_{q_i} + (q_i\partial_{q_i}W_{q}) =
uq_i\partial_{q_i} + q_i z^{\bfe_i}:
\Omega^*_{\Lambda}[u] \longrightarrow \Omega^*_{\Lambda}[u].
\end{equation}
Consider the $\Lambda[u]$-linear map
\begin{equation} \label{eq:thetamap} 
\begin{aligned}
& \Theta_u: \Omega_{\Lambda}^n[u] = R_\Lambda \, vol \longrightarrow H^*_{T}(M;\Lambda[u]), \\
& \Theta_u(q^{\vec\psi(\mathbf{n})}z^{\mathbf{n}} \,\mathit{vol})= S_\mathbf{n}(1). 
\end{aligned}
\end{equation}
The following is a simplified version of \cite[Theorem 3.21]{iritani17}:

\begin{lemma} \label{lem:shift1} 
(i) $\Theta_u$ is an isomorphism.

(ii) $\Theta_u(\mathit{vol})=1$. 

(iii) $\Theta_u$ intertwines connections: 
\begin{equation} 
\Theta_u \circ \nabla_{uq_i\partial_{q_i}} = \nabla_{uq_i\partial_{q_i}}^T \circ \Theta_u. \end{equation} 
\end{lemma}


\begin{proof}
(i) Take the increasing $\Lambda[u]$-linear filtrations
\begin{align}
\label{eq:domain-f} & F_i \Omega_{\Lambda}^n[u] \subset \Omega_{\Lambda}^n[u], \\
\label{eq:range-f} & F_i H^*_T(M;\Lambda[u]) \subset H^*_T(M;\Lambda[u])
\end{align}
where: \eqref{eq:domain-f} is generated by $q^{\vec{\psi}(\bfn)} z^{\mathbf{n}}\, \mathit{vol}$ with $w(\bfn) \leq i$; and \eqref{eq:range-f} is generated by $H^{\leq 2i}_T(M;\bZ)$. When grading $q_i$ and $u$ as in \eqref{eq:degreeQHeq}, any element of $u\Lambda[u]$ or $I\Lambda[u]$ has strictly positive degree. Therefore, it follows from Scholium \ref{thm:shiftproperties}(v) that $\Theta_u$ preserves the filtrations defined above, and from Lemma \ref{th:leadingorder} that the induced map on the graded spaces of the filtration is an isomorphism. This implies the desired result.

(ii) By Scholium \ref{thm:shiftproperties}(ii), $\Theta_u(\mathit{vol}) = S_{\mathbf{0}}(1) = 1$. 

(iii) For simplicity, let's enlarge both domain and range to
\begin{equation}
\begin{aligned}
& \Theta_u: \bQ[q_i^{\pm 1},z_j^{\pm 1}]\, \mathit{vol} \longrightarrow H^*_T(M;\bQ[q_i^{\pm 1},u]), 
\\
& \Theta_u(z^{\mathbf{n}}\,\mathit{vol}) = q^{-\vec{\psi}(\bfn)} \Theta_u( q^{\vec{\psi}(\bfn)} z^{\bfn}\, \mathit{vol}) =
q^{-\vec{\psi}(\bfn)} S_{\bfn}(1) = \tilde{S}_{\bfn}(1).
\end{aligned}
\end{equation}
By the definition of \eqref{eq:u-mirror-connection} and Scholium \ref{thm:shiftproperties}(ii),
\begin{equation} \label{eq:ha}
\Theta_u\big( \nabla_{uq_i\partial_{q_i}}(z^{\mathbf n} \,\mathit{vol})\big) = \Theta_u(q_i  z^{\mathbf{n} + \bfe_i} \,\mathit{vol}) = \tilde{S}_{\bfn+\bfe_i}(q_i\,1) = \tilde{S}_{\bfn}(q_i \tilde{S}_{\bfe_i}(1)).
\end{equation}
By Lemma \ref{lemma:intertwining},
\begin{equation} \label{eq:haha}
\nabla_{uq_i\partial_{q_i}}^T \Theta_u(z^\mathbf{n} \, \mathit{vol}) = 
\nabla_{uq_i\partial_{q_i}}^T \tilde{S}_\mathbf{n}(1) =
\tilde{S}_{\mathbf{n}}(\nabla_{uq_i \partial_{q_i}}^T(1)) = \tilde{S}_{\bfn}([D_i]_T).
\end{equation}
Scholium \ref{thm:shiftproperties}(iv) shows that the right hand sides of \eqref{eq:ha} and \eqref{eq:haha} are equal.
\end{proof}

\begin{lemma} \label{lem:lambdaction}
Take the standard basis of $\Omega^{n-1}_\Lambda$ over $R_\Lambda$,
\begin{equation}\label{eq:standardbasisomeganminusone}
\omega_j = (-1)^{j-1} \frac{dz_1}{z_1} \wedge \cdots \wedge \widehat{\frac{dz_j}{z_j}} \wedge \cdots \wedge \frac{dz_n}{z_n}.
\end{equation}
For any $j$ and $\bfn \in N$, the class \eqref{eq:lambdak} satisfies
\begin{equation}  
\Theta_u \big( (ud+dW_{q})(q^{\vec{\psi}(\bfn)} z^\bfn \omega_j)\big) = \lambda_j \smile \Theta_u(q^{\vec{\psi}(\bfn)} z^\bfn \,\mathit{vol}).
\end{equation}
\end{lemma}

\begin{proof} 
We enlarge coefficients as in the proof of Lemma \ref{lem:shift1}(iii), to simplify the computation. With that in mind,
\begin{equation}
(ud+dW_{q})(z^{\bfn} \omega_j) = \Big( u n_j z^{\bfn} + \sum_{i=1}^r q_i e_{i,j} z^{\bfn+\bfe_i} \Big) \,\mathit{vol}.
\end{equation}
Using Scholium \eqref{thm:shiftproperties}(ii),(iv) we can write
\begin{equation} \label{eq:d1}
\begin{aligned}
& \Theta_u (ud+dW_{q})(z^{\bfn} \omega_j)  = u n_j\, \tilde{S}_{\mathbf{n}}(1) + \sum_{i=1}^r q_i e_{i,j}\, \tilde{S}_{\bfn+\bfe_i}(1) 
\\[-.5em] & \qquad = u n_j\, \tilde{S}_{\mathbf{n}}(1) 
+ \sum_{i=1}^r q_i e_{i,j}\,\tilde{S}_{\mathbf{n}}(\tilde{S}_{\bfe_i}(1))
    = \tilde{S}_{\mathbf{n}}(u n_j\,1 + \sum_{i=1}^r e_{i,j} [D_i]_T) = \tilde{S}_{\bfn}(u n_j\,1 + \lambda_j).
\end{aligned}
\end{equation}
On the other hand, by Scholium \ref{thm:shiftproperties}(i) with $\xi = \lambda_j$,
\begin{equation} \label{eq:d2}
\lambda_j \smile \Theta_u(z^{\bfn} \mathit{vol}) = \lambda_j \smile \tilde{S}_{\bfn}(1) =
\tilde{S}_{\bfn}(\lambda_j + u n_j\, 1) .
\end{equation}
Multiplying both \eqref{eq:d1} and \eqref{eq:d2} by $q^{\vec{\psi}(\bfn)}$ yields the desired result.
\end{proof}

As a corollary of Lemmas \ref{lem:shift1} and \ref{lem:lambdaction}, $\Theta_u$ induces an isomorphism 
\begin{equation} \label{eq:theta-induced}
H^n(\Omega^*_\Lambda[u], ud + dW_{q}) \stackrel{\iso}{\longrightarrow} 
H^*_T(M;\Lambda[u]) /(\lambda_1,\dots,\lambda_n) H^*_T(M;\Lambda[u]) \iso H^*(M;\Lambda[u]).
\end{equation}
Here, $H^n$ is the $n$-cohomology for the standard grading of differential forms. The second isomorphism in \eqref{eq:theta-induced} is the $u$-linear extension of \eqref{eq:non-equivariant-limit}. Because that isomorphism is defined by setting the equivariant parameters $\lambda_i$ to $0$, the map \eqref{eq:theta-induced} intertwines \eqref{eq:u-mirror-connection} with the non-equivariant quantum connection,
\begin{equation} \label{eq:u-quantum-connection}
\begin{aligned}
& \nabla_{uq_i\partial_{q_i}}: H^*(M;\Lambda[u]) \longrightarrow H^*(M;\Lambda[u]), \\
& \nabla_{uq_i\partial_{q_i}}(x) = uq_i\partial_{q_i} x + [D_i] \ast_{q} x.
\end{aligned}
\end{equation}
As a concluding step, we further modify the framework by seting $u = 1$. This means that we consider $\Omega^*_\Lambda$ with the following differential and connection:
\begin{align}
& d+dW_{q}: \Omega^*_\Lambda \longrightarrow \Omega^{*+1}_{\Lambda}, \\
\label{eq:GMconnections}
& \nabla_{q_i\partial_{q_i}} = q_i\partial_{q_i} + (q_i\partial_{q_i}W_q): \Omega^*_\Lambda \longrightarrow \Omega^*_{\Lambda}.
\end{align}
On the quantum cohomology side, we set $u = 1$ in \eqref{eq:u-quantum-connection}, which yields
\begin{equation} \label{eq:multiquantum-connection}
\begin{aligned}
& \nabla_{q_i\partial_{q_i}}: H^*(M;\Lambda) \longrightarrow H^*(M;\Lambda), \\
& \nabla_{q_i\partial_{q_i}}(x) = q_i\partial_{q_i} x + [D_i] \ast_{q} x.
\end{aligned}
\end{equation}

\begin{theorem} \label{thm:toricmirror} 
There is a unique $\Lambda$-linear isomorphism
\begin{equation} 
\label{eq:mirrornoneqiso} 
\Theta: H^n(\Omega^*_\Lambda,d+dW_q) \stackrel{\iso}{\longrightarrow} H^*(M;\Lambda),
\end{equation} 
such that:
\begin{align}
& \Theta([\mathit{vol}]) = 1; \\
& \Theta \text{ intertwines the connections \eqref{eq:GMconnections} and \eqref{eq:multiquantum-connection}}.
\end{align}
\end{theorem}

\begin{proof}
If we take $\Theta$ to be the $u = 1$ specialization of \eqref{eq:theta-induced}, then the discussion so far, going back to Lemma \ref{lem:shift1}, shows that it has the desired properties. For uniqueness, note that the conditions stated above determine the image under $\Theta^{-1}$ of any class
$\nabla_{q_{i_1}\partial_{q_{i_1}}} \cdots \nabla_{q_{i_k}\partial_{q_{i_k}}}(1)$. By definition of the quantum connection,
\begin{equation} \label{eq:cup-d}
[D_{i_1}] \smile \cdots \smile [D_{i_k}] = \nabla_{q_{i_1}\partial_{q_{i_1}}} \cdots \nabla_{q_{i_k}\partial_{q_{i_k}}}(1) \;\; \text{ mod } I\Lambda.
\end{equation}
The left hand side has degree $2k$, and the right hand side degree $\leq 2k$, in our usual grading of $H^*(M;\Lambda)$. Because $I\Lambda$ is positively graded, the difference between the two sides must lie in $H^{<2k}(M;\bQ) \otimes I\Lambda$. This shows that the classes $\nabla_{q_{i_1}\partial_{q_{i_1}}} \cdots \nabla_{q_{i_k}\partial_{q_{i_k}}}(1)$ generate $H^*(M;\Lambda)$; hence $\Theta^{-1}$ is completely determined.
\end{proof}

We need one more concrete consequence:

\begin{corollary} \label{cor:generatingalgebraicGM}
Let $\Delta \subset N$ be a $\bQ$-basis set (Definition \ref{def:f-basis}). Then, the classes
\begin{equation}
[q^{\vec{\psi}(\bfn)} z^{\bfn} \, \mathit{vol}] \in H^n(\Omega^*_\Lambda,d+dW_q), \quad \bfn \in \Delta,
\end{equation}
freely generate that cohomology as a $\Lambda$-module.
\end{corollary}

\begin{proof}
By definition and Lemma \ref{th:leadingorder},
\begin{equation}
\Theta_u( q^{\vec{\psi}(\bfn)} z^{\bfn} \mathit{vol}) = S_{\bfn}(1) = \alpha_{\bfn} \text{ mod } (I\Lambda, u\Lambda),
\end{equation}
As in the proof of Lemma \ref{lem:shift1}(i), it follows that the classes $(S_{\bfn}(1))_{\bfn \in \Delta}$ freely generate $H^*(M;\Lambda[u])$ over $\Lambda[u]$. Since \eqref{eq:theta-induced} is an isomorphism, it follows that the cohomology classes of $(q^{\vec{\psi}(\bfn)} z^{\bfn}\, \mathit{vol})_{\bfn \in \Delta}$ freely generate $H^n(\Omega^*_\Lambda[u], ud+dW_q)$ over $\Lambda[u]$. Setting $u = 1$ completes the argument.
\end{proof}

\section{A mod $p$ computation\label{sec:modp}}
In this section, we discuss a mod $p$ version of the mirror symmetry result, which will be used in Section \ref{sec:dwork} as a building block for the computation of cohomology with $p$-adic coefficients. The setup here may, at first, look like a straightforward characteristic $p$ version of Section \ref{sect:toricms}. However, we are actually interested in the cohomology of a truncated differential (the associated graded with respect to a suitable filtration), and that requires a different approach. 

Start with the mod $p$ analogues of \eqref{eq:lambdageq0} and \eqref{eq:positiveLaurentring},
\begin{align} \label{eq:bar-lambda}
& \bar\Lambda = \Big\{ \sum_{|\vec{v}| \geq 0} a_{\vec{v}} q^{\vec{v}} \Big\} \subset \bF_p[q_i^{\pm 1}],
\\
\label{eq:bar-r}
&
R_{\bar\Lambda} =\Big\{\sum_{\substack{(\vec{v},\mathbf{n}) \\ |\vec{v}| \geq w(\mathbf{n})}}a_{\vec{v},\mathbf{n}}q^{\vec{v}}z^{\mathbf{n}} \Big\} \subset \bF_p[q_i^{\pm 1}, z_j^{\pm 1}],
\end{align}
as well as the corresponding spaces $\Omega^*_{\bar\Lambda}$ of differential forms. $R_{\bar\Lambda}$ and its generalization $\Omega^*_{\bar\Lambda}$ come with (decreasing, bounded above, non-complete) filtrations: $F^sR_{\bar\Lambda}$ is defined by the condition $|\vec{v}| \geq w(\mathbf{n}) + s$ in \eqref{eq:bar-r}; and that is extended to differential forms using the basis \eqref{eq:differential-form-basis}. These filtrations are compatible with wedge product of differential forms, by Lemma \ref{th:toriccombinatorics}(i), and also with the de Rham differential. We denote the associated graded spaces by $GR_{\bar\Lambda}$ and $G\Omega^*_{\bar\Lambda}$; they are still modules over $\bar\Lambda$. Explicitly, because of Lemma \ref{th:toriccombinatorics}(ii), multiplication in $GR_{\bar\Lambda}$ is given by
\begin{equation} \label{eq:new-sr}
q^{\vec{v}} z^{\mathbf{m}} \cdot q^{\vec{w}} z^{\mathbf{n}} =
\begin{cases}
q^{\vec{v}+\vec{w}} z^{\mathbf{m}+\mathbf{n}} & \text{if $\mathbf{m}$ and $\mathbf{n}$ lie in the same cone of the fan,} \\
0 & \text{otherwise.} 
\end{cases}
\end{equation}
By comparing \eqref{eq:new-sr} with \eqref{eq:sr}, we see that the map
\begin{equation} \label{eq:cohomology-coincidence}
\begin{aligned}
& GR_{\bar\Lambda} \longrightarrow H^*_T(M;\bar\Lambda), \\
& q^{\vec{\psi}(\bfn)}z^{\bfn} \longmapsto \alpha_{\bfn}
\end{aligned}
\end{equation}
is an isomorphism of rings. The superpotential is the mod $p$ version of \eqref{eq:q-superpotential}, for which we use the same notation. Write 
\begin{equation} 
\bar\lambda_k = \sum_{i=1}^r e_{i,k} q_i z^{\bfe_i} \in R_{\bar\Lambda}, \quad 1 \leq k \leq n,
\end{equation}
where $e_{i,k}$ is the $k$-th component of $\bfe_i$, as in \eqref{eq:lambdak}. Then,
\begin{equation} \label{eq:d-bar-w}
dW_q = \sum_{k=1}^n \bar\lambda_k \frac{dz_k}{z_k}.
\end{equation}
From now on, we will consider the $\bar\lambda_k$ as elements of $GR_{\bar\Lambda}$. Under the isomorphism \eqref{eq:cohomology-coincidence}, they map to the generators of $H^*_T(\mathit{point};\bar\Lambda)$, as described in \eqref{eq:lambdak}. The following is essentially \cite[Theorem 2.14, Theorem 2.17]{adolphson-sperber89} (the complex in question is isomorphic to the Koszul complex studied there, tensored with $\bar\Lambda$ over $\bF_p$); we give a more geometric argument.

\begin{lemma} \label{th:g-cohomology}
(i) We have
\begin{equation}
H^i(G\Omega^*_{\bar\Lambda}, dW_q) = 0 \quad \text{for } i<n.
\end{equation}

(ii) The map \eqref{eq:cohomology-coincidence} descends to an isomorphism
\begin{equation} \label{eq:hn}
H^n(G\Omega^*_{\bar\Lambda}, dW_q) \iso H^*(M;\bar\Lambda).
\end{equation}
\end{lemma}

\begin{proof} 
Our complex can be written as
\begin{equation} \label{eq:koszul}
0 \rightarrow GR_{\bar\Lambda} \rightarrow GR_{\bar\Lambda}^{\oplus n} \rightarrow GR_{\bar\Lambda}^{\oplus \left(\begin{smallmatrix} n \\ 2\end{smallmatrix}\right)} \rightarrow \cdots \rightarrow GR_{\bar\Lambda} \rightarrow 0.
\end{equation}
From \eqref{eq:d-bar-w}, one sees that the differentials are built from multiplication with the $\bar\lambda_k$. This is in fact the Koszul complex associated to $(\bar\lambda_1,\dots,\bar\lambda_n) \in GR_{\bar\Lambda}$. We know from \eqref{eq:equi-free} that $H^*_T(M;\bar\Lambda)$ is a free module over $H^*_T(\mathit{point};\bar\Lambda)$. Transferring this via \eqref{eq:cohomology-coincidence} shows that $GR_{\bar\Lambda}$ is a free module over $\bar\Lambda[\bar\lambda_1,\dots,\bar\lambda_n]$. Hence, the $\bar\lambda_k$ form a regular sequence. This implies that the only nonzero cohomology group in \eqref{eq:koszul} is the last one, which is statement (i); and that nonzero group is equal to $GR_{\bar\Lambda}/(\bar\lambda_1,\dots,\bar\lambda_n)GR_{\bar\Lambda}$, which by comparison with \eqref{eq:non-equivariant-limit} yields (ii).
\end{proof}

Next, following Section \ref{sect:toricms}, we introduce an auxiliary parameter $u$ and the de Rham differential:

\begin{lemma} \label{th:truncated-equivariant}
(i) We have
\begin{equation}
H^i(G\Omega^*_{\bar\Lambda}[u], ud+dW_q) = 0 \quad \text{for } i<n.
\end{equation}

(ii) The map 
\begin{equation}
H^n(G\Omega^*_{\bar\Lambda}[u], ud+dW_q) \longrightarrow H^n(G\Omega^*_{\bar\Lambda},dW_q)
\end{equation}
obtained by setting $u = 0$ is onto; and after choosing a $\bar\Lambda$-linear splitting of it, one gets a (non-canonical) $\bar\Lambda[u]$-linear isomorphism
\begin{equation} \label{eq:u-splitting}
H^n(G\Omega^*_{\bar\Lambda},dW_q)[u] \stackrel{\iso}{\longrightarrow} H^n(G\Omega^*_{\bar\Lambda}[u], ud+dW_q).
\end{equation}
\end{lemma}

\begin{proof}
Let's temporarily change the grading on $G\Omega^*_{\bar\Lambda}[u]$ by setting
\begin{equation}
\deg(u^l q^{\vec{v}} z^{\mathbf{n}}\frac{dz_{j_{1}}}{z_{j_{1}}}\wedge \frac{dz_{j_{2}}}{z_{j_{2}}} \wedge \cdots \wedge \frac{dz_{j_k}}{z_{j_k}})= 2|\vec{v}|+2l-k,
\end{equation}
which is compatible with our differential. For this alternative grading, the $u$-adic filtration is bounded in a given degree (only finitely many $u$-powers can arise); hence, the associated spectral sequence is convergent. The starting page is nonzero only in degrees congruent to $n$ mod $2$, by Lemma \ref{th:g-cohomology}, and hence the spectral sequence must degenerate. The desired statements follow from this by a standard argument.
\end{proof}

\begin{corollary} \label{th:less-than-n}
$H^i(G\Omega^*_{\bar\Lambda},dW_q) = 0$ for $i<n$.
\end{corollary}

\begin{proof}
This follows from Lemma \ref{th:truncated-equivariant} by setting $u = 1$.
\end{proof}

The next result is our version of \cite[Theorem 2.18]{adolphson-sperber89}. Formally, it is analogous to Corollary \ref{cor:generatingalgebraicGM}:

\begin{corollary} \label{cor:generatingalgebraicGM2}
Let $\Delta \subset N$ be an $\bF_p$-basis set (Definition \ref{def:f-basis}). Then: 

(i) The classes
\begin{equation} \label{eq:bar-lambda-generator}
[q^{\vec{\psi}(\bfn)} z^{\bfn} \, \mathit{vol}] \in H^n(G\Omega^*_{\bar\Lambda},dW_q), \quad
\bfn \in \Delta
\end{equation}
freely generate that cohomology group as a $\bar\Lambda$-module.

(ii) The same remains true if we replace the differential in \eqref{eq:bar-lambda-generator} with $d+dW_q$.
\end{corollary}

\begin{proof}
(i) By definition, the isomorphism \eqref{eq:hn} maps $[q^{\vec{\psi}(\bfn)}z^{\bfn}\,\mathit{vol}] \in H^n(G\Omega^*,dW_q)$ to the non-equivariant image of $\alpha_{\bfn}$. Hence, our assumption implies the statement in question.

(ii) By Lemma \ref{th:truncated-equivariant}, the same monomials freely generate $H^*(G\Omega_{\bar\Lambda}^*[u], ud+dW_q)$ over $\bar\Lambda[u]$. Finally, 
\begin{equation}
H^*(G\Omega_{\bar\Lambda}^*,d+dW_q) \iso H^*(G\Omega_{\bar\Lambda}^*[u], ud+dW_q)/(u-1),
\end{equation}
and so the desired result follows.
\end{proof}

\section{Dwork's inverse Frobenius\label{sec:dwork}}
We will construct a version of Dwork's inverse Frobenius from \cite{dwork74} for the quantum connection on Fano toric varieties. The analytic framework in the present section is not sufficient to prove overconvergence, but it easy to define and to do computations in.

\begin{convention}
This section and the following one use parameters $(t_1,\dots,t_r)$. These will ultimately be related to the previous $(q_1,\dots,q_r)$ by rescaling as in \eqref{eq:t-q}, see \eqref{eq:iota} below. We write $t^{\vec{v}} = t_1^{v_1} \cdots t_r^{v_r}$ as in \eqref{eq:multi-q}.
\end{convention}

\subsection{Basic definitions}
The first step is to introduce appropriate spaces of $p$-adic functions and differential forms, motivated by ones in \cite[\S 1]{adolphson-sperber89}. Recall that $K = \bQ_p(\mu)$; $O \subset K$ is the ring of integers; and $\pi$ is the constant \eqref{eq:pi}. 

\begin{definition} \label{th:b}
(i) Let $B$ be the ring of series 
\begin{equation} \label{eq:define-b}
\sum_{|\vec{v}| \geq 0} 
a_{\vec{v}} t^{\vec{v}}, \quad a_{\vec{v}} \in K,
\end{equation}
such that:
\begin{equation} \label{eq:b-condition}
\parbox{36em}{for any $C \in \bR$, there are only finitely many $\vec{v}$ with $\mathrm{val}(a_{\vec{v}}) \leq C$.}
\end{equation}
(One could write it as $B = K\langle t_1, (t_1^{-1}t_2)^{\pm 1},\dots\rangle$; in the language of nonarchimedean analytic geometry, it is an affinoid algebra.) Write $IB \subset B$ for the subspace where we only allow $|\vec{v}| > 0$. Write $OB \subset B$ for the subspace where $a_{\vec{v}} \in O$. Because the valuations appearing in any element of $B$ are bounded below, we have $B = OB \otimes_O K$.

(ii) Let $R_B$ be the space of series
\begin{equation} \label{eq:define-rb}
\sum_{\substack{
\vec{v}, \, \bfn \\ 
|\vec{v}| \geq w(\bfn)}}
a_{\vec{v},\bfn} \,\pi^{w(\bfn)} t^{\vec{v}} z^{\bfn}, \quad a_{\vec{v},\bfn} \in K,
\end{equation}
with the same kind of condition as before:
\begin{equation} \label{eq:valuationconditionB}
\parbox{35em}{for any $C \in \bR$, there are only finitely many $(\vec{v},\mathbf{n})$ with
$\mathrm{val}(a_{\vec{v},\bfn}) \leq C$.}
\end{equation}
Let $R_{OB} \subset R_B$ be the subspace where $a_{\vec{v},\bfn} \in O$. We again have $R_B = R_{OB} \otimes_O K$.

(iii) Let $\Omega^k_B$ be the space of $k$-forms in the variables $(z_1,\dots,z_n)$ with coefficients in $B$, which is the free $R_B$-module with basis \eqref{eq:differential-form-basis}. As before, we can write $\Omega^*_B = \Omega^*_{OB} \otimes_O K$.
\end{definition}

$R_B$ is a ring (and a $B$-algebra). To check that, take two elements written as in \eqref{eq:define-rb}, with coefficients $a^k_{\vec{v},\bfn}$, $k = 1,2$. The formal product, written in the same way, has coefficients
\begin{equation} \label{eq:prod-coeff}
        a_{\vec{v},\bfn} = \sum_{\substack{\vec{v}_1 + \vec{v}_2 = \vec{v} \\ \mathbf{n}_1 + \mathbf{n}_2 = \mathbf{n}} } a_{\vec{v}_1,\mathbf{n}_1}^1 a_{\vec{v}_2,\mathbf{n}_2}^2 \pi^{w(\mathbf{n}_1) + w(\mathbf{n}_2) - w(\mathbf{n})}.
\end{equation}
In order for this sum to be nonempty, we need $|\vec{v}_k| \geq w(\mathbf{n}_k)$ for both $k$, which implies $|\vec{v}| \geq w(\bfn)$ by Lemma \ref{th:toriccombinatorics}(i). For the same reason, the $\pi$-powers in \eqref{eq:prod-coeff} are all nonnegative. Let $C_{\operatorname{min}} \in \bR$ be the minimum of the valuations of $a_{\vec{v},\bfn}^k$ over all $(\vec{v},\bfn,k)$. Choose some $C \in \bR$. In any sum \eqref{eq:prod-coeff}, there are only finitely many $(\vec{v}_1,\bfn_1)$ where $\mathrm{val}(a_{\vec{v}_1,\bfn_1}^1) \leq C - C_{\mathrm{min}}$. For all others, we have 
\begin{equation} \label{eq:c-min-c}
\mathrm{val}\big(a_{\vec{v}_1,\mathbf{n}_1}^1 a_{\vec{v}_2,\mathbf{n}_2}^2 \pi^{w(\mathbf{n}_1) + w(\mathbf{n}_2) - w(\mathbf{n})}\big) \geq \mathrm{val}(a_{\vec{v}_1,\bfn_1}^1) +  
\mathrm{val}(a_{\vec{v}_2,\bfn_2}^2) \geq C;
\end{equation}
this implies that the series \eqref{eq:prod-coeff} is $p$-adically convergent. One can slightly refine the argument to show that the $a_{\vec{v},\bfn}$ satisfy \eqref{eq:valuationconditionB}. Namely, when looking at an expression \eqref{eq:prod-coeff}, there are two possibilities: either there are $(\vec{v}_1,\bfn_1,\vec{v}_2,\bfn_2)$ such that $\mathrm{val}(a_{\vec{v}_k,\bfn_k}^k) \leq C-C_{\operatorname{min}}$ for both $k$, which is the case for only finitely many $(\vec{v},\bfn)$; or else \eqref{eq:c-min-c} applies to all terms in the series, so that $\mathrm{val}(a_{\vec{v},\bfn}) \geq C$.


The superpotential is now
\begin{equation} \label{eq:t-superpotential}
W_t = \pi \sum_{i=1}^r t_i z^{\bfe_i} \in R_B.
\end{equation}
As before, we consider the differential and connection
\begin{align} \label{eq:d-t}
& d+dW_t : \Omega^*_B \longrightarrow \Omega^{*+1}_B, \\
\label{eq:t-connection}
& \nabla_{t_i\partial_{t_i}} = t_i\partial_{t_i} + (t_i\partial_{t_i}W_t): \Omega^*_B \longrightarrow \Omega^*_B.
\end{align}

We will also need Frobenius-twisted versions of the previous notions.

\begin{definition} \label{th:pb}
(i) Denote by $R_B^{(p)}$ the space of series
\begin{equation} \label{eq:bp-series}
\sum_{\substack{\vec{v},\,\bfn \\ |\vec{v}| \geq p\, w(\mathbf{n})}} a_{\vec{v},\mathbf{n}} \, \pi^{w(\mathbf{n})} t^{\vec{v}} z^\mathbf{n} 
\end{equation}
such that \eqref{eq:valuationconditionB} holds.

(ii) The spaces $\Omega_B^{(p),*}$ of differential forms are defined as before, but using $R_B^{(p)}$ as coefficients.
\end{definition}

Consider the Frobenius map on the base $B$,
\begin{equation} \label{eqn:t-p-power}
F(t_1, \dots, t_r) = (t_1^p, \dots, t_r^p),
\end{equation}
and the associated map (for which we use the same notation)
\begin{equation} \label{eqn:t-p-power-2}
F: R_B \longrightarrow R_B^{(p)}.
\end{equation}
If we introduce formal roots $t_i^{1/p}$, with $(t_i^{1/p})^p = t_i$, then \eqref{eqn:t-p-power-2} induces an isomorphism
\begin{equation} \label{eq:BpisFrobeniuspullback}
R_B[t_1^{1/p},\dots,t_r^{1/p}] \stackrel{\iso}{\longrightarrow} R_B^{(p)};
\end{equation}
the same holds for differential forms. In that sense, the spaces in Definition \ref{th:pb} are Frobenius pullbacks of those in Definition \ref{th:b}. Correspondingly, we use the pullback superpotential, differential, and connection:
\begin{align}
& W_{t^p} = \pi \sum_{i=1}^r t_i^p z^{\bfe_i} \in R_B^{(p)}, \\
& d + dW_{t^p}: \Omega^{(p),*}_B \longrightarrow \Omega^{(p),*+1}_B, \\
\label{eq:derham-frobenius-connection}
& \nabla_{t_i\partial_{t_i}}^{(p)} = t_i\partial_{t_i} + (t_i\partial_{t_i}W_{t^p})
: \Omega^{(p),*}_B \longrightarrow \Omega^{(p),*}_B.
\end{align}
The differential form version of \eqref{eqn:t-p-power-2} then satisfies
\begin{align}
& (d + dW_{t^p}) \circ F = F \circ (d+dW_t), \\
\label{eq:power}
& \nabla_{t_i\partial_{t_i}}^{(p)} \circ F = F \circ (p\nabla_{t_i\partial_{t_i}}), 
\end{align}
where the extra $p$ in \eqref{eq:power} comes from the effect of $t \mapsto t^p$ on the rotational vector field $t\partial_t$. 

\subsection{Functional analysis setup}
The standard Gauss norm on $B$ is
\begin{equation}
\| \sum_{\vec{v}} a_{\vec{v}} t^{\vec{v}} \| = \mathrm{max}_{\vec{v}}(|a_{\vec{v}}|) = 
p^{-\mathrm{min}_{\vec{v}}(\mathrm{val}(a_{\vec{v}}))}.
\end{equation}
We similarly equip $R_B$ with the norm
\begin{equation}\label{eqn:B-norm}
    \| \sum_{\vec{v},\mathbf{n}} a_{\vec{v},\bfn}\, \pi^{w(\mathbf{n})} t^{\vec{v}} z^\mathbf{n} \| = \mathrm{max}_{\vec{v},\bfn}(|a_{\vec{v},\bfn}|)
    = p^{-\mathrm{min}_{\vec{v},\bfn}(\mathrm{val}(a_{\vec{v},\bfn}))}.
\end{equation}
    
\begin{lemma}\label{lem:BisforBanach}
The norm \eqref{eqn:B-norm} makes $R_B$ into a Banach algebra.
\end{lemma}

\begin{proof}
Completeness is standard. To show that this is a Banach algebra, it suffices to look at the coefficients of the product, see \eqref{eq:prod-coeff}, and the first inequality in \eqref{eq:c-min-c}.
\end{proof}

One can introduce the same norm on $\Omega^*_B$, using the basis \eqref{eq:differential-form-basis}. The same works for the Frobenius-twisted versions $R_B^{(p)}$, using the norms of the coefficients in \eqref{eq:bp-series}; and then \eqref{eqn:t-p-power-2} is a map of Banach algebras. Again, the same applies to $\Omega^{(p),*}_B$.

\begin{lemma}\label{lemma:bounded}
The connection \eqref{eq:t-connection} is a bounded operator, of norm $\leq 1$.
\end{lemma}

\begin{proof}
It suffices to prove the assertion for $\Omega_B^0 = R_B$. We consider the two terms in \eqref{eq:t-connection} separately; the desired statement follows because the norm on $R_B$ is non-archimedean. We have
\begin{equation} \label{eq:t-dt}
t_i\partial_{t_i} \big( \pi^{w(\bfn)} t^{\vec{v}} z^{\bfn} \big) 
= v_i \pi^{w(\bfn)} t^{\vec{v}} z^{\bfn}.
\end{equation}
Since the $v_i$ are integral, $|v_i| \leq 1$. Hence, if we apply the same to series \eqref{eq:define-rb}, the norm of each coefficient cannot increase. For the second term, let $\vec{f}_i \in \bZ^r$ be the $i$-th unit vector, so that $t_i = t^{\vec{f}_i}$ and $(t_i\partial_{t_i}W_t) = \pi t^{\vec{f}_i} z^{\bfe_i}$. Since $w(\bfe_i) = 1$ by definition, we have
\begin{equation} \label{eq:tt}
(t_i\partial_{t_i} W_t) \cdot \big( \pi^{w(\bfn)} t^{\vec{v}} z^{\bfn} \big) =
\pi^{w(\bfn)+1} t^{\vec{v}+\vec{f}_i} z^{\bfn+\bfe_i}
= \pi^{w(\bfn)+w(\bfe_i)-w(\bfn + \bfe_i)}\, \big( \pi^{w(\bfn+\bfe_i)} t^{\vec{v}+\vec{f}_i} z^{\bfn+\bfe_i}\big),
\end{equation}
where the coefficient $\pi^{w(\bfn)+w(\bfe_i)-w(\bfn + \bfe_i)}$ has norm $\leq 1$ by \eqref{eq:convex-psi}.
\end{proof}

\subsection{Computing Dwork cohomology\label{subsec:compute-cohomology}}
Our strategy is by reduction to the mod $p$ computations from Section \ref{sec:modp}. We have 
\begin{equation}
\begin{aligned}
& OB \twoheadrightarrow OB/\pi \stackrel{\iso}{\longrightarrow} \bar\Lambda, \\
& \sum_{\vec{v}} a_{\vec{v}} t^{\vec{v}} \longmapsto \sum_{\vec{v}} \bar{a}_{\vec{v}} q^{\vec{v}},
\end{aligned}
\end{equation}
where $\bar{a}_{\vec{v}} \in \bF_p$ is the mod $\pi$ image of $a_{\vec{v}} \in O$; the condition \eqref{eq:b-condition} ensures that the sum on the target side is finite. Similarly,
\begin{equation}
\begin{aligned} \label{eq:reductionhom} 
& R_{OB} \twoheadrightarrow R_{OB}/\pi \stackrel{\iso}{\longrightarrow} GR_{\bar\Lambda}, \\
& \sum_{\bfn,\vec{v}} a_{\mathbf{n},\vec{v}}\, \pi^{w(\mathbf{n})} t^{\vec{v}} z^\mathbf{n}\longmapsto \sum_{\bfn,\vec{v}} \overline{a_{\mathbf{n},\vec{v}}}\, q^{\vec{v}} z^{\mathbf{n}};
\end{aligned}
\end{equation}
the fact that the mod $\pi$ reduction of the ring structure on $R_{OB}$ satisfies the relations \eqref{eq:new-sr} follows from Lemma \ref{th:toriccombinatorics}. The same applies to $\Omega^*_{OB} \rightarrow G\Omega^*_{\bar\Lambda}$, using the basis \eqref{eq:differential-form-basis}. Comparing \eqref{eq:q-superpotential} and \eqref{eq:t-superpotential} shows that \eqref{eq:reductionhom} maps $W_t$ to $W_q$. As a consequence, we have an isomorphism of complexes
\begin{equation} \label{eq:derham-reduction}
(\Omega^*_{OB}, d+dW_t)/\pi \iso (G\Omega^*_{\bar\Lambda}, d+dW_q)
\end{equation}

The following technical result is a variant of \cite[Prop.~A.1]{adolphson-sperber89}, with the same proof.

\begin{lem} \label{lem:ASA1}
Let 
\begin{equation}
C^* = \{ \cdots \to C^{n-2} \to C^{n-1} \to C^{n} \to 0 \} 
\end{equation}
be a bounded above complex of flat $O$-modules, such that the $\pi$-adic filtration on each $C^i$ is separated and complete. Let $\bar{C}^* = C^*/\pi$ be the associated complex of $\bF_p$-vector spaces. 

(i) The natural map $H^n(C^*)/\pi \rightarrow H^n(\bar{C}^*)$ ($n$ being the degree above which the complex is trivial) is an isomorphism.

(ii) For any $i$, $H^i(\bar{C}^*) = 0$ implies $H^i(C^*) = 0$. 

(iii) Suppose that $H^{i-1}(\bar{C}^*) = 0$ for some $i$. Then: $H^i(C^*)$ is a flat $O$-module; the module $\operatorname{im}(d: C^{i-1} \to C^{i})$ is complete with respect to the $\pi$-adic filtration; and $H^i(\bar{C}^*)$ is separated with respect to the $\pi$-adic filtration.
\end{lem} 

The following is an analogue of \cite[Prop.~3.1]{adolphson-sperber89}. 

\begin{prop} \label{prop:AS31q}  
(i) $H^i(\Omega^*_{OB}, d+dW_t)=0$ for $i<n$.

(ii) Let $\Delta \subset N$ be an $\bF_p$-basis set (Definition \ref{def:f-basis}). Then, the classes
\begin{equation} \label{eq:ob-module-basis}
[t^{\vec{\psi}(\bfn)}\pi^{w(\bfn)}z^{\bfn}]_{\bfn \in \Delta} \in H^n(\Omega^*_{OB},d+dW_t), \;\;
\bfn \in \Delta,
\end{equation}
freely generate that cohomology as an $OB$-module.

(iii) Analogues of (i) and (ii) apply to $(\Omega^*_B, d+dW_t)$, after replacing $OB$ by $B$ everywhere.
\end{prop}

\begin{proof} 
We apply Lemma \ref{lem:ASA1} with $C^*$ being $(\Omega^*_{OB}, d+dW_t)$, and $\bar{C}^*$ being $(G\Omega^*_{\bar{\Lambda}}, d+dW_q)$, using the isomorphism \eqref{eq:derham-reduction}. The necessary cohomological vanishing for $\bar{C}^*$ comes from Corollary \ref{th:less-than-n}. Part (i) then follows immediately.

Let's turn to (ii). Lemma \ref{lem:ASA1} implies that $H^n(C^*)$ is is $O$-flat and $\pi$-separated. Consider the natural map
\begin{equation} \label{eq:into-completion}
\Big( \bigoplus_{\bfn \in \Delta} t^{\vec{\psi}(\bfn)}\pi^{w(\bfn)}z^{\bfn} \, \mathit{vol} \cdot \mathit{OB} \Big) \longrightarrow H^n(C^*) \longrightarrow \widehat{H^n(C^*)},
\end{equation}
where the last term is $\pi$-adic completion. After reduction mod $\pi$, this map is an isomorphism; that follows from Lemma \ref{lem:ASA1}(i) and Corollary \ref{cor:generatingalgebraicGM2}. Nakayama's lemma (in the form \cite[Lemma 10.96.1 (2)]{stacks-project-00M9}) then shows that \eqref{eq:into-completion} is surjective. We already know that the second map in \eqref{eq:into-completion} is injective, that being separability. Therefore, the first map must also be surjective (and the second map an isomorphism). We next show that
\begin{equation} \label{eq:int-with-d}
\Big( \bigoplus_{\bfn \in \Delta} t^{\vec{\psi}(\bfn)}\pi^{w(\bfn)}z^{\bfn} \, \mathit{vol} \cdot \mathit{OB} \Big) \cap (d+dW_t)(\Omega^{n-1}_{OB}) = 0.
\end{equation}
From Lemma \ref{lem:ASA1} we know that $(d+dW_t)(\Omega^{n-1}_{OB})$ is complete, hence so is the intersection \eqref{eq:int-with-d}. From Corollary \ref{cor:generatingalgebraicGM2} it follows that if we reduce that intersection modulo $\pi$, it is zero. Therefore, it vanishes by Nakayama's Lemma.

(iii) Follows by taking the tensor product with $K$, over $O$ (this is an exact functor).
\end{proof} 

Our next concern is to compare the analytic theory constructed here with the algebraic one from Section \ref{sect:toricms}. For that, consider the change of coordinates
\begin{equation} \label{eq:iota}
\begin{aligned}
& \Lambda \longrightarrow B, \\
& q_i \longmapsto \pi t_i;
\end{aligned}
\end{equation}
similarly $R_\Lambda \rightarrow R_B$, and also for differential forms. This maps $W_q$ to $W_t$, hence we get a homomorphism of twisted de Rham complexes
\begin{equation} \label{eq:algebraictoB} 
(\Omega^*_\Lambda ,d+dW_q) \longrightarrow (\Omega^*_{B},d+dW_t). 
\end{equation}
Moreover, this map is compatible with the connections \eqref{eq:GMconnections}, \eqref{eq:t-connection}.

\begin{corollary} \label{th:compare-cohomologies}
The map \eqref{eq:algebraictoB} induces an isomorphism
\begin{equation} \label{eq:algebraictoBcomparison}
H^n(\Omega^*_\Lambda,d+dW_q) \otimes_\Lambda B \stackrel{\iso}{\longrightarrow} H^*(\Omega^*_B,d+dW_t),
\end{equation}
where the tensor product is taken with respect to \eqref{eq:iota}. 
\end{corollary}

\begin{proof}
Let $\Delta \subset N$ be an $\bF_p$-basis set. By Lemma \ref{th:f-basis}(iii), this is also a $\bQ$-basis set, hence gives a basis \eqref{eq:bar-lambda-generator} for the free $\Lambda$-module $H^*(\Omega_{\Lambda},d+dW_q)$. Under \eqref{eq:iota}, the elements of this basis map to those in \eqref{eq:ob-module-basis}, so appealing to Proposition \ref{prop:AS31q}(iii) completes the argument.
 \end{proof}


\subsection{The inverse Frobenius\label{subsec:dwork-frobenius}}
Dwork's construction for the (inverse) Frobenius comes in two steps. Start with a differential form in $\Omega^*_B$. The first step is multiplication with an invertible power series constructed from the Dwork exponential \eqref{eq:dwork-exponential}, namely
\begin{equation} \label{eq:dwork-twist-3}
\begin{aligned}
G_t & = D(\pi t_1 z^{\bfe_1}) \cdots D(\pi t_r z^{\bfe_r}) \\
& = \sum_{m_1, \dots, m_r \geq 0} d_{m_1} \cdots d_{m_r} \pi^{m_1 + \cdots + m_r}\,
t^{\vec{m}} z^{m_1 \bfe_1 + \cdots + m_r \bfe_r}.
\end{aligned}
\end{equation}

\begin{remark}
A word of caution: recall (see Section \ref{sec:p-adic-functions}) that the Dwork exponential has $p$-adic radius of convergence $<1$. As a consequence, $G_t \notin R_B$, so multiplication with that function does not preserve the spaces $\Omega^*_B$.
\end{remark}

The second step is to make the substitution $z_i \mapsto z_i^{1/p}$ and throw away non-integer powers of $z_i$, which is denoted by $\mathit{Loc}$. This is extended to differential forms so as to be compatible with the de Rham differential. Explicitly, 
\begin{equation} \label{eq:dwork-root-3}
\mathit{Loc}\Big( z_1^{i_1} \cdots z_n^{i_n} \frac{dz_{j_1}}{z_{j_1}} \wedge 
\cdots \wedge \frac{dz_{j_l}}{z_{j_l}} \Big) =
\begin{cases} 
z_1^{i_1/p} \cdots z_n^{i_n/p} \cdot p^{-l} \frac{dz_{j_1}}{z_{j_1}} \wedge 
\cdots \wedge \frac{dz_{j_l}}{z_{j_l}}
& \text{$i_1,\dots,i_n \in p\bZ$,} \\
0 & \text{otherwise.}
\end{cases}
\end{equation}

\begin{lemma} \label{th:stupid-lemma}
If $\theta \in \Omega^*_B$, then $\mathit{Loc}(G_t \cdot \theta) \in \Omega^{(p),*}_B$.
\end{lemma}

\begin{proof}
We explain the proof for degree $0$ differential forms, meaning $R_B = \Omega^0_B$; the general case works in the same way. If the original element is written as in \eqref{eq:define-rb}, the image under $\mathit{Loc}(G_t \cdot)$ is 
\begin{equation} \label{eqn:explicit-frobenius}
\sum a_{\vec{v},\bfn} d_{m_1} \cdots d_{m_r} \pi^{|\vec{m}| + w(\mathbf{n})} t^{\vec{m} + \vec{v}}z^{(\mathbf{n} + m_1 \bfe_1 + \cdots + m_r \bfe_r)/p},
\end{equation}
where the summation is over $(\vec{m},\vec{v},\bfn)$ satisfying
\begin{equation} \label{eq:mvn}
\begin{aligned}
& m_1,\dots,m_r \geq 0, \\
& |\vec{v}| \geq w(\bfn), \\ 
& \bfn + m_1 \bfe_1 + \cdots + m_r \bfe_r \in p\bZ^n.
\end{aligned}
\end{equation}
By Lemma \ref{th:toriccombinatorics}(i), each collection \eqref{eq:mvn} satisfies
\begin{equation}
\begin{aligned}
        &p \cdot w ( (\mathbf{n} + m_1 \bfe_1 + \cdots + m_r \bfe_r)/p ) \\
        &= w (\mathbf{n} + m_1 \bfe_1 + \cdots + m_r \bfe_r) \leq w(\mathbf{n}) + |\vec{m}| \leq |\vec{v} + \vec{m}|,
\end{aligned}
\end{equation}
which meets the condition placed on monomials for elements in $B^{(p)}$. To match \eqref{eq:bp-series} more closely, write \eqref{eqn:explicit-frobenius} as
\begin{equation} \label{eq:explicit-frobenius-2}
\begin{aligned}
& \sum b_{\vec{m}, \vec{v}, \mathbf{n}} \,\pi^{w((\bfn + m_1\bfe_1 + \cdots + m_r\bfe_r)/p)}
t^{\vec{m}+\vec{v}} z^{(\bfn+m_1 \bfe_1 + \cdots + m_r \bfe_r)/p}, \\
& b_{\vec{m}, \vec{v}, \bfn} = a_{\vec{v},\bfn} d_{m_1} \cdots d_{m_r} \pi^{|\vec{m}| + w(\mathbf{n}) - w((\mathbf{n} + m_1 \bfe_1 + \cdots + m_r \bfe_r)/p)}.
\end{aligned}
\end{equation}
Then
\begin{equation}\label{eqn:stupid-computation}
\begin{aligned}
    \mathrm{val}(b_{\vec{m},\vec{v},\bfn}) & = \mathrm{val}(a_{\vec{v},\bfn}) +
    \sum_{i=1}^r \mathrm{val}(d_{m_i}) \\ & \qquad \qquad + \frac{1}{p-1} \big(|\vec{m}| + w(\bfn) -
    w((\bfn + m_1\bfe_1 + \cdots + m_r\bfe_r)/p) \big)
\\ &
    \geq \mathrm{val}(a_{\vec{v},\mathbf{n}}) + |\vec{m}|\big(\frac{1-2p}{p^3 - p^2}\big) + 
    \frac{1}{p} \big(|\vec{m}| + w(\bfn)\big) \\ & \qquad \qquad
    + \frac{1}{p^2-p} \big( |\vec{m}| + w(\bfn) - w(\bfn + m_1\bfe_1 + \cdots + m_r\bfe_r) \big) 
    \\
& \geq \mathrm{val}(a_{\vec{v},\mathbf{n}}) + \frac{1}{p}w(\mathbf{n}) + \frac{p^2 - 3p +1}{p^3 - p^2} |\vec{m}|.
\end{aligned}
\end{equation}
In the third line, we have used the estimate \eqref{eq:dwork-exponential-val}; the last inequality comes from Lemma \ref{th:toriccombinatorics}(i). The last line, seen as a function of $(\vec{m},\vec{v},\bfn)$, is proper and bounded below: it is the sum of a proper bounded below function in the variables $(\vec{v},{\bfn})$ (the first term), another such function in $\vec{m}$ (the last term), and a nonnegative middle term. Hence, for any $C$ there are only finitely many $(\vec{m},\vec{v},\bfn)$ with $\mathrm{val}(b_{\vec{m},\vec{v},\bfn}) < C$.

The first consequence is that for each fixed value of $(\vec{m}+\vec{v}, \bfn + m_1\bfe_1 + \cdots + m_r\bfe_r)$, the relevant terms in \eqref{eq:explicit-frobenius-2} form a $p$-adically convergent series. Secondly, once one adds up those terms, the outcome is a series in the $(t,z)$ variables which satisfies the finiteness condition defining $B^{(p)}$.
\end{proof}

Having checked that this makes sense, we now state the outcome formally:

\begin{definition} \label{defn:dwork-frobenius}
The Dwork inverse Frobenius is the $B$-linear map
\begin{equation} \label{eqn:dwork-frobenius}
\begin{aligned}
\Psi: \Omega^*_{B} &\longmapsto \Omega_B^{(p),*}, \\
\theta &\longmapsto  \mathit{Loc}(G_t\cdot \theta).
\end{aligned}
\end{equation}
\end{definition}

\begin{remark}
The computation in Lemma \ref{th:stupid-lemma} shows that $\Psi$ is a bounded operator between Banach spaces, of norm $\leq 1$.
\end{remark}

\begin{lemma} \label{th:intertwinesconnections}
(i) $\Psi$ is a chain map.

(ii) $\Psi$ intertwines the connections \eqref{eq:t-connection} and \eqref{eq:derham-frobenius-connection}:
\begin{equation}
\nabla_{t_i\partial_{t_i}}^{(p)} \circ \Psi = \Psi \circ \nabla_{t_i \partial_{t_i}}.
\end{equation}
\end{lemma}

\begin{proof}
This is a formal computation, which can be carried out without considering the convergence conditions involved in our Banach spaces. The formula
\begin{equation}
G_t = \prod_{i=1}^r\exp(\pi(t_iz^{\bfe_i} - t_i^p z^{p\bfe_i})),
\end{equation}
indicates that multiplication with $G_t$ conjugates the twisted de Rham differentials associated to $W_t$ and $\pi \sum_i t_i^p z^{p\bfe_i}$; and $\mathit{Loc}$ relates the latter function to $W_{t^p}$. Compatibility of connections can be derived from the same conjugation argument, but we prefer a direct computation:
\begin{equation}
\begin{aligned}
& \nabla^{(p)}_{t_i \partial_{t_i}} \mathit{Loc}(G_t \, \theta) 
= \big(t_i\partial_{t_i} + (\pi p t_i^p z^{\bfe_i})\big) \mathit{Loc}(G_t \, \theta) \\
& \qquad = \mathit{Loc}\big(t_i \partial_{t_i}(G_t \, \theta) + G_t \, \pi p t_i^p z^{p \bfe_i} \theta\big) \\
            &\qquad = \mathit{Loc}\big(G_t \, (t_i \partial_{t_i} \theta) + G_t \, \pi(t_i z^{\bfe_i} - p t_i^p z^{p\bfe_i})\, \theta + G_t \, \pi p t_i^p z^{p\bfe_i} \, \theta\big) \\
            & \qquad = \mathit{Loc}\big(G_t \cdot (t_i \partial_{t_i} \theta + \pi t_i z^{\bfe_i})\theta\big) = \mathit{Loc}(G_t \cdot \nabla_{t_i \partial_{t_i}} \theta).
        \end{aligned}
\end{equation}
\end{proof}

\subsection{Conclusion\label{subsec:conclusion}}
Let's use mirror symmetry to re-express the inverse Frobenius in terms of the toric Fano manifold $M$. On the quantum cohomology side, we switch coordinates from $q_i$ to $t_i$, which means working with the multivariable version of \eqref{eq:quantum-connection-2}:
\begin{equation} \label{eq:quantum-connectionB}
\begin{aligned}
& \nabla_{t_i\partial_{t_i}}: H^*(M;B) \longrightarrow H^*(M;B), \\
& \nabla_{t_i\partial_{t_i}}(x) = t_i\partial_{t_i} x + [D_i] \ast_{\pi t} x.
\end{aligned}
\end{equation}
As usual, this is abbreviated notation, with $\pi t$ meaning the parameters $(\pi t_1,\dots,\pi t_r)$. The Frobenius pullback of that connection is
\begin{equation} \label{eq:quantum-connectionBP}
\begin{aligned}
& \nabla_{t_i\partial_{t_i}}^{(p)}: H^*(M;B) \longrightarrow H^*(M;B), \\
& \nabla_{t_i\partial_{t_i}}^{(p)}(x) = t_i\partial_{t_i} x + p [D_i] \ast_{\pi t^p} x.
\end{aligned}
\end{equation}
Let's extend constants in \eqref{eq:mirrornoneqiso} from $\Lambda$ to $B$ (by tensor product), and use Corollary \ref{th:compare-cohomologies} to compare the twisted de Rham complexes. The outcome is a $B$-linear isomorphism
\begin{equation} 
\label{eq:mirrornoneqisoB} 
\Theta: H^n(\Omega^*_B,d+dW_t) \stackrel{\iso}{\longrightarrow} H^*(M;B),
\end{equation} 
such that
\begin{align}
& \Theta([\mathit{vol}]) = 1; \\
\label{eq:i}
& \Theta \text{ intertwines the connections \eqref{eq:t-connection} and \eqref{eq:quantum-connectionB}}.
\end{align}
After further enlarging constants from $B$ to $B[t_1^{1/p},\dots,t_r^{1/p}]$ in \eqref{eq:mirrornoneqisoB}, one can use that to define another $B$-linear isomorphism $\Theta^{(p)}$, by the diagram
\begin{equation}
\xymatrix{
\ar[d]_-{\Theta^{(p)}}^-{\iso} H^n(\Omega^{(p),*}_B, d+dW_{t^p})
&& \ar[ll]^-{\iso}_-{\eqref{eq:BpisFrobeniuspullback}}
H^n(\Omega^*_B[t_1^{1/p},\dots,t_r^{1/p}], d+dW_t) 
\ar[d]^-{\Theta}_-{\iso}
\\
H^*(M;B) 
&& \ar[ll]^-{\iso}_-{F}
H^*(M;B[t_1^{1/p},\dots,t_r^{1/p}])
}
\end{equation}
Here $F$, as usual, is $(t_1,\dots,t_r) \mapsto (t_1^p,\dots,t_r^p)$; note that $\nabla_{t_i\partial_{t_i}}^{(p)} \circ F = F \circ (p\nabla_{t_i\partial_{t_i}})$, exactly as in \eqref{eq:power}. As a consequence of the definition,
\begin{align}
\label{eq:frobenius-unit}
& \Theta^{(p)}([\mathit{vol}]) = 1; \\
\label{eq:i2}
& \Theta^{(p)} \text{ intertwines the connections \eqref{eq:derham-frobenius-connection} and \eqref{eq:quantum-connectionBP}}.
\end{align}

Finally, the outcome of this section can be summarized as follows:

\begin{corollary} \label{th:multivariable}
The map \eqref{eqn:dwork-frobenius}, combined with the isomorphisms $\Theta$ and $\Theta^{(p)}$, yields an inverse Frobenius for the quantum connection. Explicitly, this is a $B$-linear map which intertwines \eqref{eq:quantum-connectionB} and \eqref{eq:quantum-connectionBP}:
\begin{align} \label{eq:frob-geometric}
& \Psi: H^*(M;B) \longrightarrow H^*(M;B), \\
& \nabla^{(p)}_{t_i\partial_{t_i}} \circ \Psi = \Psi \circ \nabla_{t_i\partial_{t_i}}.
\label{eq:geometric-intertwine}
\end{align}
\end{corollary}


\section{The initial term\label{sec:constant-term}}

The aim of this section is to compute the previously constructed inverse Frobenius modulo $IB$, relating it to the $p$-adic Gamma function. The key idea in the computation is to insert one of the operations \eqref{eq:t-connection} into the Mahler expansion \eqref{eq:mahler} with $j = 0$:
\begin{equation} \label{eq:gamma-of-nabla}
\Gamma_p(-p \nabla_{t_i\partial_{t_i}}) = \sum_{m\geq 0} (-p)^m d_{m p}\,
 \nabla_{t_i\partial_{t_i}} (\nabla_{t_i\partial_{t_i}} - 1) \cdots ( \nabla_{t_i\partial_{t_i}} - m + 1):
 \Omega_B^* \longrightarrow \Omega_B^*.
\end{equation}
This makes sense since: $\nabla_{t_i\partial_{t_i}}$ is an endomorphism of $\Omega_B^*$ of norm $\leq 1$ (Lemma \ref{lemma:bounded}), hence so is $(\nabla_{t_i\partial_{t_i}}-  k)$ for any $k \in \bZ$; and $p^m d_{m p}$ goes to zero $p$-adically as $m \rightarrow \infty$, by \eqref{eq:val-mahler}. This means that the series \eqref{eq:gamma-of-nabla} converges absolutely in the operator norm of the space of endomorphisms of $\Omega^*_B$. 
The operators $\nabla_{t_i\partial_{t_i}}$ commute with each other, hence so do the endomorphisms \eqref{eq:gamma-of-nabla}; and we can write their composition as
\begin{equation} \label{eq:gamma-of-nabla-2}
\begin{aligned}
& \prod_{i=1}^r \Gamma_p(-p \nabla_{t_i\partial_{t_i}}) \\ & \quad = \!\!\!\sum_{m_1,\dots,m_r\geq 0} \!\!\!(-p)^{m_1+\cdots+m_r} d_{m_1 p}\cdots d_{m_r p}
\prod_{i=1}^r \nabla_{t_i\partial_{t_i}} (\nabla_{t_i\partial_{t_i}} - 1) \cdots ( \nabla_{t_i\partial_{t_i}} - m_i + 1).
\end{aligned}
\end{equation}
We now apply this to the element \eqref{eq:vol}.

\begin{lemma} \label{th:mahler-1}
In $\Omega^n_B$, we have
\begin{equation} \label{eq:gamma-of-nabla-4}
\begin{aligned}
& \prod_{i=1}^r \Gamma_p(-p \nabla_{t_i\partial_{t_i}}) \mathit{vol} 
\\ & \qquad = \sum_{\!\!\!\!m_1,\dots,m_r \geq 0} d_{m_1p}\cdots d_{m_rp} \pi^{(m_1+\cdots+m_r)p} t_1^{m_1}\cdots t_r^{m_r} z^{m_1 \bfe_1 + \cdots + m_r \bfe_r} \mathit{vol}
\end{aligned}
\end{equation}
\end{lemma}

\begin{proof}
Consider 
\begin{equation} \label{eq:no-sense}
\nabla_{\partial_{t_i}} = t_i^{-1}\nabla_{t_i\partial_{t_i}} = \partial_{t_i} + \partial_{t_i}W_t = \partial_{t_i} + \pi z^{\bfe_i}. 
\end{equation}
This makes no sense as an endomorphism of $\Omega_B^*$, since it does not preserve the inequality on monomials from \eqref{eq:define-rb}. However, $t_i^m \nabla_{\partial_{t_i}}^m$ does make sense for any $m \geq 1$, and (as a general property of connections)
\begin{equation}
t_i^m \nabla_{\partial_{t_i}}^m = \nabla_{t_i\partial_{t_i}} (\nabla_{t_i \partial_{t_i}} - 1) \cdots
(\nabla_{t_i \partial_{t_i}} - m + 1).
\end{equation}
Hence, one can rewrite \eqref{eq:gamma-of-nabla} as
\begin{equation} 
\Gamma_p(-p \nabla_{t_i\partial_{t_i}}) = \sum_{m \geq 0} (-p)^m d_{m p}\, t_i^m \nabla_{\partial_{t_i}}^m.
\end{equation}
The same works for \eqref{eq:gamma-of-nabla-2}, yielding
\begin{equation} \label{eq:gamma-of-nabla-3}
\prod_{i=1}^r \Gamma_p(-p \nabla_{t_i\partial_{t_i}}) = \sum_{m_1,\dots,m_r \geq 0} (-p)^{m_1+\cdots +m_r} d_{m_1p} \cdots d_{m_rp} t_1^{m_1}\cdots t_r^{m_r} \nabla_{\partial_{t_1}}^{m_1} \cdots
\nabla_{\partial_{t_r}}^{m_r}.
\end{equation}
By applying \eqref{eq:no-sense} repeatedly, one gets
\begin{equation}
\begin{aligned}
& (-p)^{m_1+\cdots+m_r} t_1^{m_1}\cdots t_r^{m_r} \nabla_{\partial_{t_1}}^{m_1}\cdots \nabla_{\partial_{t_r}}^{m_r} \mathit{vol} 
\\ & \qquad 
= (-p)^{m_1+\cdots+m_r} t_1^{m_1}\cdots t_r^{m_r} \pi^{m_1+\cdots+m_r} z^{m_1 \bfe_1 + \cdots + m_r \bfe_e} \mathit{vol}
\\ & \qquad = 
\pi^{(m_1+\cdots+m_r)p} t_1^{m_1} \cdots t_r^{m_r} z^{m_1 \bfe_1 + \cdots + m_r \bfe_r} \mathit{vol}.
\end{aligned}
\end{equation}
Inserting that into \eqref{eq:gamma-of-nabla-3} yields \eqref{eq:gamma-of-nabla-4}.
\end{proof}

\begin{lemma} \label{th:mahler-2}
In $\Omega^{(p),n}_B$, we have
\begin{equation} \label{eq:gamma-of-nabla-5}
\begin{aligned}
& \prod_{i=1}^r \Gamma_p(-\nabla^{(p)}_{t_i\partial_{t_i}})\, \mathit{vol}
\\ & \quad = \!\!\! \sum_{\substack{m_1,\dots,m_r \geq 0 \\ m_1 \bfe_1+ \cdots + m_r \bfe_r \in pN}} \!\!\! d_{m_1}\cdots d_{m_r} \pi^{m_1+\cdots+m_r} t_1^{m_1}\cdots t_r^{m_r} z^{(m_1 \bfe_1 + \cdots + m_r \bfe_r)/p} \mathit{vol} \;\; \mathrm{mod}\; IB. 
\end{aligned}
\end{equation}
\end{lemma}

\begin{proof}
Take the result from Lemma \ref{th:mahler-1} and map it to the Frobenius-twisted version by \eqref{eqn:t-p-power-2} to the Frobenius-twisted version. The outcome, bearing \eqref{eq:power} in mind, is
\begin{equation} \label{eq:gamma-p-vol}
\begin{aligned}
& \prod_{i=1}^r \Gamma_p(-\nabla_{t_i\partial_{t_i}}^{(p)})\, \mathit{vol} 
= F\Big( \prod_{i=1}^r \Gamma_p(-p\nabla_{t_i\partial_{t_i}})\, \mathit{vol}\Big)
\\ & \qquad = \sum_{\substack{m_1,\dots,m_r \geq 0\\ m_i \in \bZ}} d_{m_1p}\cdots d_{m_rp} \pi^{(m_1+\cdots+m_r)p} t_1^{pm_1}\cdots t_r^{pm_r} z^{m_1 \bfe_1 + \cdots + m_r \bfe_r} 
\\ & \qquad = \sum_{\substack{m_1,\dots,m_r \geq 0\\ m_i \in p\bZ}} d_{m_1}\cdots d_{m_r} \pi^{m_1+\cdots+m_r} t_1^{m_1}\cdots t_r^{m_r} z^{(m_1 \bfe_1 + \cdots + m_r \bfe_r)/p} \mathit{vol};
\end{aligned}
\end{equation}
in the last line, we have simply changed the indices $m_i$ to $p m_i$.

From Lemma \ref{th:toriccombinatorics} we see that a monomial $t_1^{m_1}\cdots t_r^{m_r} z^{m_1 \bfe_1+\cdots+m_r \bfe_r}$ lies in $IB \cdot R_B$ unless the following holds:
\begin{equation} \label{eq:single-cone}
\parbox{35em}{
For all $i$ such that $m_i>0$, the rays $\bfe_i$ belong to a common cone of the fan.
}
\end{equation}
In $R_B^{(p)}$, the same applies to $t_1^{m_1}\cdots t_r^{m_r} z^{(m_1\bfe_1+\cdots+m_r \bfe_r)/p}$. Because of Lemma \ref{th:toric-smoothness}, we also have:
\begin{equation} \label{eq:single-cone-2}
\parbox{35em}{If $m_1\bfe_1 + \cdots + m_r\bfe_r \in pN$ and \eqref{eq:single-cone} holds, then $m_i \in p\bZ$ for all $i$.}
\end{equation}
With that in mind, when working modulo $IB$ we can ignore all the summands in \eqref{eq:gamma-p-vol} where \eqref{eq:single-cone} fails, and then apply \eqref{eq:single-cone-2} to replace the condition $m_i \in p\bZ$ by $m_1\bfe_1+\cdots+m_r\bfe_r \in pN$. This yields the desired formula.
\end{proof}

We will use $\sim$ to indicate that two elements of $\Omega^n_B$ are homologous modulo $IB$:
\begin{equation} \label{eq:theta-and-eta}
\alpha \sim \beta\;\; \Leftrightarrow\;\; \alpha-\beta = (d+dW_t)(\gamma) + \delta, \text{ for some } \gamma \in \Omega^{n-1}_B,\; \delta \in IB \cdot \Omega^n_B.
\end{equation}
The same notation will also be applied to $\Omega_B^{(p),n}$. Mirror symmetry yields the following:

\begin{lemma} \label{th:modt-1}
For any $k_1,\dots,k_r \geq 0$ such that $k_1 + \cdots + k_r > n$, we have
\begin{equation}
\nabla_{t_1\partial_{t_1}}^{k_1} \cdots \nabla_{t_r\partial_{t_r}}^{k_r} \mathit{vol} \sim 0.
\end{equation}
Moreover, the relevant elements $\gamma = \gamma_{k_1,\dots,k_r}$ and $\delta = \delta_{k_1,\dots,k_r}$ in \eqref{eq:theta-and-eta} can be chosen so as to be uniformly bounded, with respect to the Banach norm on $\Omega^*_B$.
\end{lemma}

\begin{proof}
For a single $(k_1,\dots,k_r)$, the desired statement follows from the corresponding property of the quantum connection, and \eqref{eq:i}: the cup product of more than $n$ divisor classes is zero for degree reasons, and hence the associated composition of quantum connections is zero modulo $IB$. Note that there are finitely many $(k_1,\dots,k_r)$ such that $k_1 + \cdots + k_r = n+1$. The general case can be derived from those finitely many by applying some composition of the operators $\nabla_{t_i\partial_{t_i}}$; since those have norm $\leq 1$, we get the desired uniform bound.
\end{proof}

Let's introduce some temporary notation:
\begin{equation}
z(z-1)\cdots(z-{m+1}) = \sum_k c_{m,k} z^k,
\end{equation}
where $c_{m,k} \in \bZ$ (for $k > m$ we set $c_{m,k} = 0$). By definition, the coefficients in \eqref{eq:gamma-taylor} are 
\begin{equation} \label{eq:c-g}
g_k = \sum_m (-p)^{m-k} d_{mp} c_{m,k}.
\end{equation}

\begin{lemma} \label{th:modt-2}
For any $(m_1,\dots,m_r)$, we have
\begin{equation} \label{eq:bounded-gamma-expansion}
\begin{aligned}
& \prod_{i=1}^r \nabla_{t_i\partial_{t_i}}(\nabla_{t_i\partial_{t_i}}-1)\cdots (\nabla_{t_i\partial_{t_i}}-m_i+1)\, \mathit{vol} 
\\ & \qquad \sim \sum_{\!\!\!\!\substack{k_1,\dots, k_r \geq 0 \\ k_1+\cdots+k_r \leq n}} c_{m_1,k_1} \cdots c_{m_r,k_r} \nabla_{t_1\partial_{t_1}}^{k_1} \cdots \nabla_{t_r\partial_{t_r}}^{k_r} \mathit{vol},
\end{aligned}
\end{equation}
Moreover, as before, the relevant forms $\gamma = \gamma_{m_1,\dots,m_r}$ and $\delta = \delta_{m_1,\dots,m_r}$ from \eqref{eq:theta-and-eta} can be taken uniformly bounded.
\end{lemma}

\begin{proof}
When we take the left hand side of \eqref{eq:bounded-gamma-expansion} and expand out the compositions, the outcome is the right hand side without the restriction $k_1 + \dots + k_r \leq n$. However, the terms which violate that restriction are covered by Lemma \ref{th:modt-1} (and since the $c_{m,k}$ are integers, multiplication with them preserves the given bounds on the norm).
\end{proof}

\begin{lemma} \label{th:modt-3}
We have
\begin{equation}
\prod_{i=1}^r \Gamma_p(-p\nabla_{t_i\partial_{t_i}}) \mathit{vol} \sim
\sum_{\substack{k_1,\dots,k_r \geq 0 \\ k_1+\cdots+k_r \leq n}} g_{k_1}\cdots g_{k_r} (-p\nabla_{t_1\partial_{t_1}})^{k_1} \cdots
(-p\nabla_{t_r\partial_{t_r}})^{k_r} \mathit{vol}.
\end{equation}
\end{lemma}

\begin{proof}
We start with \eqref{eq:gamma-of-nabla-2} and apply Lemma \ref{th:modt-2} to each term; this can be done for all terms at the same time, due to the uniform boundedness properties derived above. The outcome is
\begin{equation}
\begin{aligned}
&\prod_{i=1}^r \Gamma_p(-p\nabla_{t_i\partial_{t_i}}) \mathit{vol} \sim
\sum_{\!\!\!\!\substack{k_1,\dots, k_r \geq 0 \\ k_1+\cdots+k_r \leq n}} 
\sum_{m_1,\dots,m_r \geq 0} d_{m_1p} \cdots d_{m_rp} c_{m_1,k_1} \cdots c_{m_r,k_r}
\\[-1.5em] & \qquad \qquad \qquad \qquad \qquad \qquad \qquad \qquad \qquad \qquad (-p)^{m_1+\cdots+m_r} \nabla_{t_1\partial_{t_1}}^{k_1} \cdots \nabla_{t_r\partial_{t_r}}^{k_r}\, \mathit{vol}.
\end{aligned}
\end{equation}
Applying \eqref{eq:c-g} to the right hand simplifies this to the desired expression.
\end{proof}

The following brings the computational part of our argument to its conclusion:

\begin{lemma} \label{th:final-computation}
Dwork's inverse Frobenius (from Definition \ref{defn:dwork-frobenius}) satisfies
\begin{equation}
\Psi(\mathit{vol}) \sim p^{-n} 
\sum_{\!\!\!\!\substack{m_1,\dots, m_r \geq 0 \\ m_1+\cdots+m_r \leq n}} 
g_{m_1}\cdots g_{m_r} (-\nabla^{(p)}_{t_1\partial_{t_1}})^{m_1} \cdots
(-\nabla^{(p)}_{t_r\partial_{t_r}})^{m_r}\, \mathit{vol}.
\end{equation}
\end{lemma}

\begin{proof} 
By definition,
\begin{equation}
\Psi(\mathit{vol}) = p^{-n} \sum_{\substack{m_1,\dots,m_r \geq 0 \\ m_1\bfe_1+\cdots+m_r\bfe_r \in pN}} d_{m_1}\cdots d_{m_r} \pi^{m_1+\cdots+m_r} t_1^{m_1} \cdots t_r^{m_r} z^{(m_1\bfe_1 + \cdots + m_r\bfe_r)/p} \mathit{vol}.
\end{equation}
The desired equality now follows from combining Lemma \ref{th:mahler-2} with the image under \eqref{eqn:t-p-power-2} of Lemma \ref{th:modt-3}.
\end{proof}

Let's switch to thinking of the Frobenius as acting on $H^*(M;B)$, as in Corollary \ref{th:multivariable}. 

\begin{theorem} \label{th:constant-term}
Modulo $IB$, the Frobenius \eqref{eq:frob-geometric} agrees with (the $t$-independent map)
\begin{equation} \label{eq:gamma-const-term}
x \longmapsto p^{\mathrm{deg}/2-n} (\Gamma_p(TM)^{-1} \smile x),
\end{equation}
where $p^{\mathrm{deg}/2}$ multiplies $H^{2i}(M)$ with $p^i$.
\end{theorem}

\begin{proof}
Remember from \eqref{eq:frobenius-unit} that the cohomology class of $\mathit{vol} \in \Omega^{(p),n}_B$ corresponds to $1 \in H^0(M;B)$. The connection \eqref{eq:quantum-connectionBP} satisfies
\begin{equation} 
(\nabla_{t_1\partial_{t_1}}^{(p)})^{m_1} \cdots (\nabla_{t_r\partial_{t_r}}^{(p)})^{m_r} (1) =
p^{m_1+\cdots+m_r} [D_1]^{m_1} \smile \cdots \smile [D_r]^{m_r}\;\; \text{ mod } IB.
\end{equation}
Therefore, Lemma \ref{th:final-computation} translates into
\begin{equation}
\Psi(1) = \sum_{\!\!\!\!\substack{k_1,\dots, k_r \geq 0 \\ k_1+\cdots+k_r \leq n}} 
p^{m_1+\cdots+m_r-n}
g_{m_1}\cdots g_{m_r} (-[D_1])^{m_1} \smile \cdots \smile (-[D_r])^{m_r} \;\; \text{ mod } IB.
\end{equation}
Comparing this with \eqref{eq:inverse-toric-gamma}, where $\Gamma_p(-[D_i])$ is defined in terms of the Taylor expansion, yields exactly the $x = 1$ case of \eqref{eq:gamma-const-term}. The fact that $\Psi$ interwines the connection and its Frobenius pullback implies that
\begin{equation}
\Psi([D_i] \smile y) = p[D_i] \smile \Psi(y) \;\;\text{ mod } IB.
\end{equation}
Since the $[D_i]$ generate $H^*(M;B)$ as a ring, it follows that
\begin{equation}
\Psi(x \smile y) = p^{\mathrm{deg}(x)/2} x \smile \Psi(y) \;\; \text{ mod } IB
\end{equation}
for all $x,y \in H^*(M;B)$. In particular, $\Psi(x) = p^{\mathrm{deg}(x)/2} x \smile \Psi(1)$, which yields the general case of \eqref{eq:gamma-const-term}.
\end{proof}

\section{The single-variable theory\label{sec:single-variable}}
So far, we have worked in the multivariable framework introduced at the start of Section \ref{sec:ms}. Having completed the step for which that setup was required (the proof of Theorem \ref{th:constant-term}), we want to revert to the single-variable version, which will be used for the concluding construction in Section \ref{sec:overconvergent}. The current section has a very simple aim: to review the results, after that simplification has been applied to them.

\begin{convention}
From now on, $q$ or $t$ will again denote a single variable. The relation to the previous multivariable theory is that we set all $q_i$ (or $t_i$) to be equal.
\end{convention}

Let us begin with the mirror symmetry statements. Our ground ring is $\bQ[q]$, and we simplify the multivariable theory by changing coefficients along the map
\begin{align} \label{eq:q-equal}
\Lambda \twoheadrightarrow \Lambda/(q_iq_j^{-1} - 1) \stackrel{\iso}{\longrightarrow} \bQ[q], \quad  q_i \longmapsto q \text{ for all $i$.} 
\end{align}
The ring of functions becomes
\begin{equation} \label{eq:rq}
R_{\bQ[q]} = R_{\Lambda}/(q_iq_j^{-1} - 1) = R_{\Lambda} \otimes_\Lambda \bQ[q] = \Big\{ \sum_{\substack{k, \,\mathbf{n} \\ k \geq w(\bfn)}} a_{k,\mathbf{n}} q^k z^\mathbf{n} \Big\} \subset \bQ[q^{\pm 1},z_j^{\pm 1}],
\end{equation}
and correspondingly differential forms $\Omega^*_{\bQ[q]}$. We keep the same notation for the image of the superpotential under $R_{\Lambda} \rightarrow R_{\bQ[q]}$: concretely,
\begin{equation}
W_q = \sum_{i=1}^{r} q\, z^{\mathbf{e}_i} \in R_{\bQ[q]}.
\end{equation}
The resulting single-variable twisted de Rham differential and single-variable connection are
\begin{equation}
\begin{aligned} 
& d + dW_q: \Omega^{*}_{\bQ[q]} \longrightarrow \Omega^{*+1}_{\bQ[q]}, \\ \label{eq:GMconnectionversionq}
& \nabla_{q\partial_{q}} = q\partial_{q} + (q\partial_q W_{q}) = q\partial_q + W_q: \Omega^{*}_{\bQ[q]} \longrightarrow \Omega^{*}_{\bQ[q]}. 
\end{aligned}
\end{equation}
Base-changing Theorem \ref{thm:toricmirror} (which is unproblematic since the cohomology groups involved are free $\Lambda$-modules) yields:

\begin{corollary} 
There is a $\bQ[q]$-linear isomorphism 
\begin{equation} \label{eq:one-variable-mirror}
\Theta: H^n(\Omega^*_{\bQ[q]}, d+dW_q) \stackrel{\iso}{\longrightarrow} H^*(M;\bQ[q]),
\end{equation}
such that:
\begin{align}
& \Theta([\mathit{vol}]) = 1; \\
& \text{$\Theta$ intertwines \eqref{eq:GMconnectionversionq} with the quantum connection \eqref{eq:quantum-connection}.}
\end{align}
\end{corollary}

For the construction of the Dwork's inverse Frobenius, our ground ring is the single-variable Tate algebra (the $p$-adic completion of $K[t]$):
\begin{equation}
K\langle t \rangle = \Big\{ \sum_{i \geq 0} a_k t^k \;:\; a_k \in K, \, \mathrm{val}(a_k) \rightarrow +\infty\Big\}.
\end{equation}
As before, we have a homomorphism from $B$ (Definition \ref{th:b}) to it:
\begin{equation} \label{eq:t-equal}
B \longrightarrow K \langle t \rangle, \quad t_i \longmapsto t \text{ for all $i$.}
\end{equation}
(This can again be viewed as the quotient by the ideal generated by $t_it_j^{-1}-1$. At first sight, one might think that the the closure of the ideal would be relevant; but it is already closed, like any ideal in an affinoid algebra \cite[Prop.~3 on p.~222]{bosch-guentzer-remmert}.) Let $R_{K\langle t \rangle}$ be the space of series 
\begin{equation}\label{eq:RBonevariabledefinition}
\sum_{\substack{k, \, \bfn \\ k \geq w(\bfn)}} a_{k,\bfn} \pi^{w(\bfn)} t^k z^{\bfn}
\end{equation}
with the obvious analogue of \eqref{eq:valuationconditionB}:
\begin{equation} \label{eq:valuationconditionBonevar}
\parbox{35em}{For any $C \in \bR$, there are only finitely many $(k,\bfn)$ with $\mathrm{val}(a_{k,\bfn}) \leq C$;}
\end{equation}
and similarly for $\Omega^*_{K\langle t \rangle}$. Setting $t_i = t$ yields maps linear over \eqref{eq:t-equal},
\begin{equation} \label{eq:b-to-1}
R_B \longrightarrow R_{K\langle t\rangle}, \quad \Omega^*_B \rightarrow \Omega^*_{K\langle t \rangle}.
\end{equation}
The superpotential is 
\begin{equation} \label{eq:one-variable-superpotential}
W_t = \pi \sum_{i=1}^{r} t z^{\mathbf{e}_i}, 
\end{equation}
which gives rise to the twisted de Rham differential $d+dW_t$ and connection $\nabla_{t\partial_t}$, as before. It is true that \eqref{eq:b-to-1} is an algebraic quotient; but like in the case of \eqref{eq:t-equal}, that relies on results from non-archimedean analysis. A more straightforward way is to just repeat the computations from Section \ref{subsec:compute-cohomology} in a single-variable context, leading to a counterpart of Corollary \ref{th:compare-cohomologies}:

\begin{corollary}
The inclusion $\Omega^*_{\bQ[q]} \subset \Omega^*_{K\langle t \rangle}$, with the variables related by \eqref{eq:t-q}, induces an isomorphism (compatible with connections)
\begin{equation} \label{eq:one-variable-compare}
H^n(\Omega^*_{\bQ[q]},d+dW_q) \otimes_{\bQ[q]} K\langle t \rangle \stackrel{\iso}{\longrightarrow}
H^n(\Omega^*_{K\langle t \rangle},d+dW_t).
\end{equation}
\end{corollary}

The Frobenius-twisted version $R_{K\langle t \rangle}^{(p)}$ consists of series
\begin{equation} 
\sum_{\substack{k,\mathbf{n} \\ k \geq p\, w(\mathbf{n})}} a_{k,\mathbf{n}} \pi^{w(\mathbf{n})} t^{k} z^\mathbf{n} 
\end{equation}
with the same condition \eqref{eq:valuationconditionBonevar}. Similarly for $\Omega^{(p),*}_{K\langle t \rangle}$, which carries the differential $d+dW_{t^p}$ and connection $\nabla_{t\partial_t}^{(p)}$. The specialization of Dwork's construction uses multiplication with
\begin{equation} \label{eq:dwork-twist-4}
\begin{aligned}
G_t & = D(\pi t z^{\bfe_1}) \cdots D(\pi t z^{\bfe_r}) \\
& = \sum_{m_1, \dots, m_r \geq 0} d_{m_1} \cdots d_{m_r} \pi^{m_1 + \cdots + m_r}
t^{m_1+\cdots+m_r} z^{m_1 \bfe_1 + \cdots + m_r \bfe_r},
\end{aligned}
\end{equation}
followed by the same map $\mathit{Loc}$ as before, see \eqref{eq:dwork-root-3}. The composition of those two gives a $K\langle t \rangle$-linear chain map, compatible with connections,
\begin{equation} \label{eq:DworkKlanglet} 
\Psi: (\Omega^*_{K\langle t \rangle}, d+dW_t) \longmapsto (\Omega^{(p),*}_{K \langle t \rangle}, d+dW_{t^p}).
\end{equation}
As in Section \ref{subsec:conclusion}, the combination of \eqref{eq:one-variable-mirror} and \eqref{eq:one-variable-compare} turns the cohomology level map induced by \eqref{eq:DworkKlanglet} into a $K\langle t \rangle$-linear map
\begin{equation} \label{eq:single-variable-psi}
\Psi: H^*(M;K\langle t \rangle) \longrightarrow H^*(M;K\langle t \rangle).
\end{equation}
By construction, this fits into a commutative diagram 
\begin{equation} \label{eq:specialize-frobenius}
\xymatrix{
H^*(M;B) \ar[d]_-{\eqref{eq:frob-geometric}} \ar[rrr]^-{\eqref{eq:t-equal}} &&& H^*(M;K\langle t \rangle) 
\ar[d]^-{\eqref{eq:single-variable-psi}}
\\
H^*(M;B) \ar[rrr]_-{\eqref{eq:t-equal}} &&& H^*(M;K\langle t \rangle).
}
\end{equation}

\begin{corollary} \label{th:b-frob}
(i) $\Psi$ is an inverse Frobenius for the quantum connection \eqref{eq:quantum-connection-2}. Explicitly, for that connection and its Frobenius pullback, one has 
\begin{equation} \label{eq:FrobeniuspropertyonevariablePsi}
\nabla_{t\partial_t}^{(p)} \circ \Psi = \Psi \circ \nabla_{t\partial_t}.
\end{equation}
The constant ($t = 0$) term of $\Psi$ is \eqref{eq:gamma-const-term}.

(ii) The formal power series inverse $\Phi = \Psi^{-1}$ is a Frobenius for the quantum connection, as defined in Section \ref{sec:diff-eq}. The constant term of $\Phi$ is \eqref{eq:b-frobenius} with $b = \Gamma_p(TM)$. Moreover, $\Phi(t)$ converges on the open disc of radius $1$.

(iii) Any Frobenius structure for the quantum connection (on a Fano toric variety) converges $p$-adically on the open disc of radius $1$.
\end{corollary}

\begin{proof}
(i) The relation between multivariable and single-variable quantum connections is given by the commutative diagram
\begin{equation}
\xymatrix{ 
H^*(M;B) \ar[rrr]^-{\eqref{eq:t-equal}} \ar[d]_-{\nabla_{t_1\partial_{t_1}} + \cdots + \nabla_{t_r\partial_{t_r}}} &&&
H^*(M;K\langle t \rangle) \ar[d]^-{\nabla_{t\partial_t}}
\\
H^*(M;B) \ar[rrr]_-{\eqref{eq:t-equal}} &&& H^*(M;K\langle t \rangle)
}
\end{equation}
Using that and \eqref{eq:specialize-frobenius}, one can read off the desired properties directly from the previously considered multivariable case (Corollary \ref{th:multivariable} and Theorem \ref{th:constant-term}). 

(ii) All of this a consequence of the results in (i). Since $\Gamma_p(TM)$ is an invertible cohomology class, the constant term of $\Psi$ is an isomorphism. Hence, it makes sense to invert $\Psi$ as a formal power series, which gives our $\Phi$; moreover, the constant term of $\Phi$ is the inverse of that for $\Psi$. Finally, convergence of $\Phi$ is inherited from that for $\Psi$ via Lemma \ref{th:make-frobenius}(iii).

(iii) Follows from (ii) and Lemma \ref{th:general-structure}(iii).
\end{proof}

Corollary \ref{th:b-frob}(iii) leaves us in a bit of a quandary: we have constructed the Dwork Frobenius for the quantum differential equation, and proved that its constant term is given by the $\Gamma_p$-class; but the convergence radius $1$ property arising from that construction carries over to all other Frobenius structures, blurring the significance of the Dwork Frobenius. This will change in the next section, where we prove overconvergence, a property that other Frobenius structures don't usually share.

\section{Overconvergence\label{sec:overconvergent}}

Our final task towards proving Theorem \ref{th:toric} is to show that the Frobenius structure on the quantum cohomology of a Fano toric variety, constructed in Corollary \ref{th:b-frob}, is overconvergent. The Banach spaces involved generalize those in \cite{adolphson-sperber89,dwork74}; the main cohomological computation (Theorem \ref{claim}) is based on methods from \cite[\S 3]{adolphson-sperber89}.

\subsection{The mod $p$ ingredient}
We need a single-variable version of some material from Section \ref{sec:modp}. There is nothing new here, but it's worth while restating the result in the precise form used later on. Take
\begin{equation}
R_{\bF_p[q]} =\Big\{\sum_{\substack{(k,\mathbf{n}) \\ k \geq w(\mathbf{n})}} a_{k,\mathbf{n}}q^{i}z^{\mathbf{n}} \Big\} \subset \bF_p[q, z_j^{\pm 1}],
\end{equation}
and its a filtration $F^sR_{\bF_p[q]}$ given by allowing only monomials $q^k z^{\bfn}$ with $k-w(\bfn) \geq s$. As before, we write $GR_{\bF_p[q]}$ for the associated graded ring; and similarly for differential forms $G\Omega^*_{\bF_p[q]}$.



\begin{lem} \label{lem:adolphsonsperbermodpLb} 
Fix $k_0 \geq w_0 \geq 0$. Take some 
\begin{equation} \label{eq:theta-form}
\alpha = \sum_{w(\bfn) = w_0} a_{\bfn} q^{k_0} z^{\bfn} \, \mathit{vol} \in G\Omega^n_{\bF_p[q]}.
\end{equation}
Let $\Delta \subset N$ be an $\bF_p$-basis set; and $\omega_j$ the basis of $(n-1)$-forms \eqref{eq:standardbasisomeganminusone}. Then one can write
\begin{equation} \label{eq:theta-beta-gamma}
\alpha = \delta + dW_q \wedge \phi
\end{equation}
(in spite of the notation, remember that this is not the wedge product itself, but rather the induced map on the associated graded space) for
\begin{align}
\label{eq:beta-form}
& \delta = \sum_{\substack{\bfn \in \Delta \\ w(\bfn) = w_0}} d_{\bfn} q^{k_0} z^{\bfn} \, \mathit{vol} \in G\Omega^n_{\bF_p[q]}, \\
\label{eq:gamma-form}
& \phi = \!\!\!\!\! \sum_{\substack{1 \leq j \leq n \\ w(\bfn) = w_0-1}} f_{j,\bfn} q^{k_0-1} z^{\bfn} \, \omega_j \in G\Omega^{n-1}_{\bF_p[q]}.
\end{align}
\end{lem}

\begin{proof}
The single-variable version of the argument from Corollary \ref{cor:generatingalgebraicGM2}(i) shows that we have a decomposition \eqref{eq:theta-beta-gamma}, where: $\delta$ lies in the $\bF_p[q]$-subspace with basis $(q^{w(\bfn)}z^{\bfn} \, \mathit{vol})_{\bfn \in \Delta}$; and $\phi \in G\Omega^{n-1}_{\bF_p[q]}$. Moreover, $\delta$ is uniquely determined.

Besides the differential form grading, the complex $(G\Omega^*_{\bF_p[q]}, dW_q)$ has two additional gradings: one by powers of $q$, which the differential increases by $1$; and the other inherited from the original filtration, which is preserved by the differential. Because of the specific form of $\alpha$ prescribed in \eqref{eq:theta-form}, and the uniqueness of $\delta$, it follows that $\delta$ must be of the form \eqref{eq:beta-form}. Moreover, one can throw away those terms in $\phi$ which are of bidegree different than $(k_0-1,k_0-w_0)$ for our two gradings in view of \eqref{eq:theta-beta-gamma}; the remaining part of $\phi$ is as in \eqref{eq:gamma-form}, and still satisfies \eqref{eq:theta-beta-gamma}.
\end{proof}
%
%

\subsection{Banach space setup} 
Fix some $b>0$. Our ground ring is
\begin{equation}
L(b) = \Big\{ \sum_{k \geq 0} a_k t^k \; : \; \mathrm{val}(a_k) - bk \text{ is bounded below}\Big\}.
\end{equation}
These power series converge for $\mathrm{val}(t) >-b$, or equivalently $|t| < p^b$, an open disc of radius greater than $1$. The appropriate Gauss norm on $L(b)$ is
\begin{equation} \label{eq:b-gauss}
\big\|\sum_{k \geq 0} a_k t^k \big\|_b = \mathrm{sup}_k(|a_k| (p^b)^k) = p^{-\mathrm{inf}_k(\mathrm{val}(a_k) + bk)}. 
\end{equation}
Given $c \in \bR$, we write $L(b,c) \subset L(b)$ for the closed ball around the origin of radius $p^{-c}$ with respect to \eqref{eq:b-gauss}. Explicitly,
\begin{equation} \label{eq:lbc}
L(b,c) = \Big\{ \sum_{k \geq 0} a_kt^k\;:\; \mathrm{val}(a_k) \geq bk+c \;\; \Leftrightarrow \;\;
|a_k| (p^b)^k \leq p^{-c} \Big\}.
\end{equation}
The space of functions is correspondingly 
\begin{equation} \label{eq:lwithq}
R_{L(b)} = \Big\{ 
\sum_{k \geq w(\mathbf{n})} a_{k,\bfn} t^k z^\mathbf{n}, \;  \mathrm{val}(a_{k,\mathbf{n}}) - bk \text{ is bounded below} \Big \}.
\end{equation}
One defines a norm $\|\cdot\|_b$ as before, making $R_{L(b)}$ into a Banach algebra. Define subspaces $R_{L(b,c)} \subset R_{L(b)}$ as in \eqref{eq:lbc}. To understand those subspaces better, one can write them as direct products by separating out the powers of $t$:
\begin{align}\label{eqn:prod-R-L(b,c)}
R_{L(b,c)} = \prod_{k \geq 0} t^k R_{L(b,c),k}, \qquad R_{L(b,c),k}= \Big\{ 
\sum_{k \geq w(\mathbf{n})} a_{k,\bfn}  z^\mathbf{n}, \;  \mathrm{val}(a_{k,\mathbf{n}}) \geq bk+c\Big\}.
\end{align}
Note that the elements of $R_{L(b,c),k}$ are finite sums (Laurent polynomials). The same works for differential forms, giving spaces $\Omega^*_{L(b)}$ and subspaces $\Omega^*_{L(b,c)}$, both of which are closed under the de Rham differential (which as usual applies to the $z$ variables only). The superpotential \eqref{eq:one-variable-superpotential} satisfies
\begin{equation}
W_t \in R_{L(b,\frac{1}{p-1} - b)}.
\end{equation}
From the definitions, we see that
\begin{align} 
\label{eq:algebraicincludestoL}
& \Omega^*_{\bQ[q]} \subset \Omega^*_{L(b)} \text{ for all $b$, with the usual change of variables \eqref{eq:t-q}}; \\
&
\label{eq:Lincludestob1}
\Omega^*_{L(b)} \subset \Omega^*_{K\langle t \rangle}, \text{ assuming that $b > 1/(p-1)$.}
\end{align}

\subsection{Cohomology computation}
Throughout the following part, we assume that
\begin{equation} \label{eq:bbiggerthan}
b > 1/(p-1).
\end{equation}

\begin{thm} \label{claim}
(i) The inclusion \eqref{eq:Lincludestob1} induces an isomorphism
\begin{equation} \label{eq:Lincludestob1iso}
H^n(\Omega^*_{L(b)},d+dW_t) \otimes_{L(b)} K\langle t\rangle \cong H^n(\Omega^*_{K \langle t \rangle},d+dW_t).
\end{equation}

(ii) The inclusion \eqref{eq:algebraicincludestoL} induces an isomorphism:
\begin{equation} \label{eq:algebraicincludestoLiso} 
H^n(\Omega^*_{\bQ[q]},d+dW_q)\otimes_{\bQ[q]} L(b) \cong H^n(\Omega^*_{L(b)}, d+ dW_t).  
\end{equation}
\end{thm}

Fix an $\bF_p$-basis set $\Delta$. We define 
\begin{equation}
V(b) = \bigoplus_{\bfn \in \Delta} L(b)\cdot t^{w(\bfn)} z^{\bfn} \mathit{vol}
\subset \Omega^n_{L(b)},
\end{equation}
and
\begin{equation}
V(b,c) = V(b) \cap \Omega^n_{L(b,c)} = \Big\{
\sum_{\substack{\bfn \in \Delta \\ k \geq w(\bfn)}} a_{k,\bfn} t^k z^{\bfn} \mathit{vol}, \;
\mathrm{val}(a_{k,\bfn}) \geq bk + c\Big\} \subset \Omega^n_{L(b,c)}.
\end{equation}
The following preparatory result should be compared with \cite[Proposition 3.2, Corollary 3.3]{adolphson-sperber89}.

\begin{lemma} \label{lem:polynomialHi} 
Fix $k_0 \geq 0$. Take 
\begin{equation}
\lambda = \sum_{w(\bfn) \leq k_0} l_\bfn t^{k_0} z^{\bfn} \, \mathit{vol} \in \Omega^n_{L(b,c)}.
\end{equation}
Then there are 
\begin{align}
& \rho = \sum_{\substack{\bfn \in \Delta \\ w(\bfn) \leq k_0}} r_{\bfn} t^{k_0} z^{\bfn} \, \mathit{vol} \in V(b,c), \\
& \sigma = \sum_{\substack{1 \leq j \leq n \\ w(\bfn) \leq k_0 - 1}} s_{j,\bfn} t^{k_0-1} z^{\bfn} \, \omega_j \in \Omega^{n-1}_{L(b,b+c-1/(p-1))},
\end{align}
such that
\begin{equation}
\lambda - \rho - dW_t \wedge \sigma \in \Omega^n_{L(b,c+1/(p-1))}.
\end{equation}
\end{lemma}

\begin{proof} 
Suppose that all the nonzero terms in $\lambda$ belong to monomials $t^{k_0} z^{\bfn}$ with $w(\bfn) \leq w_0$, for some $w_0 \leq k_0$. In other words,
\begin{equation}
\lambda = \sum_{w(\bfn) \leq w_0} l_\bfn t^{k_0} z^{\bfn} \, \mathit{vol} \in \Omega^n_{L(b,c)}.
\end{equation}
Suppose also that $w_0$ is the smallest number with that property (there is no such $w_0$ for $\lambda = 0$, but in that case there is nothing to prove anyway). Let 
\begin{equation} \label{eq:v-bound}
v = \mathrm{min}_{w(\bfn) = w_0}(\mathrm{val}(l_{\bfn})) \geq bk_0 + c,
\end{equation}
where the inequality comes from the assumption $\lambda \in L(b,c)$. We take the mod $p$ reductions 
\begin{equation}
a_{\bfn} = \overline{\pi^{-(p-1)v}l_{\bfn}} \in \bF_p, \;\; w(\bfn) = w_0,
\end{equation}
as the coefficients of $\alpha$ from \eqref{eq:theta-form}. Apply Lemma \ref{lem:adolphsonsperbermodpLb} to that $\alpha$, which yields $(\delta, \phi)$ as in \eqref{eq:beta-form}, \eqref{eq:gamma-form}. Lift the coefficients of those forms back to characteristic zero, which means choosing 
\begin{align}
\label{eq:mu-form}
& \rho = \sum_{\substack{\bfn \in \Delta \\ w(\bfn) = w_0}} r_{\bfn} t^{k_0} z^{\bfn}\, \mathit{vol} 
&& \begin{aligned}
& \text{where } \mathrm{val}(r_{\bfn}) \geq v \\
& \text{and } \overline{\pi^{-(p-1)v}r_{\bfn}} = d_{\bfn};
\end{aligned}
\\ \label{eq:nu-form}
& \sigma = \sum_{w(\bfn) = w_0-1} s_{j,\bfn} t^{k_0-1} z^{\bfn}\, \omega_j && 
\begin{aligned}
& \text{where }\mathrm{val}(s_{j,\bfn}) \geq v-1/(p-1) \\ 
& \text{and } \overline{\pi^{-(p-1)v+1} s_{j,\bfn}} = f_{\bfn}.
\end{aligned}
\end{align}
Because of \eqref{eq:v-bound}, we have $\rho \in V(b,c)$ and $\sigma \in \Omega^{n-1}_{L(b,b+c-1/(p-1))}$. Consider
\begin{equation} \label{eq:lambda-difference}
\lambda - \rho - dW_t \wedge \sigma \in \Omega^n_{L(b,c)}.
\end{equation}
In this expression, the only nonzero monomials which occur are again of the form $t^{k_0}z^{\bfn}$ for $w(\bfn) \leq w_0$. Crucially, the mod $p$ equality \eqref{eq:theta-beta-gamma} implies the following:
\begin{equation} \label{eq:crucial}
\parbox{35em}{for those monomials in \eqref{eq:lambda-difference} where $w(\bfn) = w_0$, the associated coefficient has valuation which is strictly larger than $v$, hence $\geq v + 1/(p-1) \geq bk_0 + c + 1/(p-1)$.}
\end{equation}
In order to get from \eqref{eq:theta-beta-gamma} to this statement, the following two observations are relevant:
\begin{itemize} \itemsep.5em
\item
There is a discrepancy between $W_t$ and $W_q$, because the first one has a $\pi$ while the second doesn't. In our construction, that is accounted for by using one extra power of $\pi^{-1}$ in \eqref{eq:nu-form} compared to \eqref{eq:mu-form}. 
\item
Lemma \ref{lem:adolphsonsperbermodpLb} works not in the algebra $\Omega^*_{\bF_p[q]}$ of differential forms, but in the associated graded of the filtration $F^s$, in this case for $s = k_0 - w_0$. However, if one forms $\alpha - \delta - dW_q$ in $\Omega^*_{\bF_p[q]}$, then any terms lying in $F^s$ for $s>k_0 - w_0$ involve monomials with $w(\bfn) < w_0$, which do not make any difference to the validity of \eqref{eq:crucial}.
\end{itemize}
One can reformulate \eqref{eq:crucial} as follows. Let $\lambda' \in \Omega^n_{L(b,c)}$ be the sum of all terms in \eqref{eq:lambda-difference} with $w(\bfn) < w_0$. Then
\begin{equation}
\lambda - \rho - dW_t \wedge \sigma \in \lambda' + \Omega^n_{L(b,c+1/(p-1))}.
\end{equation}
One can now repeat the same argument with $\lambda'$ instead of $\lambda$. Since $w_0$ decreases in each step, we eventually reach the desired goal.
\end{proof}

\begin{lemma} \label{prop:generationHi} 
We have
\begin{equation}
\Omega^n_{L(b, c)} = V(b, c) + dW_t \wedge \Omega^{n-1}_{L(b, c + b - 1/(p-1))}.
\end{equation}
\end{lemma} 

\begin{proof} 
Given $\lambda = \lambda_0 \in \Omega^n_{L(b,c)}$, in view of \eqref{eqn:prod-R-L(b,c)}, we can apply Lemma \ref{lem:polynomialHi} to each $t^{k_0}$ component of $\lambda_0$ separately, and then form the series consisting of the results. The outcome are 
\begin{equation}
\begin{aligned}
& \rho_0 \in V(b,c), \quad \sigma_0 \in \Omega^{n-1}_{L(b, c + b - 1/(p-1))}, \\
& \text{such that } \lambda_1 = \lambda_0 - \rho_0 - dW_t \wedge \sigma_0 \in \Omega^n_{L(b,c+1/(p-1))}.
\end{aligned}
\end{equation}
One can iterate the argument, producing sequences $(\lambda_r,\rho_r,\sigma_r)$ with
\begin{equation}
\begin{aligned}
& \textstyle \rho_r \in V(b,c + \frac{r}{p-1}), \quad \sigma_r \in \Omega^{n-1}_{L(b,c + b + \frac{r-1}{p-1})}, \\ 
& \text{such that } \lambda_{r+1} = \lambda_r - \rho_r - dW_t \wedge \sigma_r \in \Omega^n_{L(b,c+\frac{r-1}{p-1})}.
\end{aligned}
\end{equation}
As a consequence (both infinite series being absolutely convergent, and $dW_t \wedge$ being continuous),
\begin{equation}
\lambda = \sum_{r\geq 0} \rho_r + dW_t \wedge \sum_{r \geq 0} \sigma_r.
\end{equation}
%
\end{proof}
  
\begin{lemma} \label{prop:generationDi} 
We have
\begin{equation}
\Omega^n_{L(b, c)} = V(b, c) + (d+dW_t) \Omega^{n-1}_{L(b, c + b - 1/(p-1))}.
\end{equation}
\end{lemma}

\begin{proof}  
Given $\lambda = \lambda_0 \in \Omega^n_{L(b,c)}$, use Lemma \ref{prop:generationHi} to find 
\begin{equation}
\begin{aligned}
& \rho_0 \in V(b,c), \quad \sigma_0 \in \Omega^{n-1}_{L(b,c+b-1/(p-1))}, \\
& \text{such that } \lambda_0 = \rho_0 + dW_t \wedge \sigma_0, \\
& \text{and therefore } \lambda_1 = \lambda_0 - \rho_0 - (d+dW_t)(\sigma_0) = -d\sigma_0 
\in \Omega^n_{L(b,c+b-1/(p-1))}.
\end{aligned}
\end{equation}
We can iterate this to produce $(\lambda_r,\rho_r,\sigma_r)$ with
\begin{equation}
\begin{aligned}
& \rho_r \in V(b,c + r(b-1/(p-1))), \quad \sigma_r \in \Omega^{n-1}_{L(b,c+(r+1)(b-1/(p-1)))}, \\
& \text{such that } \lambda_r = \rho_r + dW_t \wedge \sigma_r, \\
& \text{and therefore } \lambda_{r+1} = \lambda_r - \rho_r - (d+dW_t)(\sigma_r) = -d\sigma_r 
\in \Omega^n_{L(b,c+ (r+1)(b-1/(p-1)))}.
\end{aligned}
\end{equation}
As in the proof of the previous Lemma, this yields
\begin{equation}
\lambda = \sum_{r \geq 0} \rho_r + (d+dW_t \wedge) \sum_{r \geq 0} \sigma_r.
\end{equation}
Note that this time, the convergence of these series depends on \eqref{eq:bbiggerthan}.
\end{proof}

\begin{proof}[Proof of Theorem \ref{claim}]
(i) Consider the diagram
\begin{equation}
\xymatrix{
V(b) \otimes_{L(b)} K\langle t \rangle \ar@{->>}[rr] \ar@{=}[d] 
&&
H^n(\Omega^*_{L(b)}, d+dW_t) \otimes_{L(b)} K\langle t \rangle \ar[d]^-{\eqref{eq:Lincludestob1iso}}
\\
\bigoplus_{\bfn \in \Delta} K\langle t \rangle \, z^{\bfn} \mathit{vol} \ar[rr]_{\iso}
&&
H^n(\Omega^*_{K\langle t \rangle}, d+dW_t)
}
\end{equation}
The horizontal maps are projections from cocycles to cohomology. The lower $\rightarrow$ is an isomorphism by the single-variable version of Proposition \ref{prop:AS31q}; and the upper $\rightarrow$ is onto by Lemma \ref{prop:generationDi}. It follows that the upper $\rightarrow$ must also be injective, and therefore \eqref{eq:Lincludestob1iso} must be an isomorphism as well.

(ii) It follows from part (i) that the natural map \begin{align}\label{eq:VbisLbcohomology} V(b) \longrightarrow H^n(\Omega^*_{L(b)}, d+dW_t) \end{align} is injective. By Lemma \ref{prop:generationDi}, it is also surjective and hence an isomorphism. Now consider the diagram
\begin{equation}
\xymatrix{
\bigoplus_{\bfn \in \Delta} L(b) \, z^{\bfn} \ar@{=}[d] \ar[rr]^-{\iso}
&&
H^n(\Omega_{\bQ[q]}, d+dW_q) \otimes_{\bQ[q]} L(b) \ar[d]^-{\eqref{eq:algebraicincludestoLiso}}
\\
V(b) \ar[rr]^-{\iso}_-{\eqref{eq:VbisLbcohomology}}
&&
H^n(\Omega^*_{L(b)}, d+dW_t)
}
\end{equation}
The upper $\rightarrow$ is an isomorphism by the single-variable version of Proposition \ref{prop:AS31q}; it follows that \eqref{eq:algebraicincludestoLiso} must be an isomorphism as well.
\end{proof}

In more geometric terms, this computation yields the following mirror symmetry result.

\begin{cor} \label{cor:toricmirrorsymmetryoverS}
Assuming \eqref{eq:bbiggerthan}, there is an isomorphism
\begin{align} \label{eq:mirrornoneqisoLp} 
\Theta: H^n(\Omega_{L(b)}^*,d+dW_t) \stackrel{\iso}{\longrightarrow} H^*(M;L(b)),
\end{align}
intertwining the connections $\nabla_{t\partial_t}$ on both sides.
\end{cor}

\subsection{Dwork's inverse Frobenius reconsidered}
For a careful choice of parameter $b$ for the spaces $L(b,c)$ and their Frobenius pullbacks, one can carry over the construction from Section \ref{subsec:dwork-frobenius} to the overconvergent context. First, we need to introduce the Frobenius pullback spaces. For $b>0$,
\begin{equation}
R^{(p)}_{L(b/p)} = \Big\{ \sum_{k \geq p\, w(\bfn)} a_{k,\bfn} t^k z^{\bfn}, \;
\mathrm{val}(a_{k,\bfn}) - kb \text{ is bounded below} \Big\}.
\end{equation}
A usual, we write $\Omega^{(p),*}_{L(b/p)}$ for differential forms with coefficients in $R^{(p)}_{L(b/p)}$. As in \eqref{eq:BpisFrobeniuspullback}, the map $F(t) = t^p$ induces an isomorphism
\begin{equation} \label{eq:LpisFrobeniuspullback}
\Omega^*_{L(b)}[t^{1/p}] \stackrel{\iso}{\longrightarrow} \Omega^{(p),*}_{L(b/p)}.
\end{equation}
One can use that to compute the cohomology of $H^n(\Omega^{(p),*}_{L(b/p)},d+dW_{t^p})$, by reducing it to Theorem \ref{claim}. In particular, assuming \eqref{eq:bbiggerthan}, one can define an isomorphism $\Theta^{(p)}$ which sits in a commutative diagram with that from Corollary \ref{cor:toricmirrorsymmetryoverS}:
\begin{equation}
\xymatrix{
\ar[d]_-{\Theta^{(p)}}^-{\iso} H^n(\Omega^{(p),*}_{L(b/p)}, d+dW_{t^p})
&& \ar[ll]^-{\iso}_-{\eqref{eq:LpisFrobeniuspullback}}
H^n(\Omega^*_{L(b)}[t^{1/p}], d+dW_t) 
\ar[d]^-{\Theta}_-{\iso}
\\
H^*(M;L(b/p)) 
&& \ar[ll]^-{\iso}_-{F}
H^*(M;L(b)[t^{1/p}])
}
\end{equation}

%
%
%
%

Formally, the construction of the inverse Frobenius proceeds exactly as before. Consider the function \eqref{eq:dwork-twist-4}. It follows from \eqref{eq:dwork-exponential-val} that $D(\pi tz^{\bfe_i}) \in L(\frac{p-1}{p^2})$ and hence
\begin{equation}
G_t \in L({\textstyle\frac{p-1}{p^2}}).
\end{equation}
Next, the same formula as in \eqref{eq:dwork-root-3} defines a map
\begin{equation} \label{eq:dwork-root-4} 
\mathit{Loc}_t: \Omega^*_{L(b)} \longrightarrow \Omega^{(p),*}_{L(b)}
\end{equation}
for any $b$ (and in particular $b = (p-1)/p^2$). The composition of the two yields the inverse Frobenius 
\begin{equation} \label{eq:FrobeniusLb}
\Psi: \Omega^*_{L\big(\frac{p-1}{p}\big)} \hookrightarrow \Omega^*_{L\big(\frac{p-1}{p^2}\big)} \xrightarrow{G_t}
\Omega^*_{L\big(\frac{p-1}{p^2}\big)} \xrightarrow{\eqref{eq:dwork-root-4}} \Omega^{(p),*}_{L(\frac{p-1}{p^2})}.
\end{equation}
The basic properties (compatibility with differentials and with connections) are proved by the same computations as before.

\begin{proof}[Proof of Theorem \ref{th:toric}]
There is an inclusion which is the Frobenius pullback of \eqref{eq:Lincludestob1},
\begin{align} \label{eq:Lincludestob} 
(\Omega^{(p),*}_{L(b/p)},d+dW_{t^p}) \to (\Omega^{(p),*}_{K\langle t \rangle}, d+ dW_{t^p}), \text{ provided that } b>1/(p-1).
\end{align} 
By construction, the two inverse Frobenius maps sit in a commutative diagram
\begin{equation}
\xymatrix{
H^n(\Omega^*_{L(\frac{p-1}{p})},d+dW_t) \ar[r]^{\eqref{eq:FrobeniusLb}} \ar[d]_{\eqref{eq:Lincludestob1}} & H^n(\Omega^{(p),*}_{L(\frac{p-1}{p^2})},d+dW_{t^{p}}) \ar[d]^{\eqref{eq:Lincludestob}} \\
H^n(\Omega^*_{K\langle t \rangle},d+dW_t) \ar[r]^{\eqref{eq:DworkKlanglet}} & H^n(\Omega^{(p),*}_{K \langle t \rangle},d+dW_{t^{p}})
}
\end{equation}
Applying the mirror symmetry isomorphisms  \eqref{eq:one-variable-mirror}, \eqref{eq:one-variable-compare} and \eqref{eq:mirrornoneqisoLp} turns this into
\begin{equation} \label{eq:1more}
\xymatrix{
H^*(M;L(\frac{p-1}{p})) \ar[r]^\Psi \ar[d] & H^*(M;L(\frac{p-1}{p^2})) \ar[d] \\
H^*(M;K\langle t \rangle) \ar[r]^\Psi & H^*(M;K\langle t \rangle)
}
\end{equation}
The top $\rightarrow$ is an inverse Frobenius structure for the quantum connection, which has convergence radius $>1$ by definition, and whose $t = 0$ term is the same as in Corollary \ref{th:b-frob}, because of the compatibility statement \eqref{eq:1more}. We invoke Lemma \ref{th:make-frobenius}(iii) to see that its inverse, which is the Frobenius structure from Conjecture \ref{th:gamma-conjecture}, is also overconvergent.
\end{proof}

\section{Grassmannians\label{sec:grassmannian}}
Take $M = \mathit{Gr}(k,N)$. Let $E \rightarrow M$, $Q \rightarrow M$ be the tautological subbundle and quotient bundle, respectively. We write $r_i$ for the Chern roots of $E^\vee$, so that cohomology classes on $M$ are described by symmetric polynomials in these variables. Note that any symmetric polynomial involving $r_i^N$ gives the trivial cohomology classes (as one can see from the diagram $\bC P^{N-1} \leftarrow \mathit{Fl}(N) \rightarrow \mathit{Gr}(k,N)$ involving the full flag variety). 

We adopt the same strategy as in \cite[Section 6]{galkin-golyshev-iritani16}, based on the ``Satake isomorphism''
\begin{equation} \label{eq:satake}
S: \Lambda^k H^*(\bC P^{N-1};\bQ) \stackrel{\iso}{\longrightarrow} H^{*-k(k-1)}(M;\bQ).
\end{equation}
This has a simple formula in terms of Schubert cycles, see \cite[Section 2]{bertram-ciocan-fontanine-kim05}, \cite[Section 1.6]{golyshev-manivel11}, or \cite[Equation (6.2.4)]{galkin-golyshev-iritani16}. Equivalently, in terms of Schur polynomials (see e.g.\ \cite[Chapters 6, 9]{fulton-young-tableaux} for background)
\begin{equation} \label{eq:satake-schur}
S(x^{d_1} \wedge \cdots \wedge x^{d_k}) = \frac{\mathrm{det}\big( (r_i^{d_j})_{ij} \big)}{\mathrm{det}\big( (r_i^{k-j})_{ij} \big)}.
\end{equation}
Here, $x \in H^2(\bC P^{N-1})$ is the hyperplane class. On the right hand side, both numerator and denominator are determinants of $k \times k$ matrices with coefficients in $\bQ[r_1,\dots,r_k]$; the quotient is a symmetric polynomial. The signs are set up so that
\begin{equation} \label{eq:satake-1}
S(x^{k-1} \wedge x^{k-2} \wedge \cdots \wedge x \wedge 1) = 1.
\end{equation}
In the formula \eqref{eq:satake-schur} we can allow $d_j \geq N$, in which case both sides are zero, because of the previously mentioned observation about powers of the Chern roots.

\begin{lemma} \label{th:satake}
For any $m \geq 1$, the isomorphism \eqref{eq:satake} fits into a commutative diagram
\begin{equation} \label{eq:satake-chern}
\xymatrix{
\Lambda^k H^*(\bC P^{N-1};\bQ) 
\ar_-{\iso}^-{S}[rr] 
\ar[d]_-{\Lambda^k_{\mathrm{Lie}}(\mathit{ch}_m(\scrO(1)) \smile \cdot)} 
&& 
H^{*-k(k-1)}(M;\bQ) 
\ar[d]^-{\mathit{ch}_m(E^\vee) \smile \cdot} 
\\
\Lambda^k H^*(\bC P^{N-1};\bQ) \ar_-{\iso}^-{S}[rr] && H^{*-k(k-1)}(M;\bQ)
}
\end{equation}
Here, the left $\downarrow$ is shorthand for the following: take the cup product with $\mathit{ch}_m(\scrO(1)) = x^m/m!$, as an endomorphism of $H^*(\bC P^{N-1})$; and consider its exterior product 
$\Lambda^k_{\mathrm{Lie}}$ in the sense of Lie algebra representations,
\begin{equation} 
x^{d_1} \wedge \cdots \wedge x^{d_k} \longmapsto \frac{1}{m!} \sum_l
x^{d_1} \wedge \cdots \wedge x^{d_l+m} \wedge \cdots \wedge x^{d_k}.
\end{equation}
The right $\downarrow$ is cup product with $\mathit{ch}_m(E^\vee) = \frac{1}{m!} (r_1^m + \cdots + r_k^m)$. 
\end{lemma}

\begin{proof}
We write out the Schur polynomials using Leibniz' formula, and then directly look at the effect on the numerator of \eqref{eq:satake-schur} following the order left $\downarrow$ then bottom $\rightarrow$:
\begin{equation}
\sum_{\sigma,l} (-1)^{\mathrm{sign}(\sigma)} r_{\sigma(1)}^{d_1} \cdots r_{\sigma(l)}^{d_l+m} \cdots r_{\sigma(k)}^{d_k} = 
(r_1^m + \cdots + r_k^m) \sum_{\sigma} (-1)^{\mathrm{sign}(\sigma)} r_{\sigma(1)}^{d_1} \cdots r_{\sigma(k)}^{d_k}.
\end{equation}
\end{proof}

We apply this to the $p$-adic Gamma classes:

\begin{lemma} \label{th:gamma-grassmannian}
Let $\Phi_{\bC P^{N-1},0}$ and $\Phi_{M,0}$ be the endomorphisms \eqref{eq:b-frobenius} for those two manifolds. They fit into a commutative diagram
\begin{equation}
\xymatrix{
\Lambda^k H^*(\bC P^{N-1};\bQ_p) 
\ar_-{\iso}^-{S}[rr] 
\ar[d]_-{\Lambda^k_{\mathrm{Group}}(\Phi_{\bC P^{N-1},0})} 
&& 
H^{*-k(k-1)}(M;\bQ_p) 
\ar[d]^-{p^{-k(k+1)/2} \Phi_{M,0}} 
\\
\Lambda^k H^*(\bC P^{N-1};\bQ_p) \ar_-{\iso}^-{S}[rr] && H^{*-k(k-1)}(M;\bQ_p)
}
\end{equation}
Here $\Lambda^k_{\mathrm{Group}}(\Phi_{\bC P^{N-1},0})$ is the exterior power in the sense of group representations, meaning the restriction of $\Phi_{\bC P^{N-1},0} \otimes \cdots \otimes \Phi_{\bC P^{N-1},0}$ to the antisymmetric part.
\end{lemma}

\begin{proof}
As a virtual bundle,
\begin{equation}
TM = \mathit{Hom}(E,Q) \iso \bC^N \otimes Q - \mathit{Hom}(Q,Q).
\end{equation}
Since $\mathit{Hom}(Q,Q)$ is self-dual, it has vanishing odd Chern classes. The additive characteristic class $\log\,\Gamma_p$ only uses the odd components of the Chern character, see \eqref{eq:log-gamma-class}, and therefore
\begin{equation} \label{eq:gamma-grassmannian}
\log\,\Gamma_p(TM) = N \log\,\Gamma_p(Q) = -N \log\, \Gamma_p(E) = N \log\,\Gamma_p(E^\vee);
\end{equation}
the $k = 1$ special case being
\begin{equation}
\log\,\Gamma_p(T\bC P^{N-1}) = N \log\,\Gamma_p(\scrO(1)).
\end{equation}
By applying Lemma \ref{th:satake} to each term in \eqref{eq:log-gamma}, one therefore gets 
\begin{equation} \label{eq:satake-chern-2}
\xymatrix{
\Lambda^k H^*(\bC P^{N-1};\bQ_p) 
\ar_-{\iso}^-{S}[rr] 
\ar[d]_-{\Lambda^k_{\mathrm{Lie}}(\log\,\Gamma_p(T\bC P^{N-1}) \smile \cdot)}
&& 
H^{*-k(k-1)}(M;\bQ_p) 
\ar[d]^-{\log\, \Gamma_p(TM) \smile \cdot} 
\\
\Lambda^k H^*(\bC P^{N-1};\bQ_p) \ar_-{\iso}^-{S}[rr] && H^{*-k(k-1)}(M;\bQ_p)
}
\end{equation}
Exponentiating the $\downarrow$ yields a similar diagram involving $\Lambda^k_{\mathrm{Group}}(\Gamma_p(T\bC P^{N-1}) \smile \cdot)$ respectively $\Gamma_p(TM) \smile \cdot$. 
The desired result is then obtained by inserting the $p^{-\mathrm{deg}/2}$ terms; the resulting discrepancy comes from the grading shift in \eqref{eq:satake}.
\end{proof}

There is a quantum version of Lemma \ref{th:satake}: the quantum Pieri formula \cite{bertram97}, see also \cite[Proposition 6.2.3, 6.2.5]{galkin-golyshev-iritani16}. We only need the $m=1$ case:

\begin{lemma} \label{th:quantum-satake}
The isomorphism \eqref{eq:satake} fits into a commutative diagram
\begin{equation} \label{eq:quantum-satake}
\xymatrix{
\Lambda^k H^*(\bC P^{N-1};\bQ)[q]
\ar_-{\iso}^{S}[rr] 
\ar[d]_-{\Lambda^k_{\mathit{Lie}}(c_1(\scrO(1)) \ast_{\zeta q} \cdot)} 
&& 
H^{*-k(k-1)}(M;\bQ)[q]
\ar[d]^-{c_1(E^\vee) \ast_q \cdot} 
\\
\Lambda^k H^*(\bC P^{N-1};\bQ)[q] \ar_-{\iso}[rr]^{S} && H^{*-k(k-1)}(M;\bQ)[q]
}
\end{equation}
In the left $\downarrow$, we take the endomorphism of $H^*(\bC P^{N-1};\bQ)[q]$ given by quantum multiplication with the first Chern class, except for applying a substitution 
\begin{equation} \label{eq:satake-substitution}
q \longmapsto \zeta q, \;\; \text{ where } \zeta^N = (-1)^{k-1}
\end{equation}
(note that the quantum product only involves $q^N$; so no actual roots of unity arise, it's just a matter of changing signs if $k$ is even). We then let that endomorphism act on the exterior product as in \eqref{eq:satake-chern}. 
\end{lemma}

\begin{example}
As a check on the sign \eqref{eq:satake-substitution}, one can use the geometrically uninteresting case $k = N-1$. Under the identification $M = \bC P^{N-1}$, $c_1(E^\vee) = x$. By starting with \eqref{eq:satake-1} and repeatedly applying the $m=1$ case of Lemma \ref{th:satake}, one sees that the Satake isomorphism is
\begin{equation}
S(x^{N-1} \wedge \cdots \wedge x^{i+1} \wedge x^{i-1} \wedge \cdots 1) = x^{N-1-i}.
\end{equation}
On the left column of \eqref{eq:quantum-satake} one has 
\begin{equation} \label{eq:trivial-quantum-satake}
\Lambda_{\mathit{Lie}}^{N-1}(x \ast_q \cdot)\big(
x^{N-1} \wedge \cdots \wedge x^{i+1} \wedge x^{i-1} \wedge \cdots
\big) = 
\begin{cases}
x^{N-1} \wedge \cdots \wedge x^i \wedge x^{i-2} \wedge \cdots 
& i>0, \\
(-1)^{N-2} q^N ( x^{N-2} \wedge \cdots \wedge 1) & i = 0;
\end{cases}
\end{equation}
the change of variables \eqref{eq:satake-substitution} makes the extra sign in \eqref{eq:trivial-quantum-satake} disappear; which matches what happens to the $x^{N-1-i}$ in the right column.
\end{example}

As a consequence, the quantum connections on $\Lambda^k H^*(\bC P^{N-1})$ and on $H^*(M)$ match if, in the latter, one changes coordinates as in \eqref{eq:satake-substitution}. Combining that with Lemma \ref{th:gamma-grassmannian} yields:

\begin{corollary}
Take the Frobenius structures $\Phi_{\bC P^{N-1}}$ and $\Phi_M$, with constant terms as in Lemma \ref{th:gamma-grassmannian}. These fit into a commutative diagram
\begin{equation} \label{eq:frobenius-satake}
\xymatrix{
\Lambda^k H^*(\bC P^{N-1};K)[[t]]
\ar_-{\iso}^{S}[rr] 
\ar[d]_-{\Lambda^k_{\mathit{Group}}(\Phi_{\bC P^{N-1}}(\zeta t))} 
&& 
H^{*-k(k-1)}(M;K)[[t]]
\ar[d]^-{p^{-k(k+1)/2}\Phi_M(t)} 
\\
\Lambda^k H^*(\bC P^{N-1};K)[[t]] \ar_-{\iso}[rr]^{S} && H^{*-k(k-1)}(M;K)[[t]]
}
\end{equation}
Here, as in Lemma \ref{th:quantum-satake}, the root of unity $\zeta$ only results in a change of sign, since both Frobenius structures are functions of $t^N$ (this explains why we can work over $K$, without worrying whether it contains $\zeta$).
\end{corollary}

Hence, overconvergence of $\Phi_{\bC P^{N-1}}$ (a special case of Theorem \ref{th:toric}) implies the same for $\Phi_M$. 

As mentioned in the Introduction, Conjecture \ref{th:gamma-conjecture-2} was proved in \cite{sperber80} for $\bC P^{N-1}$, assuming that $p \geq N+2$. To be precise, \cite[Theorem 2.35]{sperber80} (see also the explanation on p.~5 of the paper) concerns the connection 
\begin{equation} \label{eq:s-connection}
s\partial_s + \begin{pmatrix} 0 & &&& \!\!\!\!\pi^N s \\ 1 & 0 && \\ & 1 & 0 && \\ && \cdots && \\ 
&&& 1 & 0 \end{pmatrix}
\end{equation}
which transforms into \eqref{eq:quantum-connection-2} for $\bC P^n$ after setting $s = t^N$. The value $t = \zeta$ for the quantum connection corresponds to $s = (-1)^{k-1}$. The computation of Frobenius eigenvalues for \eqref{eq:s-connection} applies to $s = \pm 1$, since those are both $(p-1)$-st roots of unity. The outcome, in our terminology, is that the eigenvalues of $\Phi_{\bC P^{N-1}}(\zeta)$ have the same valuations as those of $G_{\bC P^{N+1}}(x) = p^{-\mathrm{deg}(x)/2} x$. From \eqref{eq:frobenius-satake}, it follows that the eigenvalues of $\Phi_M(1)$ have the same valuations as those of $p^{k(k+1)/2} \Lambda^k_{\mathit{Group}}(G_{\bC P^{N-1}}) = G_M$. This concludes the proof of Theorem \ref{th:grassmannian}.

\appendix

\section{Numerical evidence}
Computations in this section (see \cite{seidel-url} for code) are done with respect to the variable $q$ for the quantum connection, so the Frobenius is \eqref{eq:frobenius-2}. Remember that in this context, overconvergence means that the radius of convergence of $\bar\Phi(q)$ is $> p^{-1/(p-1)}$. 

The practical aspects of the computation are fairly straightforward. We have
\begin{equation}
\Gamma_p(\mathit{TM}) = \gamma^1 b^1 + \cdots + \gamma^r b^r \in H^{\mathrm{even}}(M;\bQ_p),
\end{equation}
where each $b^k$ is a monomial in the Chern classes of $M$ (including $c_0(TM) = 1$), and each $\gamma^k$ is a rational polynomial in the derivatives $\Gamma^{(j)}_p(0)$, $j \leq \mathrm{dim}_{\bC}(M)$. We set $b = b^k$ in \eqref{eq:b-frobenius} and consider the resulting solutions $\bar\Phi^k(q)$ of \eqref{eq:frobenius-2}. The desired Frobenius structure is the same linear combination
\begin{equation}
\bar\Phi(q) = \gamma^1 \bar\Phi^1(q) + \cdots + \gamma^r \bar\Phi^r(q).
\end{equation}

We want to compute the valuation of the $q^m$-coefficients $\bar\Phi_m$ of $\bar\Phi(q)$, for all $m$ up to some fixed bound $N$. The corresponding terms $\bar\Phi^k_m$ of $\bar\Phi^k(q)$ are rational matrices, and can be computed exactly using \eqref{eq:order-by-order-2}. Suppose that, after having done this computation, we find that the matrices all have $p$-adic valuation $\geq H$. Compute rational approximations $\tilde{\gamma}^k \in \bQ$ such that $\mathrm{val}(\tilde{\gamma}^k - \gamma^k) \geq K$, which we can do for any $K$ using Lemma \ref{th:approximate-derivatives}. We then know that $\bar\Phi_m - (\tilde{\gamma}^1 \bar\Phi^1_m + \cdots + \tilde{\gamma}^r \bar\Phi^r_m)$ has valuation $\geq K+H$, and hence
\begin{equation} \label{eq:approximate-valuation}
\mathrm{val}(\tilde{\gamma}^1 \bar\Phi^1_m + \cdots + \tilde{\gamma}^r \bar\Phi^r_m) < K+H
\;\;\Rightarrow \;\;\mathrm{val}(\bar\Phi_m) = \mathrm{val}(\tilde{\gamma}^1 \bar\Phi^1_m + \cdots + \tilde{\gamma}^r \bar\Phi^r_m).
\end{equation}
It suffices to choose $K$ sufficiently large so that the condition in \eqref{eq:approximate-valuation} kicks in, and that computes $\mathrm{val}(\bar\Phi_m)$.

In the examples below, we have plotted $(x,y) = (m,-\val(\bar\Phi_m))$, for $p = 3,5$. For comparison the dashed line is $y = x/(p-1)$, so overconvergence means that our graph has a smaller linear growth rate than that. Obviously, computations up to a finite order in $q$ can never prove overconvergence, not even for a single prime $p$. However, the gap in growth rates is sufficiently large to be considered serious supporting evidence for Conjecture \ref{th:gamma-conjecture}. Moreover, our computations point towards the growth rate being $\leq (2p-1)/(p^3-p^2)$, which would correspond to the same radius of convergence as $D(q)$; however, equality does not always hold (Examples \ref{th:cp2} and \ref{th:cubic-surface} with $p = 3$).

The other aspect concerns the valuations of the eigenvalues of $\bar\Phi(\pi)$, which can be computed within $\bQ_p$ using Lemma \ref{th:newton} applied to \eqref{eq:characteristic-polynomial-2}, and \eqref{eq:fractional-valuation}. In our context, where we only have a finite order approximation to $\bar\Phi(q)$, any such computation is tentative, since it could be in principle be disrupted by higher order terms with low valuation; still, the results match Conjecture \ref{th:gamma-conjecture-2}. 

\begin{example} \label{th:cp2}
Take $M = \bC P^2$. The experimentally observed growth rates are $\sigma \approx 0.16$ for $p = 3$, $\sigma \approx 0.09$ for $p = 5$:
\begin{center}
\includegraphics[scale=0.45]{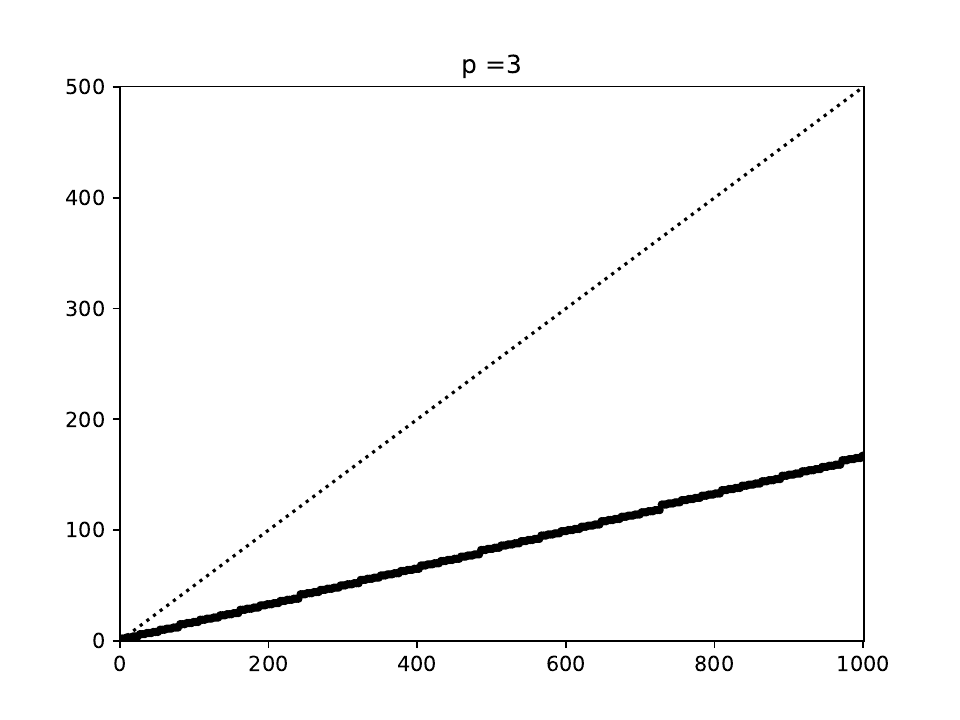}
\includegraphics[scale=0.45]{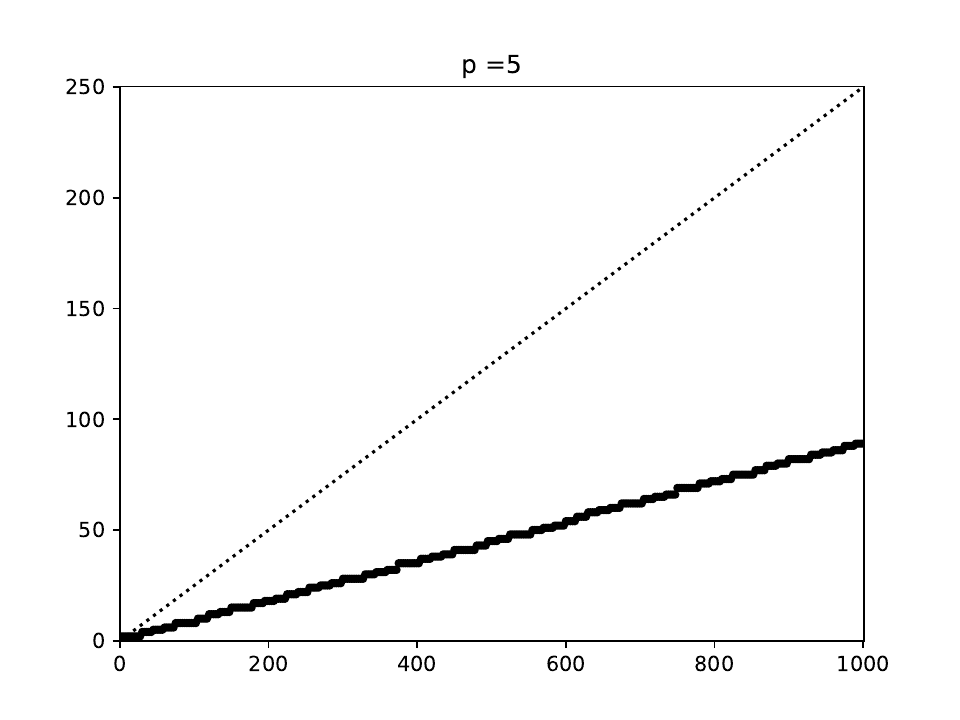}
\end{center}
At both primes, the Newton polygon of our approximation to $\bar\phi(q,\pi)$ has vertices at $(0,-3)$,$(1,-3)$, $(2,-2)$, $(3,0)$; hence slopes $(0,1,2)$.
\end{example}

\begin{example} \label{th:cubic-surface}
Take the cubic surface $M \subset \bC P^3$. We consider the quantum connection only on the subspace spanned by $\{1,\, c_1(TM),\, [point]\} \in H^*(M)$, where it is given by \cite{crauder-miranda94}
\begin{equation}
\partial_q + \begin{pmatrix} 
0 & 108 q & 252 q^2 \\
q^{-1} & 9 & 36 q \\
0 & 3q^{-1} & 0 \end{pmatrix}
\end{equation}
One has $c_1(TM)^2 = 3[\mathit{point}]$; based on that, $\Gamma_p(TM) = (1,\Gamma_p'(0),(3/2)\Gamma_p'(0)^2)$, and the constant term of our Frobenius structure is 
\begin{equation}
\Phi_0 = \begin{pmatrix} 1 & 0 & 0 \\ \Gamma_p'(0) & 1/p & 0 \\ (3/2) \Gamma_p'(0)^2 & 3\Gamma_p'(0)/p & 1/p^2  \end{pmatrix}
\end{equation}
The experimentally observed growth rates are $\sigma \approx 0.16$ for $p = 3$, $\sigma \approx 0.09$ for $p=5$:
\begin{center}
\includegraphics[scale=0.45]{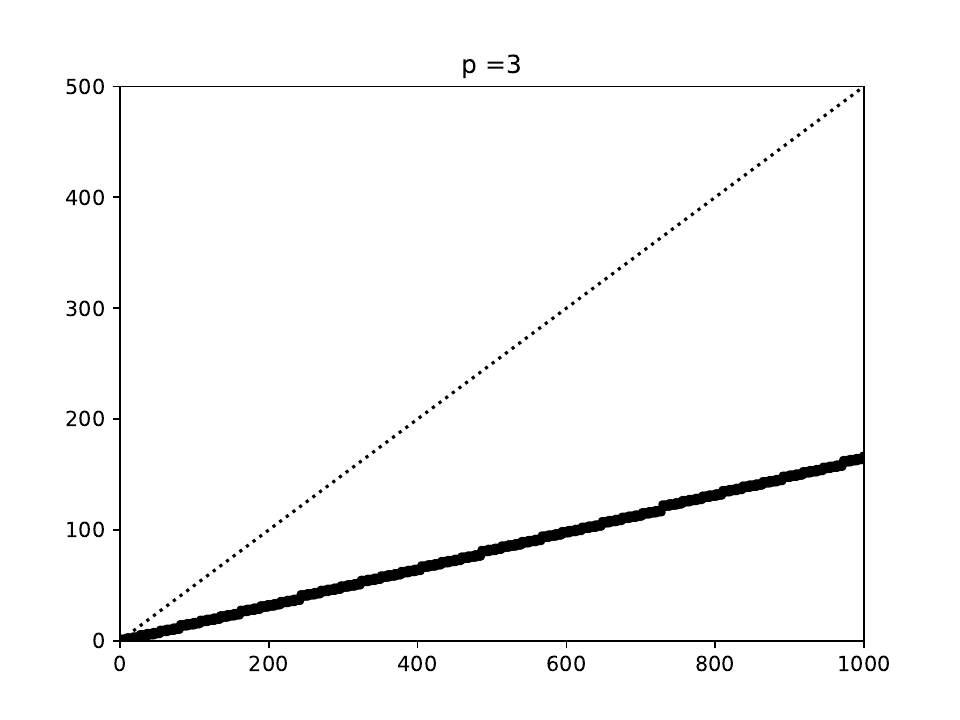}
\includegraphics[scale=0.45]{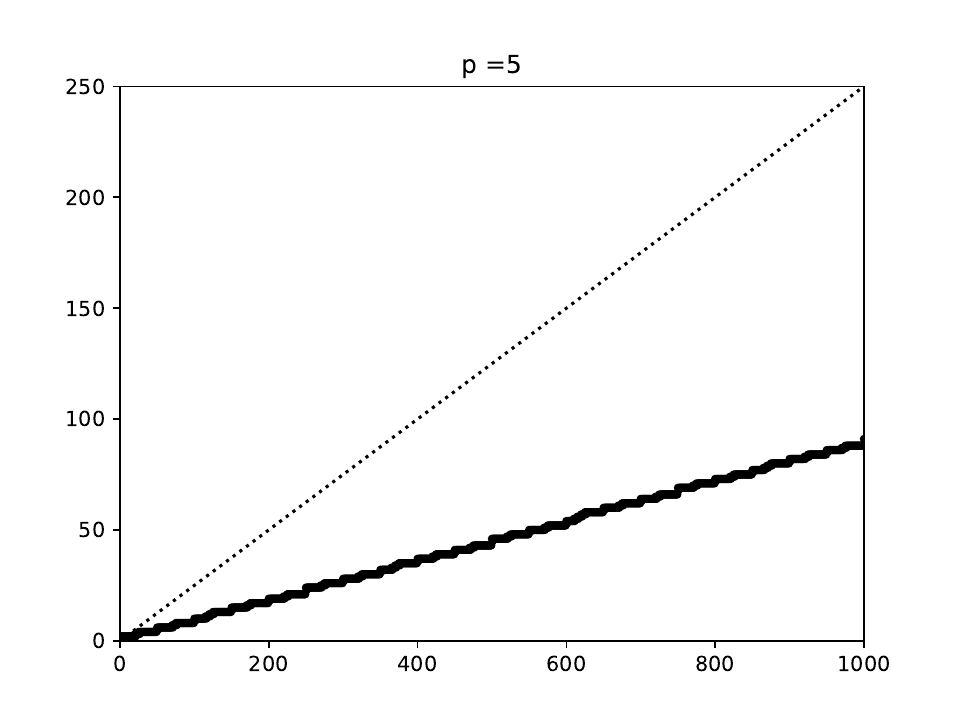}
\end{center}
At both primes, the Newton polygon of our approximation to $\bar\phi(q,\pi)$ has vertices at  $(0,-3)$, $(1,-3)$, $(2,-2)$, $(3,0)$; hence slopes $(0,1,2)$ (matching the Betti numbers for the part of the cohomology we are considering).
\end{example}

\begin{example} \label{th:f1}
Take $M$ to be the Hirzebruch surface $F_1$ (see Section \ref{sec:toric}). In the basis 
\begin{equation}
\{1,[\text{\it exceptional divisor}],[fibre],[point]\}, 
\end{equation}
where $c_1(TM) = (0,2,3,0)$, the quantum connection is \cite{crauder-miranda94}
\begin{equation}
\partial_q + 
\begin{pmatrix}
 0 & 2q &0 & 3q^2 \\
 2q^{-1} & -1 & 1 & 0 \\
 3q^{-1} & 0 & 0 & 2q \\
 0 & q^{-1} & 2q^{-1} & 0
\end{pmatrix}
\end{equation}
One has $c_1(TM)^2 = 8[\mathit{point}]$, hence
\begin{equation}
\Phi_0 = \begin{pmatrix}
1 & 0 & 0 & 0 \\
2 \Gamma_p'(0) & 1/p & 1/p & 0 \\
3 \Gamma_p'(0) & 0 & 0 & 0 \\
4 \Gamma_p'(0)^2 & 1/p & 2/p & 1/p^2 
\end{pmatrix} 
\end{equation}
The experimentally observed growth rates are $\sigma \approx 0.28$ for $p = 3$, $\sigma \approx 0.09$ for $p=5$:
\begin{center}
\includegraphics[scale=0.45]{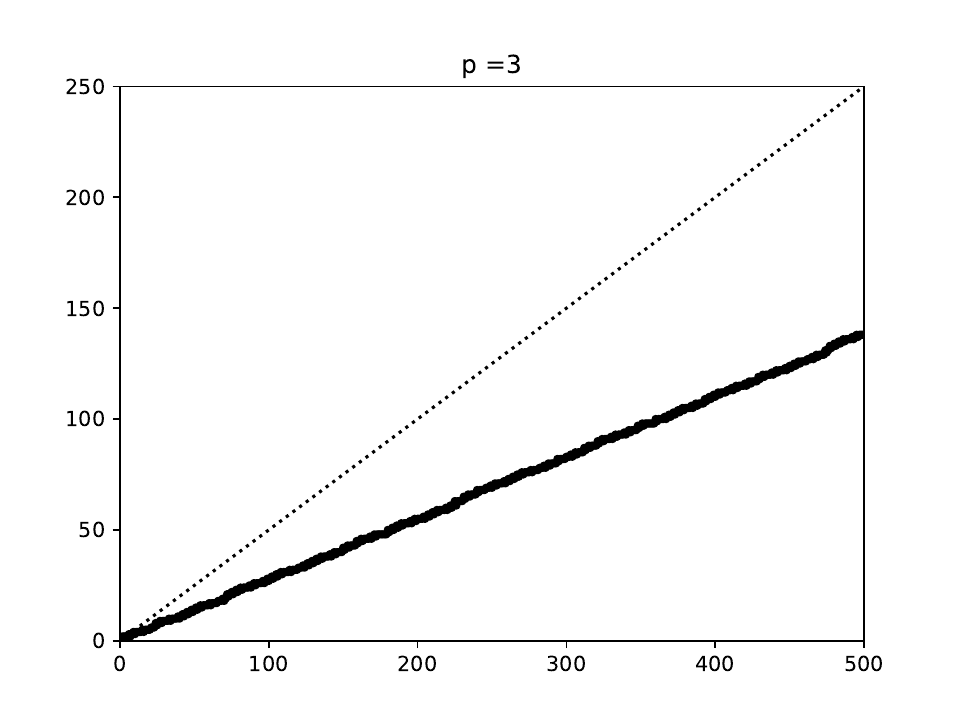}
\includegraphics[scale=0.45]{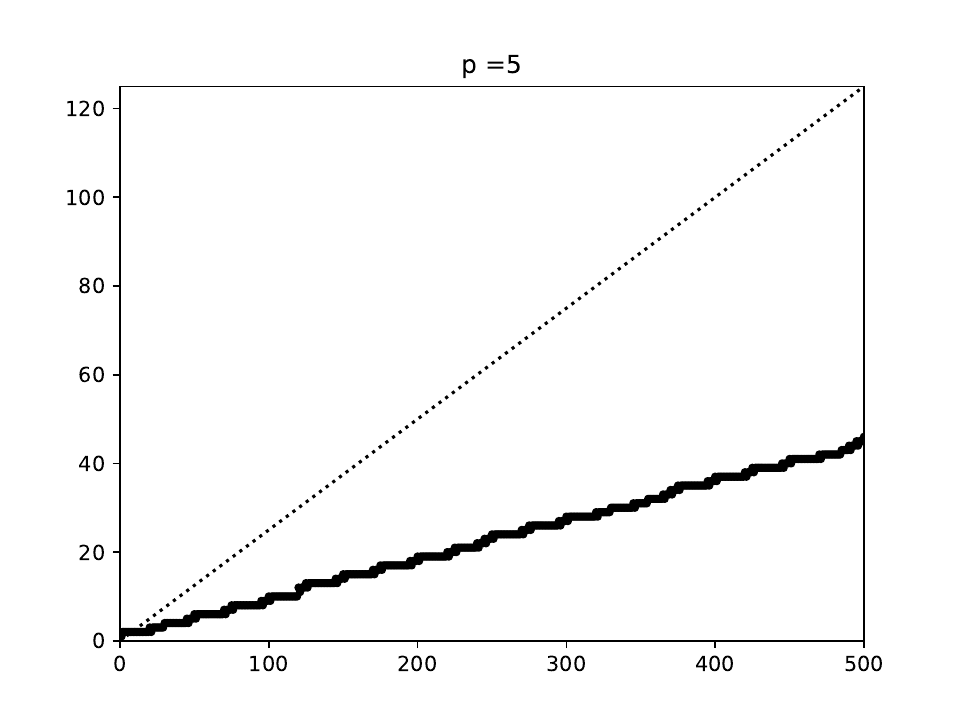}
\end{center}
At both primes, the Newton polygon of our approximation to $\bar\phi(q,\pi)$ has vertices at  $(0,-4)$, $(1,-4)$, $(3,-2)$, $(4,0)$; hence slopes $(0,1,1,2)$.
\end{example}

\begin{example}
Let $M$ be the intersection of two quadrics in $\bC P^5$ (our first example of an irrational variety). This has integral cohomology $\bZ$ in each even degree. Let's choose the integer generators to be positive multiplies of the standard generators in $\bC P^5$: in degree $2$, this is the hyperplane class $x$; in degree $4$, it is $x^2/4$; and in degree $6$, it is $x^3/4 = [\mathit{point}]$. The quantum connection is \cite{donaldson93}
\begin{equation}
\partial_q + 2 \begin{pmatrix} 
0  &   4q  & 0     & 4q^3 \\
q^{-1} &   0 &   2q &    \\
  0 & 4q^{-1} & 0     & 4q \\
 0  & 0   &  q^{-1}  & 0
\end{pmatrix}
\end{equation}
The total Chern class is $c(TM) = (1+x)^6(1+2x)^{-2}$, which means $c_1(TM) = 2x$, $c_2(TM) = 3x^2$, $c_3(TM) = 0$. Hence,
\begin{equation}
\begin{aligned}
\Gamma_p(TM) & =  1 + 2\Gamma'_p(0) x + 2\Gamma_p'(0)^2 x^2 + (3\Gamma_p'(0)^3 - (5/3)\Gamma_p'''(0)) x^3 
\\ &  = \big(
1, 2\Gamma_p'(0), 8\Gamma_p'(0)^2, 12\Gamma_p'(0)^3 - (20/3) \Gamma_p'''(0) 
\big) 
\end{aligned}
\end{equation}
and the constant term of our Frobenius structure is
\begin{equation}
\Phi_0 = \begin{pmatrix} 1 &0 &0 & 0 \\ 
2\Gamma'_p(0) & 1/p &0 & 0
\\ 8\Gamma'_p(0)^2 & 8\Gamma'_p(0)/p & 1/p^2 & 0\\ 
12\Gamma_p'(0)^3 - (20/3) \Gamma_p'''(0)
& 32\Gamma_p'(0)^2/p & 2\Gamma_p'(0)/p^2 & 1/p^3 
\end{pmatrix}
\end{equation}
The experimentally observed growth rates are $\sigma \approx 0.28$ for $p = 3$, $\sigma \approx 0.09$ for $p=5$:
\begin{center}
\includegraphics[scale=0.45]{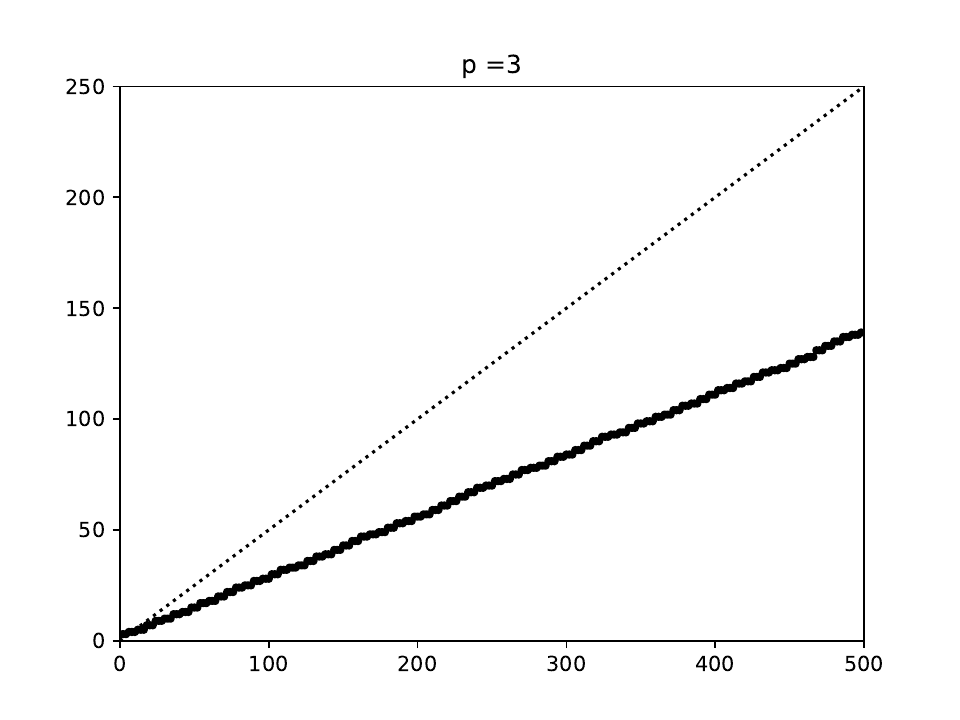}
\includegraphics[scale=0.45]{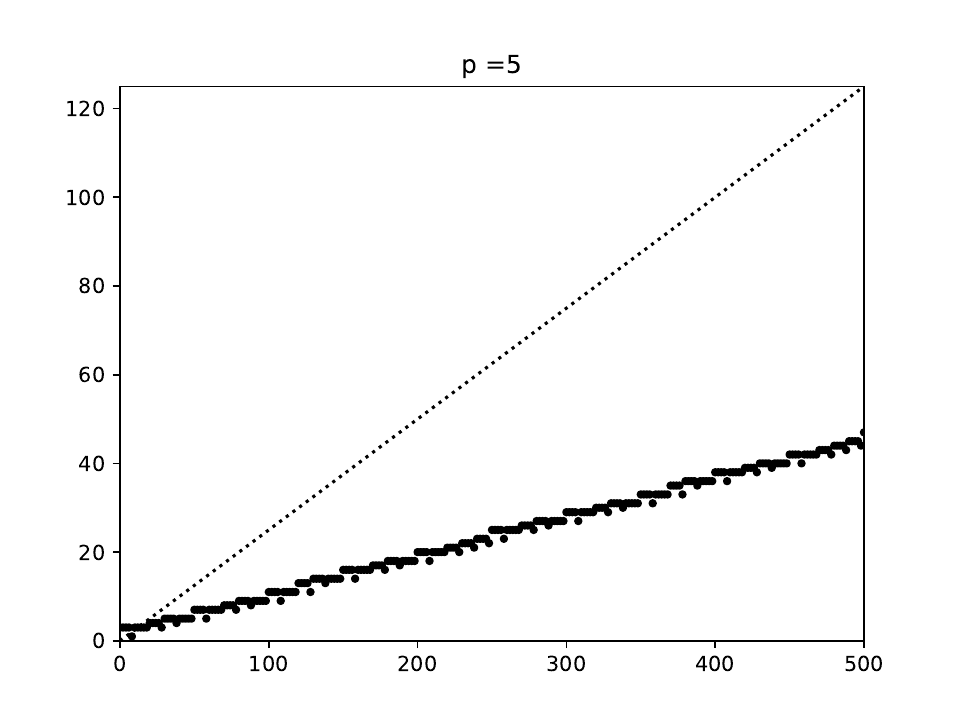}
\end{center}
At both primes, the Newton polygon of our approximation to $\bar\phi(q,\pi)$ has vertices at  $(0,-6)$, $(1,-6)$, $(2,-5)$, $(3,-3)$, $(4,0)$; hence slopes $(0,1,2)$ (matching the Betti numbers for the part of the cohomology we are considering).
\end{example}

%

\begin{example} \label{th:twistor-space}
The twistor spaces studied in \cite{reznikov93,fine-panov10, evans14, hugtenburg24b} yield examples of non-K{\"a}hler monotone symplectic manifolds. Such a manifold (in the lowest possible dimension) comes with a fibration $\bC P^3 \rightarrow M \rightarrow B$, where $B$ is a closed oriented hyperbolic $6$-manifold, assumed to be {\em Spin} for simplicity. Since $c_1(TM)/2$ restricts to the hyperplane class on $\bC P^3$, one has 
\begin{equation}
H^*(M;\bQ) \iso H^*(B;\bQ) \oplus H^*(B;\bQ)c_1(TM) \oplus H^*(B;\bQ)c_1(TM)^2 \oplus H^*(B;\bQ)c_1(TM)^3
\end{equation}
as an $H^*(B;\bQ)$-module. Let $e(TB) \in H^6(B;\bQ) \subset H^6(M;\bQ)$ be the Euler class. Then (\cite[Theorem 6.1]{hugtenburg24b}, based on the computation from \cite{evans14})
\begin{equation} \label{eq:8}
c_1(TM)^4 = -8 e(TB) c_1(TM).
\end{equation}
We need the following higher characteristic classes \cite[Lemma 7.3]{hugtenburg24b}:
\begin{equation} \label{eq:twistor-chern-character}
\begin{aligned}
& \mathit{ch}_3(TM) = c_1(TM)^3/6 + e(TB), \\
& \mathit{ch}_5(TM) = \frac{7}{120} e(TB) c_1(TM)^2. 
\end{aligned}
\end{equation}
The quantum connection splits into several blocks. The simpler blocks are those spanned by $\{y, y\,c_1(TM), y\,c_1(TM)^2, y\,c_1(TM)^3\}$ for some fixed nonzero $y \in H^d(B;\bQ)$, $0 < d < 6$ (of which only the even degree cases $d = 2,4$ stricty fall under our purviev). The cup product with $e(TB)$ acts trivially on this block. In the given basis, we have \cite[Theorem 6.1]{hugtenburg24b}
\begin{equation}
\nabla_{\partial_q} = \partial_q + \begin{pmatrix}
0 & 4q & 0 & 0 \\
q^{-1} & 0 & 0 & 0 \\
0 & q^{-1} & 0 & 4q \\
0 & 0 & q^{-1} & 0
\end{pmatrix}
\end{equation} 
and 
\begin{equation}
\Phi_0 = p^{-d/2} \begin{pmatrix} 1 & 0 & 0 & 0 \\ \Gamma'_p(0) & 1/p & 0 & 0 \\ \Gamma_p'(0)^2/2 & \Gamma_p'(0)/p & 1/p^2 & 0 \\ 
\Gamma_p'''(0)/6 & \Gamma_p'(0)^2/(2p) & \Gamma_p'(0)/p^2 & 1/p^3 \end{pmatrix}
\end{equation}
Here's the outcome of numerical computations: $\sigma \approx  0.28$ for $p = 3$, and $\sigma \approx 0.09$ for $p = 5$;
\begin{center}
\includegraphics[scale=0.45]{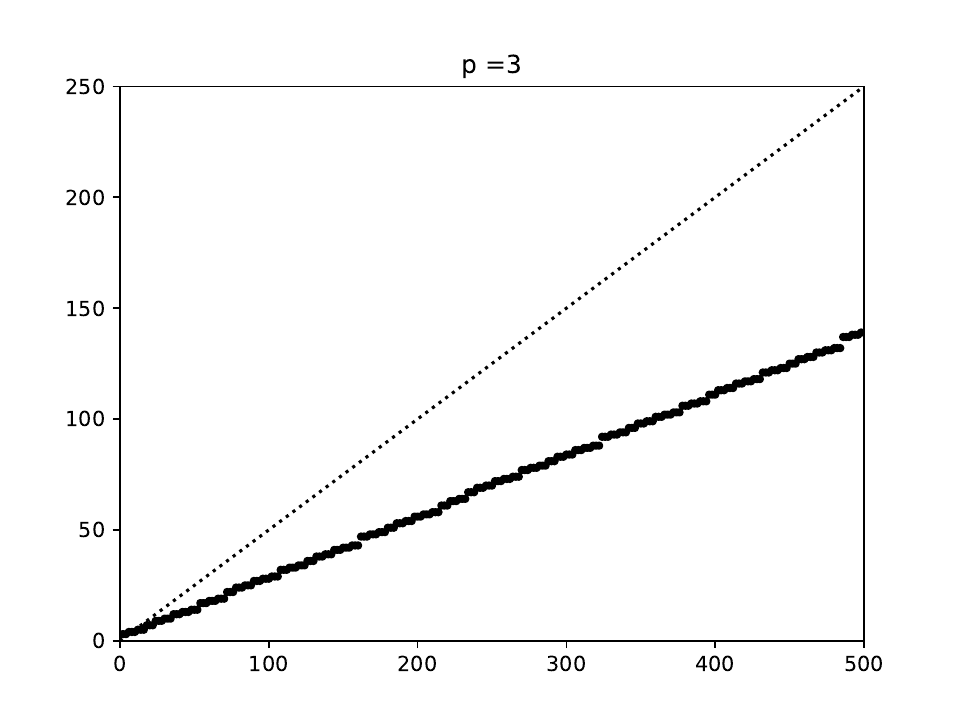}
\includegraphics[scale=0.45]{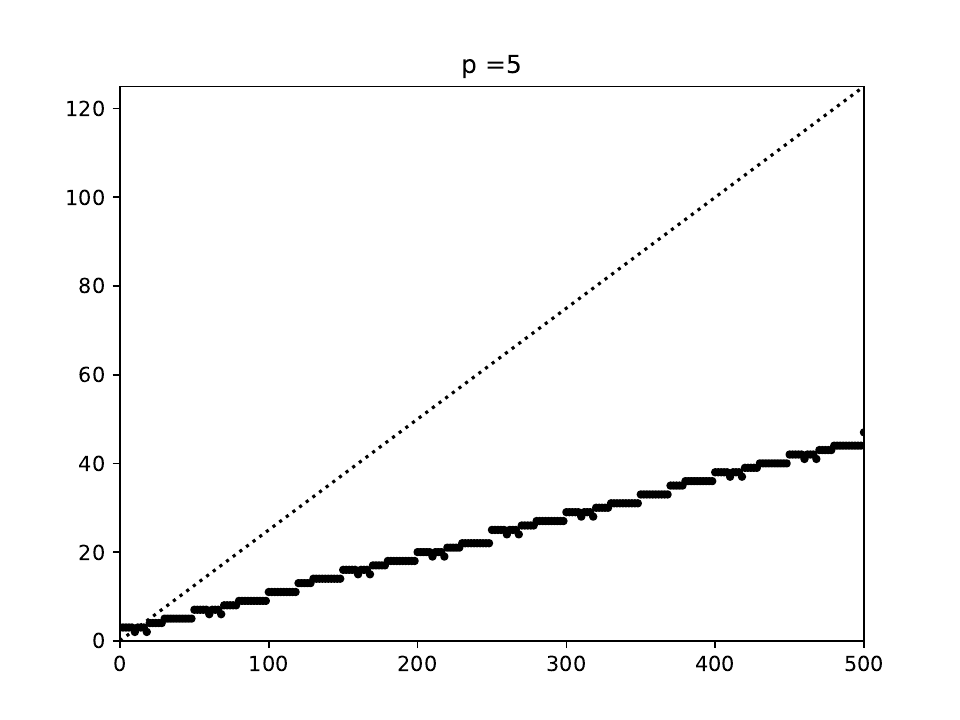}
\end{center}
(Strictly speaking, the graphics here is for the geometrically irrelevant case $d = 0$; but changing $d$ only shifts the curves vertically, since it rescales $\bar\Phi(q)$ by a constant power of $p$). At both primes, the Newton polygon of our approximation to $\bar\phi(q,\pi)$ has vertices at $(0,-6-2d)$, $(1,-6-3/2d)$, $(2,-5-d)$, $(3,-3-d/2)$, $(4,0)$; hence slopes $(d/2,d/2 + 1,d/2 + 2,d/2 + 3)$ (matching the relevant part of the cohomology).

The complicated block is spanned by 
\begin{equation}
\begin{aligned}
& \{1,c_1(TM),c_1(TM)^2,c_1(TM)^3,
\\ & \qquad
e(TB),e(TB)c_1(TM), e(TB)c_1(TM)^2, e(TB)c_1(TM)^3\}. 
\end{aligned}
\end{equation}
In that basis \cite[Theorem 6.1]{hugtenburg24b}
\begin{equation} \label{eq:big-twistor-space}
\nabla_{\partial_q} = \partial_q + \begin{pmatrix}
0 & 4q & 0 & 0 &&&& \\
q^{-1} & 0 & 0 & 0 &&&& \\
0 & q^{-1} & 0 & 4q &&&& \\
0 & 0 & q^{-1} & 0 &&&& \\
&&&& 0 & 4q & 0 & 0 \\
&&& 8q^{-1} & q^{-1} & 0 & 0 & 0 \\
&&&& 0 & q^{-1} & 0 & 4q \\
&&&& 0 & 0 & q^{-1} & 0
\end{pmatrix}
\end{equation}
with the non-block-diagonal term coming from \eqref{eq:8}. Note that this is not irreducible: it's an extension of two copies of the same rank $4$ connection.

The constant term $\Phi_0$ can be read off from the exponential of \eqref{eq:log-gamma-class} and  \eqref{eq:twistor-chern-character}. Numerical data suggest $\sigma \approx 0.28$ for $p = 3$, $\sigma \approx 0.1$ for $p = 5$ (bearing in mind that the observed convergence is slower than in other cases):
\begin{center}
\includegraphics[scale=0.45]{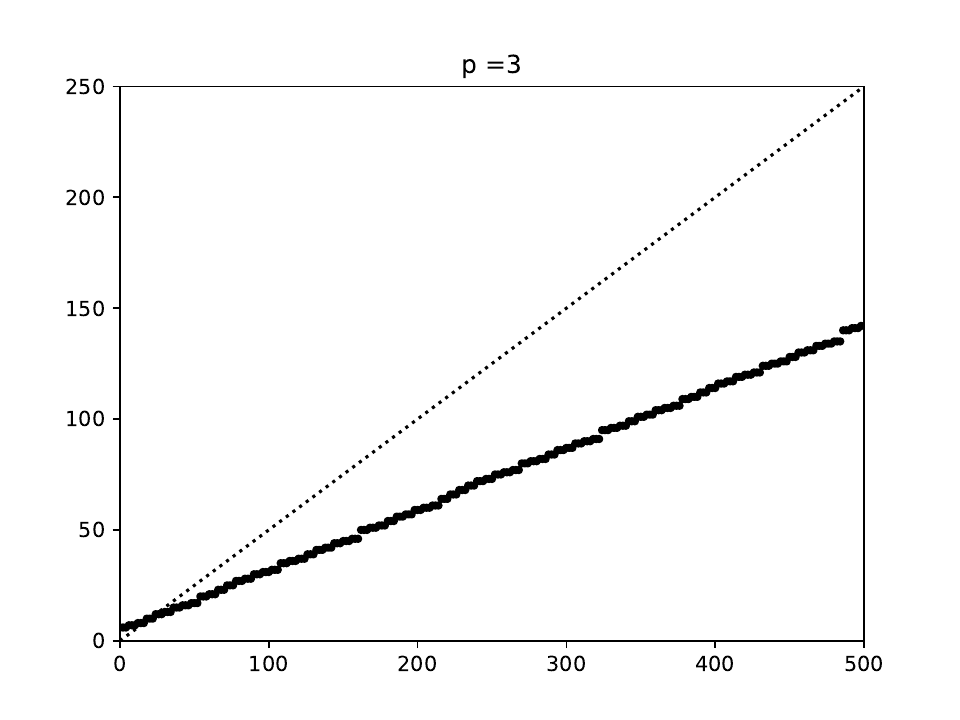}
\includegraphics[scale=0.45]{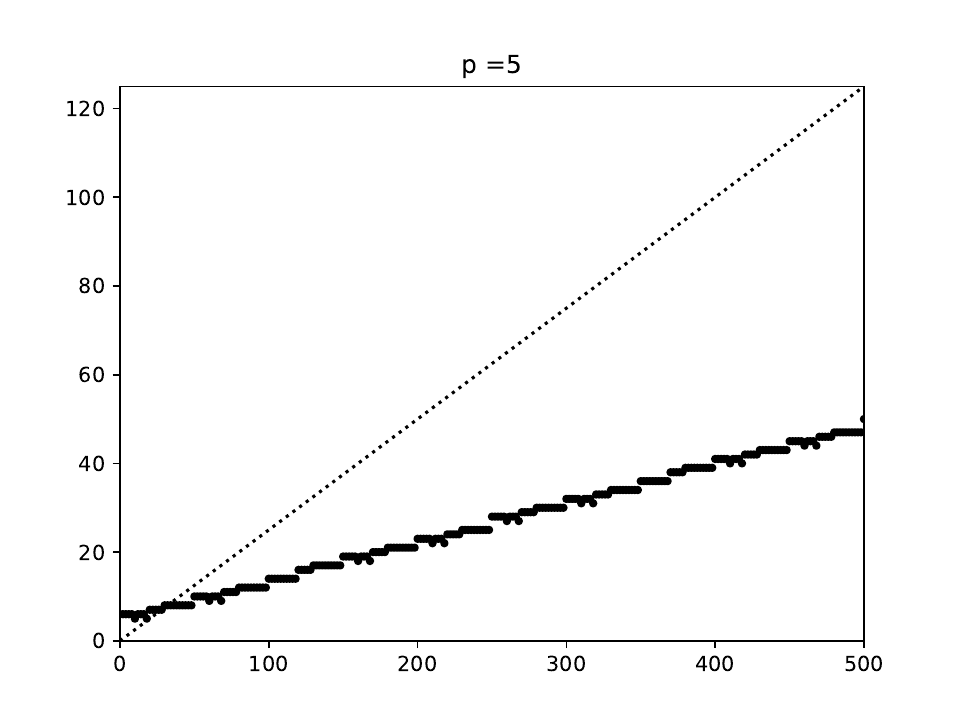}
\end{center}
Unfortunately, computations involving the Newton polygon for \eqref{eq:big-twistor-space} exceeded our modest resources.
\end{example}

\end{document}